\documentclass[final]{siamltex}

\usepackage{graphicx}
\usepackage{amsmath}
\usepackage{algorithm}
\usepackage{algorithmicx}
\usepackage[noend]{algpseudocode}
\usepackage{color}
\usepackage{bm}
\usepackage{amsfonts}
\usepackage{float}
\usepackage{subfig}
\usepackage{epstopdf}
\usepackage{multirow}
\usepackage{tablefootnote}
\usepackage{threeparttable}
\usepackage{makecell}

\newcommand{\R}{\mathbb{R}}
\newcommand{\tr}{^{\sf T}}
\newcommand{\m}[1]{{\bf{#1}}}

\newtheorem{remark}{Remark}[section]
\newtheorem{assum}{Assumption}
\newtheorem{defin}{Definition}

\title{An Adaptive Decentralized Quasi-Newton Method with Stepsizes Independent of the Local-Update Budget
\thanks{This work was supported by the National Natural Science Foundation of China (Grant Number 11971231).
}}
\author{
	Hao Wu and
	Liping Wang\thanks{School of Mathematics, Nanjing University of Aeronautics and Astronautics, Nanjing, China. Email: {\tt wuhoo104@nuaa.edu.cn; wlpmath@nuaa.edu.cn}}
}

\begin{document}
\maketitle
\begin{abstract}
This paper proposes a novel Adaptive Decentralized Quasi-Newton (AdaDQN) method for solving smooth nonconvex optimization problems over undirected networks. While standard decentralized algorithms with multiple fixed local updates typically admit a convergence stepsize inversely proportional to the number of local updates, we show that this scaling is worst-case tight for the typical unscaled fixed-local-update scheme. We revisit a class of gradient-tracking methods with fixed local updates from a surrogate-based perspective and establish a Robust Inexact Algorithm (RIA) framework. Inspired by this framework, AdaDQN integrates a safeguarded consensus-aware termination criterion, a standard event-triggered communication protocol, and a scalable memoryless BFGS update. We establish an $\mathcal{O}(1/T)$ best-iterate rate for first-order stationarity. For a squared stationarity tolerance $\delta$, the guaranteed gradient complexity of the fixed scheme is $\mathcal{O}(nK_g/\delta)$, whereas AdaDQN attains $\mathcal{O}(n/\delta+n\alpha\tilde\varepsilon^{-2})$, independent of the maximum local-update budget. Numerical experiments demonstrate that AdaDQN achieves a superior computation-communication tradeoff, outperforming state-of-the-art decentralized methods across various performance metrics.

\end{abstract}


\pagestyle{myheadings}
\thispagestyle{plain}
\section{Introduction}
Decentralized optimization involves a network of interconnected nodes aiming to collaboratively solve a global optimization problem. The network is modeled as a graph $\mathcal{G}=\left(\mathcal{V},\mathcal{E}\right)$, where $\mathcal{V} = \{1,\ldots,n\}$ is the set of nodes, and $\mathcal{E} \subseteq  \mathcal{V} \times \mathcal{V}$ is the set of communication links between nodes. Each node $i \in \mathcal{V}$ possesses its own local cost function $f_i$ and the collective goal is to find a consensual solution that tackles the consensus optimization problem:
\begin{equation}\label{obj_fun1}
	\mathop {\min } \limits_{\m{z} \in {\R^p}} \; F(\m{z}):= \frac{1}{n}\sum\limits_{i = 1}^n {{f_i}(\m{z})},
\end{equation}	
where $f_i:\R^{p} \rightarrow \R$ is a continuously differentiable and possibly nonconvex objective function of the $i$-th node. 
This problem formulation motivates a decentralized and parallelizable approach to finding consensus solutions as each node only has access to its own local function and coordination is achieved solely through communication with its immediate neighbors in the network. Using a fully decentralized system breaks through the centralized system's bottleneck due to single point of failure, bandwidth limit, latency requirement and large communication overhead. Thus, decentralized optimization has wide applications, including but not limited to:  decentralized resources control \cite{fusco2021decentralized}, wireless networks \cite{jeong2022asynchronous}, decentralized machine learning \cite{zhang2022distributed}, power systems \cite{chen2020fully}, federated learning \cite{pillutla2022robust}.

An efficient decentralized method is highly
needed such that the practical tasks can be performed efficiently by using computational nodes across the network.
Well-known methods include Decentralized Gradient Descent (DGD)  \cite{nedic2009distributed,yuan2016convergence,zeng2018nonconvex}, Gradient Tracking (GT) \cite{xu2015augmented,qu2017harnessing,nedic2017achieving,nedic2017geometrically,xin2019distributed,gao2022achieving,zhang2020distributed,song2024optimal}, Exact Diffusion (ED) $\setminus$ EXTRA \cite{shi2015extra,yuan2018exact,alghunaim2024local,wu2025music}, and primal-dual methods \cite{shi2014linear,jakovetic2020primal,mancino2023decentralized}.

Classic methods like DGD, GT, and ED typically decompose iterations into a local update step and a communication step. However, the cost of these two steps varies significantly across applications and hardware environments. Let us examine the following scenarios \cite{berahas2018balancing}.\\
\noindent \textbf{Scenario 1 (Communication-Limited):} Control a swarm of battery-powered robots equipped with low-energy computation modules, interconnected via an energy-intensive communication protocol. Given the robots' constrained energy capacity, communication processes can be particularly costly, whereas extended task completion durations may not pose significant issues. \\
\noindent \textbf{Scenario 2 (Computation-Limited):} In a large-scale machine learning problem utilizing a computer cluster with either physical interconnections or shared memory architecture, communication overhead becomes negligible, whereas computational expense may become substantial, primarily determined by the scale of the dataset on each individual machine.

Therefore, for adapting to different application environments, customized algorithms are proposed and endowed with flexibility in terms of communication, computation, or a hybrid of both. 

In communication-limited scenarios, to mitigate the high communication burden, a direct and efficient idea is to reduce the number of communications. Most commonly used approaches include {\it local update} and {\it event triggering}. The former means that multiple local updates are allowed between two consecutive communication rounds. Such communication-centric methods include GT-based methods  \cite{ge2023gradient,nguyen2023performance,liu2024decentralized} and ED-based methods \cite{alghunaim2024local,wu2025music,liu2025guaranteeing}. 
{\it Event triggering}, another above-mentioned approach to reduce the number of communications, is to exchange information only when certain conditions related to the algorithm iterates are met during algorithm execution. The triggering event can be devised so that the information is exchanged only when necessary. This can potentially reduce the communication cost without degrading the algorithm performance much.  
The event-triggered strategy has been extensively adopted to enhance communication efficiency in decentralized optimization, including its application to decentralized stochastic gradient descent \cite{george2020distributed}, gradient tracking methods under both convex \cite{gao2021event} and nonconvex \cite{liu2023event,gao2024distributed} settings, decentralized Adam \cite{okamoto2023distributed}, the continuous-time decentralized primal-dual method \cite{xu2023distributed}, as well as decentralized ADMM variants \cite{liu2019communication,li2019communication,liu2021dqc,zhang2023decentralized}. 
On the other hand, computation-centric methods \cite{berahas2018balancing,mansoori2019general} perform multiple communication steps between two consecutive computation steps (e.g. gradient evaluation). In a primal-dual framework \cite{mansoori2019general}, each node performs a predetermined number of primal updates where this node repeatedly communicates with its neighbors without recomputing its gradient.

More practical consideration lies in balancing communication and computation within an algorithmic framework. Related works include \cite{berahas2021convergence,liu2022decentralized,huang2023computation,berahas2024balancing,huang2025pareto}, where multiple communication and local update steps are performed at each iteration. Specifically, such strategy was incorporated in DGD for strongly convex setting \cite{berahas2021convergence} and stochastic nonconvex setting \cite{liu2022decentralized} while applied to GT for strongly convex setting \cite{huang2023computation,berahas2024balancing} and stochastic nonconvex setting \cite{huang2025pareto}. We summarize and compare the most relevant algorithms in Table \ref{Comparisons}. 
The communication and computation patterns of decentralized algorithms in the aforementioned works \cite{berahas2021convergence,liu2022decentralized,huang2023computation,berahas2024balancing,huang2025pareto}
follow some predefined rules. Specifically, these algorithms require the predetermined number of communication and local update steps per iteration, uniformly applied across all nodes. We argue that this approach imposes significant limitations. \\
\noindent \textbf{Theoretical limitations:} The worst-case analysis of these local update methods typically forces the convergence stepsize to be inversely proportional to the number of local steps, denoted as $K_g$. This theoretical restriction implies that any computational gain from increasing $K_g$ is canceled out by a shrinking stepsize. However, practical empirical studies contradict this limitation, showing that multiplying $K_g$ by a certain factor does not actually require a corresponding proportional reduction in the stepsize to ensure convergence. This gap indicates that fixing the local update steps is overly conservative.\\
\noindent \textbf{Practical limitations:} Predetermined computation and communication schedules may be inefficient because the local optimization progress and the variation of communicated variables evolve across nodes and iterations. This motivates an adaptive framework that determines local update and broadcast using locally available algorithmic information. 

	\begin{table}
	\centering
	\caption{Comparisons with existing algorithms}\label{Comparisons}
	\begin{threeparttable}
		\begin{tabular}{cccccc}
			\hline
			\textbf{Alg.}&\textbf{Objective} &\textbf{Comm.}&\textbf{Comp.} &\textbf{Order} & \textbf{Oracle}\\
			\hline
			GT\cite{lu2019gnsd} &NC \tnote{1} &Single \tnote{2}&Single &First &SG \tnote{3}\\
			LU-GT \cite{nguyen2023performance} &NC &Single&Multi. \tnote{2} &First &FG \tnote{3}\\
			LSGT\cite{ge2023gradient} &NC &Single&Multi. &First &SG\\
			$\mbox{DGD}^{t}$ \cite{berahas2018balancing} &SC \tnote{1} &Multi.&Single &First &FG\\
			DETSGRAD \cite{george2020distributed}&NC &Adapt. \tnote{2}&Single &First & SG\\
			ET-DOGT \cite{liu2023event}&NC &Adapt.&Single &First & FG\\
			DETNO \cite{gao2024distributed}&NC+PL\tnote{1} &Adapt.&Single &First & FG\\
			NEAR-DGD \cite{berahas2021convergence}&SC &Multi.&Multi. &First & FG\\
			DFL \cite{liu2022decentralized}&NC &Multi.&Multi. &First & SG\\
			GTA \cite{berahas2024balancing}&SC &Multi.&Multi. &First & FG\\
			FlexGT \cite{huang2025pareto}&NC &Multi.&Multi. &First & SG\\
			\hline
			\multirow{2}{*}{AdaDQN (Ours)} &\multirow{2}{*}{NC} &\multirow{2}{*}{\shortstack{Multi.\\+Adapt.}}&\multirow{2}{*}{\shortstack{Multi.\\+Adapt. }}&\multirow{2}{*}{\shortstack{Quasi \\ Second \tnote{4}}} & \multirow{2}{*}{FG}\\
			& & & & &\\
			\hline 
		\end{tabular}
		\begin{tablenotes}
			\item [1] ``NC'',``SC'', and ``PL'' respectively mean ``Nonconvex'', ``Strongly convex'', and ``Polyak-{\L}ojasiewicz''.
			\item [2] ``Single'', ``Multi.'', and ``Adapt.'' have different meanings across columns. In the Comm. column, they indicate a single communication round, multiple communication rounds, and adaptive broadcasts, respectively; in the Comp. column, they indicate a single, multiple, and adaptive number of local computation steps, respectively.
			\item[3] ``SG'' and ``FG'' mean that the algorithm has access to stochastic gradients or full gradients.
			\item[4] ``Quasi second" means the second-order information is captured by Hessian approximations using gradient information. 
		\end{tablenotes}
	\end{threeparttable}
\end{table}

Our main contributions can be summarized as follows. 
\begin{itemize}
	\item [1.] We revisit the fixed-local-update gradient-tracking schemes examined in this paper through locally constructed surrogate models and establish a Robust Inexact Algorithm (RIA) framework, under which their local update recursions satisfy a common surrogate-based representation. Our analysis demystifies why these schemes work by demonstrating that their convergence is driven fundamentally by two regularity conditions, specifically sufficient descent and bounded deviation. It also shows that, for the unscaled fixed-local-update gradient-tracking scheme, the $\mathcal{O}(1/K_g)$ sufficient stepsize scaling is worst-case tight up to constant factors.
	\item [2.] We propose an Adaptive Decentralized Quasi-Newton (AdaDQN) method, which achieves double-level adaptivity through adaptive local updates and adaptive broadcasts based on locally available algorithmic states. Specifically, we introduce a safeguarded consensus-aware termination criterion for the local update loop, which applies the bounded deviation condition as a safeguard and automatically terminates the local updates when the consensus error dominates the local optimization gain. We further employ a standard event-triggered communication protocol that broadcasts variables only when the local state evolves significantly.
	\item[3.]  We establish that the first-order stationarity residual of AdaDQN converges to zero with an $\mathcal{O}(1/T)$ best-iterate rate under appropriate stepsize choices. More importantly, we prove that the convergence stepsize is decoupled from the reciprocal of the maximum number of local updates. For a squared stationarity tolerance $\delta$, the fixed-local-update scheme requires $\mathcal{O}(nK_g/\delta)$ gradient evaluations, while AdaDQN admits a $K_g$-independent bound $\mathcal{O}(n/\delta+n\alpha\tilde\varepsilon^{-2})$. Numerical results on various datasets confirm the efficacy of AdaDQN in saving both communication and computation, and demonstrate that our method achieves superior performance in terms of iteration count, communication cost, and gradient evaluations compared to state-of-the-art relative decentralized algorithms.
\end{itemize}
The rest of this paper is organized as follows. Section 2 formally introduces the problem setting, outlining the necessary assumptions on the objective functions and the network topology. Section 3 presents a surrogate-based perspective for gradient-tracking methods with multiple local updates, establishes the RIA framework, and provides its convergence analysis. In Section 4, we detail the development of the proposed AdaDQN method and establish its convergence guarantees. Numerical experiments of comparing our AdaDQN method with other well-established methods for solving decentralized nonconvex optimization are presented in Section 5.
We finally draw some conclusions in Section 6.

\subsection{Notation}
We use uppercase boldface letters, e.g.~$\m{W}$, for matrices and lowercase boldface letters,  e.g.~$\m{w}$, for vectors.
For any vectors $\m{v}_i \in \R^p$, $i=1,\ldots,n$, we define  $\bar{\m{v}}=\frac{1}{n} \sum_{i=1}^n \m{v}_i$ and $\m{v} = [\m{v}_1; \m{v}_2; \ldots, \m{v}_n] \in \R^{np}$.
Given an undirected network  $\mathcal{G}=\left(\mathcal{V},\mathcal{E}\right)$,
let $\m{x}_i$ denote the local copy of the global variable $\m{z}$ at node $i$ and $\mathcal{N}_i$ denote the set consisting of the neighbors of node $i$ 
(for convenience, we treat node $i$ itself as one of its neighbors). 
We define $f(\m{x}) = \sum_{i = 1}^n {{f_i}({\m{x}_i})}$ and use $\m{g}^t$, $\m{g}_i^t$ to stand for $\nabla f(\m{x}^t)$, $\nabla f_i(\m{x}^t_{i})$ respectively,
where, for clarification, the gradient of $f(\m{x})$ is defined as
$ \nabla f(\m{x}) =\left[ \nabla f_1(\m{x}_{1}); \nabla f_2(\m{x}_{2}); \ldots, \nabla f_n(\m{x}_{n}) \right] \in \R^{np}$. 
In addition,  we define $\overline{\nabla} f(\m{x}^t)=\frac{1}{n} \sum_{i=1}^n \nabla f_i(\m{x}_i^t) 
\in \R^p$.
We say that $\m{x}$ is consensual or gets consensus if ${\m{x}_1}={\m{x}_2}=\ldots={\m{x}_n}$. 
$\m{I}_p$ denotes the $p \times p$ identity matrix and $\m{I}$ denote $\m{I}_{np}$ for simplicity. Kronecker Product is denoted as $\otimes$.
Given a symmetric matrix $\m{N}$, $\m{N}\tr$  denotes transpose of $\m{N}$;
${\lambda _{\min }(\m{N})}$,  ${\lambda _{\max }(\m{N})}$, and $\rho(\m{N})$ denote smallest eigenvalue, largest eigenvalue, and  the spectral radius of $\m{N}$,  respectively.
For matrices $\m{N}_1$ and $\m{N}_2$ with same dimension, $\m{N}_1 \succeq \m{N}_2$ means $\m{N}_1 - \m{N}_2$ is positive semidefinite, while $\m{N}_1 \ge \m{N}_2$ means $\m{N}_1 - \m{N}_2$ is component-wise nonnegative.
We define $\m{M}=\frac{1}{n} \m{1}_n\m{1}_n\tr\otimes \m{I}_p$ where $\m{1}_n \in \R^n$ denotes the
vector with all components ones. 

\section{Problem setting}
The following are several necessary assumptions for the objective function.

\begin{assum}\label{as0}
	Every local objective function $f_i$ is proper (i.e., not everywhere infinite) and the global objective is bounded below (i.e.,	$F(\m{z})> \underline{F} :=\inf_{\m{z}}F(\m{z})>-\infty$ for any $\m{z} \in \mathbb{R}^p$).
\end{assum}

\begin{assum}\label{as1}
	Each local gradient $\nabla f_i$ is Lipschitz continuous with constant $L>0$, i.e.,
	\begin{equation}\label{3.1}
		\left\Vert {\nabla {f_i}(\m{z}) - \nabla {f_i}(\tilde{\m{z}})} \right\Vert \le L \left\Vert {\m{z} - \tilde{\m{z}}} \right\Vert,
	\end{equation}
	$\forall ~\m{z},\tilde{\m{z}} \in {\mathbb{R}^p},i \in \mathcal{V}$.
\end{assum}

\begin{assum}\label{as_network}
	The communication graph $\mathcal{G}=(\mathcal{V},\mathcal{E})$ is fixed, connected, and undirected.
\end{assum}

Under Assumption~\ref{as_network}, communication over $\mathcal{G}$ is parameterized by an admissible mixing matrix $\tilde{\m{W}}=[\tilde{W}_{ij}]\in\R^{n\times n}$ defined as follows.
\begin{defin}\label{mix}
	(Admissible mixing matrix $\tilde{\m{W}}$ for a given network $\mathcal{G}=\left(\mathcal{V},\mathcal{E}\right)$)
	\begin{itemize}
		\item [1.] $\tilde{\m{W}}$ is nonnegative and respects the graph topology, i.e., $\tilde{W}_{ij}>0$ if $j\in\mathcal{N}_i$ and $\tilde{W}_{ij}=0$ otherwise. In particular, $\tilde{W}_{ii}>0$ for every $i\in\mathcal{V}$.
		\item [2.] $\tilde{\m{W}}$ is symmetric and doubly stochastic, i.e., $\tilde{\m{W}}=\tilde{\m{W}}\tr$ and $\tilde{\m{W}}\m{1}_n=\m{1}_n$.\\
	\end{itemize}
\end{defin}
There are a few common choices for the mixing matrix $\tilde{\m{W}}$, such as the Laplacian-based constant edge weight matrix \cite{sayed2014diffusion} and
the Metropolis constant edge weight matrix \cite{xiao2007distributed}.
Let $\lambda_{i}(\tilde{\m{W}})$ denote the $i$-th largest eigenvalue of $\tilde{\m{W}}$ and $\sigma$ be the second largest magnitude eigenvalue of $\tilde{\m{W}}$. 
Then, the following properties hold.
\begin{lemma}\label{property W}
	For $\tilde{\m{W}}$ defined in Definition~\ref{mix} and  $\m{W} :=\tilde{\m{W}}\otimes\m{I}_p$, we have
	\begin{itemize}
		\item [1.] $1=\lambda_{1}(\tilde{\m{W}})>\lambda_{2}(\tilde{\m{W}})\geq\ldots\geq\lambda_{n}(\tilde{\m{W}})>-1$;
		\item [2.] $	0\leq\rho(\m{W}-\m{M})=\sigma=\max \left\{|\lambda_{2}(\tilde{\m{W}})|, |\lambda_{n}(\tilde{\m{W}})|\right\}<1$;
		\item [3.] $\m{M}=\m{M}\m{W}=\m{W}\m{M}$;
		\item [4.]  
		$\Vert\m{W}^{k}\m{x}-\m{M}\m{x}\Vert=\Vert(\m{W}^{k}-\m{M})(\m{x}-\m{M}\m{x})\Vert \leq \sigma^{k} \Vert\m{x}-\m{M}\m{x}\Vert$
		for any $\m{x} \in \mathbb{R}^{np}$ and $k \geq 1$. 
	\end{itemize}
\end{lemma}
These standard properties follow from the symmetry, double stochasticity, and positive diagonal of $\tilde{\m{W}}$; see, e.g., \cite{xin2019distributed}.

We say that a point $\m{x}^* \in \mathbb{R}^{np}$ is a first-order stationary point of \eqref{obj_fun1} if it satisfies 
\begin{equation}\label{stationarity}
	\Vert\overline{\nabla} f(\m{x}^*)\Vert^2+\Vert\m{x}^*-\m{M}\m{x}^*\Vert^2=0.
\end{equation}

\section{A new perspective for gradient-tracking methods with multiple local updates}
\subsection{Surrogate-based interpretation}
A class of gradient-tracking methods with a single local update can be written in the following form:
\begin{align}
	&\m{x}_i^{t+1}=\sum_{j \in \mathcal{N}_i} \tilde{W}_{ij}\left(\m{x}_j^t-\alpha \m{H}_j^t\m{v}_j^t\right),\label{update1}\\
	&\m{v}_i^{t+1}=\sum_{j \in \mathcal{N}_i} \tilde{W}_{ij}\left(\m{v}_j^t+\m{g}_j^{t+1}-\m{g}_j^t\right),\label{update2}
\end{align}
initialized with $\m{v}_i^0=\m{g}_i^0$ for $i=1,\ldots,n$. $\m{v}_i^t$ tracks the global gradient $\nabla F(\bar{\m{x}}^t)$ via dynamic average consensus \cite{zhu2010discrete}. $\m{H}_i^t$ is a suitable symmetric positive definite matrix. By properly choosing $\m{H}_i^t$, we retrieve standard methods, in particular:
\begin{itemize}
	\item if $\m{H}_i^t=\m{I}_p$, we obtain the adapt-then-combine gradient-tracking methods \cite{xu2015augmented,nedic2017geometrically}.
	\item if $\m{H}_i^t$ is a positive definite matrix obtained with suitable update rules, we obtain the quasi-Newton methods \cite{wu2026unified}.
\end{itemize}

We provide a surrogate-based interpretation of the above updates \eqref{update1} and \eqref{update2}. This perspective constructs locally computable models whose gradients incorporate the tracked global-gradient information. The core challenge in decentralized nonconvex optimization lies in the nonconvexity of local objectives $f_i$ and the unavailability of other nodes' functions $f_j$ ($j \neq i$), preventing node $i$ from directly  solving the problem \eqref{obj_fun1}.
To address this, we construct a locally computable surrogate for the {global} objective $F(\m{z})=f_i(\m{z})+F(\m{z})-f_i(\m{z})$. At iteration $t$, each node $i$ first constructs a local quadratic model $\tilde{f}_i(\cdot;\m{x}_i^t)$ for $f_i$ at $\m{x}_i^t$:
\begin{align*}
	\tilde{f}_i(\m{z};\m{x}_i^t) = f_i(\m{x}_i^t) + \langle\nabla f_i(\m{x}_i^t),\m{z}-\m{x}_i^t\rangle + \frac{1}{2\alpha}(\m{z}-\m{x}_i^t)\tr (\m{H}_i^t)^{-1}(\m{z}-\m{x}_i^t),
\end{align*}
and then combines it with a linear approximation of the rest of the network's objective using the tracking variable $\m{v}_i^t$:
\begin{align*}
	\tilde{F}(\m{z},\m{x}_i^t): = \tilde{f}_i(\m{z};\m{x}_i^t) + \langle\m{v}_i^t - \nabla f_i(\m{x}_i^t),\m{z}\rangle.
\end{align*}
The update step $\m{x}_i^{t+1/2} = \m{x}_i^t - \alpha \m{H}_i^t\m{v}_i^t$ corresponds to minimizing this surrogate $\tilde{F}(\m{z};\m{x}_i^t)$. The local surrogate update is then followed by the consensus step 
$$
\m{x}_i^{t+1} = \sum_{j \in \mathcal{N}_i} \tilde{W}_{ij} \m{x}_j^{t+1/2}
$$ 
and the gradient tracking step \eqref{update2}.


Now we extend the updates \eqref{update1} and \eqref{update2} to incorporate multiple local updates ($K_g \geq 1$) and communication rounds ($K_c \geq 1$) per iteration, and retain the surrogate-based interpretation. 
Specifically, the mixing matrix $\m{W}$ in \eqref{update1} and \eqref{update2} is replaced by $\m{W}^{K_c}$, representing $K_c$ consecutive communication steps. Furthermore, we allow nodes to perform multiple local updates before communicating. The detailed procedure is presented in Algorithm \ref{alg:Framwork1}.
\begin{algorithm}[htb]
	\caption{Gradient Tracking with Multiple Computations and Communications}
	\label{alg:Framwork1}
	\begin{algorithmic}[1]
		\Require Initial point $\m{x}^0$, Maximum iteration $T$, Stepsize $\alpha>0$, Mixing matrix $\m{W}$, Maximum computations $K_{g} \geq 1$, Maximum communications $K_c \geq 1$.
		\State Initialize $t=0$, $\m{v}^{0}=\m{g}^{0}$, $\m{H}^0=\m{I}$.
		\While{$t < T$}
		\State $\m{x}^{t,0}=\m{x}^t$, $\m{v}^{t,0}=\m{v}^t$.
		\For{$k = 0$ to $K_g-2$}
		\State $\m{x}^{t,k+1}= \m{x}^{t,k}-\alpha \m{H}^{t,k} \m{v}^{t,k}$.
		\State $\m{v}^{t,k+1}=\m{v}^{t,k}+\m{g}^{t,k+1}-\m{g}^{t,k}$.
		\State Update $\m{H}^{t,k+1}$.
		\EndFor
		\State $\m{x}^{t,K_g} = \m{x}^{t,K_g-1} - \alpha \m{H}^{t,K_g-1} \m{v}^{t,K_g-1}$.
		\State $\m{x}^{t+1} = \m{W}^{K_c} \m{x}^{t,K_g}$.
		\State $\m{v}^{t,K_g} = \m{v}^{t,K_g-1} + \m{g}^{t+1} - \m{g}^{t,K_g-1}$.
		\State $\m{v}^{t+1} = \m{W}^{K_c} \m{v}^{t,K_g}$.
		\State Update $\m{H}^{t+1}$.
		\State $t \leftarrow t+1$.
		\EndWhile
		\Ensure $\m{x}^T$.
	\end{algorithmic}
\end{algorithm}

\begin{remark}\label{remark_2.2}
	Algorithm \ref{alg:Framwork1} specializes to the following schemes considered in this paper: setting $K_g=1, K_c=1$ recovers the DNQN method, fixing $\m{H}^t=\m{I}$ recovers the GTA method \cite{berahas2024balancing}, and setting $K_c=1$ and $\m{H}^t=\m{I}$ yields the local update recursion used by LU-GT \cite{nguyen2023performance}, LSGT \cite{ge2023gradient}, and FlexGT \cite{huang2023computation}. Notably, LSGT and FlexGT update $\m{v}^{t,K_g} = \m{v}^{t,K_g-1} + \m{g}^{t,K_g} - \m{g}^{t,K_g-1}$ and then conduct the communication step $\m{x}^{t+1} = \m{W}^{K_c} \m{x}^{t,K_g}$, which is slightly different from Steps 9 and 10 in Algorithm \ref{alg:Framwork1}.
\end{remark}

In fact, Algorithm \ref{alg:Framwork1} admits a direct surrogate-based interpretation. At iteration $t$, rather than constructing a strongly convex majorizing surrogate, each node $i$ directly considers the original local objective corrected by the global gradient estimate:
\begin{align}\label{eq:mm_target}
	\min_{\m{z}} \phi_i^t(\m{z}) := f_i(\m{z}) + \langle\m{v}_i^t - \nabla f_i(\m{x}_i^t), \m{z}\rangle.
\end{align}
Observe that $\nabla \phi_i^t(\m{x}_i^t) = \m{v}_i^t$, ensuring the descent direction aligns with the global gradient estimate. Rather than solving \eqref{eq:mm_target} exactly, which may be intractable for nonconvex $f_i$, Algorithm \ref{alg:Framwork1} performs $K_g$ inner 
iteration steps. This constitutes an inexact solution procedure for the local surrogate, corresponding to Steps 4-8 of Algorithm \ref{alg:Framwork1}. 
The surrogate-based interpretation proceeds as follows. At outer iteration $t$, each node $i$:
\begin{enumerate}
	\item Performs inexact minimization via $K_g$ iteration steps:
	\begin{align*}
		\m{x}_i^{t,k+1} = \m{x}_i^{t,k} - \alpha \m{H}_i^{t,k}\nabla\phi_i^t(\m{x}_i^{t,k}), \quad k=0,\dots,K_g-1;
	\end{align*}
	\item Achieves consensus through $K_c$ communication rounds of form with all nodes stacked:
	\begin{align*}
		\m{x}^{t+1} &= \m{W}^{K_c}\m{x}^{t,K_g}, \\
		\m{v}^{t+1} &= \m{W}^{K_c}\left(\m{v}^{t}+\m{g}^{t+1}-\m{g}^t\right);
	\end{align*}
	\item Updates $\m{H}_i^{t+1}$ for the next cycle.
\end{enumerate}
It is worth noting that the local update recursions of the schemes \cite{berahas2024balancing,nguyen2023performance,ge2023gradient,huang2023computation} listed in Remark \ref{remark_2.2} admit the above surrogate-based representation. The RIA framework below is therefore used to analyze this specific class of updates. A key question is whether the local surrogate problem
\eqref{eq:mm_target} should be solved exactly. Under the general
assumptions considered here, exact minimization is neither required
nor necessarily well-defined.

\begin{remark}[On Exact Local Minimization]\label{Remark 2.3}
	The surrogate $\phi_i^t$ may be nonconvex and noncoercive, and hence
	an exact minimizer need not exist. Even when it exists, exact local
	minimization may produce large and heterogeneous displacements across
	the nodes, thereby increasing the consensus and gradient-tracking
	errors. We therefore employ a limited or adaptively selected number
	of local steps and require only the sufficient-descent and
	bounded-deviation conditions introduced below. Exact minimization may
	still be applicable under additional structural assumptions.
\end{remark}

\subsection{Robust inexact algorithm framework}\label{RIAsection}
Motivated by the limitations of fixed local updates and the above considerations, we propose the Robust Inexact Algorithm (RIA) framework. This framework formalizes the local update as an adaptive procedure that terminates based on solution quality rather than a fixed iteration count $K_g$.
The RIA framework abstracts the local update as an adaptive inexact procedure and takes the following form:
\begin{align}
	\m{x}_i^{t+1/2} &\leftarrow \text{InexactSolver}\left(\min_{\m{z}} \phi_i^t(\m{z})\right) \label{ia1}, \\
	\m{x}^{t+1} &= \m{W}^{K_c}\m{x}^{t+1/2} , \label{ia2} \\
	\m{v}^{t+1} &= \m{W}^{K_c}\left(\m{v}^{t} + \m{g}^{t+1} - \m{g}^t\right). \label{ia3}
\end{align}
Critically, $\m{x}_i^{t+1/2}$ is not required to be a stationary point. Instead, to guarantee convergence, the inexact solution should satisfy two specific regularity conditions stated in the following assumption.
\begin{assum}\label{IC}
	The approximate solution $\m{x}_i^{t+1/2}$ should satisfy:
	\begin{itemize}
		\item [1.] \textbf{Sufficient Descent: } 
	\begin{align}\label{as3.1}
		\phi_i^t(\m{x}_i^{t+1/2}) \leq \phi_i^t(\m{x}_i^{t})-\gamma_1 \Vert\m{v}_i^t\Vert^2
	\end{align}
	for some $\gamma_1>0$. 
		\item [2.] \textbf{Bounded Deviation:}
		 \begin{align}\label{as3.2}
			\Vert\m{x}_i^{t+1/2}-\m{x}_i^{t}\Vert^2 \leq \gamma_2\Vert\m{v}_i^t\Vert^2
		\end{align}
		for some $\gamma_2>0$.
	\end{itemize}
\end{assum}
The sufficient descent condition ensures that each inexact minimization makes sufficient progress in reducing the local surrogate function. On the other hand, the bounded deviation condition limits movement from the current iterate, preventing excessive deviations that could disrupt consensus.

\subsection{Convergence analysis for RIA}
We now establish the convergence guarantees for the proposed RIA applied to problem \eqref{obj_fun1}. We assume each local objective $f_i$ ($i=1,\ldots,n$) is differentiable with Lipschitz continuous gradients, but possibly nonconvex. Throughout this subsection, we consider $0<\sigma<1$; the exact-averaging case $\sigma=0$ will only be used in Proposition \ref{prop:tightness}, where its dynamics are analyzed directly. From the updating formula \eqref{update2}, the following invariant property holds by induction \cite{qu2017harnessing}.
\begin{equation}\label{vequalg}
	\m{M}{\m{v}}^t=\m{M}{\m{g}}^t \quad \Longleftrightarrow \quad \bar{\m{v}}^t=\bar{\m{g}}^t.
\end{equation}
This relation \eqref{vequalg} ensures that the average of the tracking variables precisely captures the average of the true local gradients $\bar{\m{g}}^t =\overline{\nabla} f(\m{x}^t)= \frac{1}{n} \sum_{i=1}^n \nabla f_i(\m{x}_i^t)$.
Given the invariant $\bar{\m{v}}^t = \bar{\m{g}}^t$ and the definition of stationary point \eqref{stationarity}, approaching stationarity requires simultaneously driving $\Vert\m{x}^t-\m{M}\m{x}^t\Vert$, $\Vert\m{v}^t-\m{M}\m{v}^t\Vert$, and $\Vert\m{v}^t\Vert$ to zero, and thus we define $\m{x}^t$ as an $\epsilon$-stationary solution if it satisfies
\begin{equation}\label{eps-stationary}
	\Vert\m{v}^t\Vert^2+\Vert\m{v}^t-\m{M}\m{v}^t\Vert^2+\Vert\m{x}^t-\m{M}\m{x}^t\Vert^2\leq \epsilon.
\end{equation}
To demonstrate that RIA generates a sequence satisfying \eqref{eps-stationary}, we construct the following potential function:
\begin{equation}\label{potential}
	P(\m{x}^{t},\m{v}^{t})=F(\bar{\m{x}}^{t})+\Vert\m{x}^{t}-\m{M}\m{x}^{t}\Vert^2+\frac{1-(1+\tau){\sigma}^{2K_c}}{16(1+1/\eta)L^2{\sigma}^{2K_c}}\Vert\m{v}^{t}-\m{M}\m{v}^{t}\Vert^2,
\end{equation}
where $\tau$ and $\eta$ are some positive constants introduced by Young's inequality. $F(\bar{\m{x}}^t)$ behaves as the global objective value at the averaged iterate $\bar{\m{x}}^t$, $\Vert\m{x}^{t}-\m{M}\m{x}^{t}\Vert^2$ represents the consensus error at $\m{x}^{t}$, and $\Vert\m{v}^{t}-\m{M}\m{v}^{t}\Vert^2$ behaves as the gradient tracking error and quantifies the disagreement between the average gradient approximation $\m{v}^t$ and the average gradient $\m{M}{\m{v}}^t=\m{M}{\m{g}}^t$.
Since the objective function is bounded below, the potential function $P(\m{x}^t, \m{v}^t)$ is naturally bounded below. Our analysis proceeds by proving the sufficient descent property of $P(\m{x}^t, \m{v}^t)$.

We first establish a descent bound at $\bar{\m{x}}^t$ and shared by RIA and AdaDQN. The bound only requires the two inexact-solution conditions and is therefore unaffected by the particular communication implementation.
\begin{lemma}\label{lem:descent_lemma}
	Suppose that Assumptions \ref{as0}, \ref{as1}, and \ref{IC} hold. Let the communication step satisfy $\m{M}\m{v}^t=\m{M}\m{g}^t$ and $\bar{\m{x}}^{t+1}=\bar{\m{x}}^{t+1/2}$. Then
	\begin{align}\label{term1}
		F(\bar{\m{x}}^{t+1})\leq& F(\bar{\m{x}}^{t})-\frac{2\gamma_1-L\gamma_2(2+d+b) }{2n}\Vert\m{v}^t\Vert^2\notag\\
		&+\frac{L}{2bn}\Vert\m{M}\m{x}^t-\m{x}^t \Vert^2+\frac{1}{2Ldn}\Vert{\m{v}}^t-\m{M}{\m{v}}^t\Vert^2.
	\end{align}
	where $b,d>0$ are arbitrary constants and $\gamma_1,\gamma_2$ are given by Assumption \ref{IC}.
\end{lemma}
\begin{proof}
	By the $L$-Lipschitz continuity of $\nabla F$, we have
	\begin{align}\label{Fbarx11}
		F(\bar{\m{x}}^{t+1})\leq F(\bar{\m{x}}^{t})+\Big\langle\nabla F(\bar{\m{x}}^{t}),\bar{\m{x}}^{t+1}-\bar{\m{x}}^{t}\Big\rangle+\frac{L}{2}\Vert\bar{\m{x}}^{t+1}-\bar{\m{x}}^{t}\Vert^{2}.
	\end{align}
	Substituting the relation $\bar{\m{x}}^{t+1}=\bar{\m{x}}^{t+1/2}$ into the inner product term in \eqref{Fbarx11} yields
	\begin{align*}
		&\Big\langle\nabla F(\bar{\m{x}}^{t}),\bar{\m{x}}^{t+1}-\bar{\m{x}}^{t}\Big\rangle=\frac{1}{n}\sum_{i=1}^n \Big  \langle\nabla F(\bar{\m{x}}^{t}),{\m{x}}_i^{t+1/2}-{\m{x}}_i^{t}\Big\rangle\\
		=&\underbrace{\frac{1}{n}\sum_{i=1}^n \Big  \langle \nabla F(\bar{\m{x}}^{t}) -\overline{\nabla} f({\m{x}}^t),{\m{x}}_i^{t+1/2}-{\m{x}}_i^{t}\Big\rangle}_{\text{term~I}}\\
		&+\underbrace{\frac{1}{n}\sum_{i=1}^n \Big  \langle \overline{\nabla} f({\m{x}}^t)- \m{v}_i^t,{\m{x}}_i^{t+1/2}-{\m{x}}_i^{t}\Big\rangle}_{\text{term~II}}+\underbrace{\frac{1}{n}\sum_{i=1}^n \Big  \langle  \m{v}_i^t,{\m{x}}_i^{t+1/2}-{\m{x}}_i^{t}\Big\rangle}_{\text{term~III}}.
	\end{align*}
	We bound each term separately. For term I, using Young's inequality with parameter $1/(bL)>0$. we have
	\begin{eqnarray}\label{termB}
		&&	\frac{1}{n} \sum_{i=1}^n \Big\langle\nabla F(\bar{\m{x}}^t)-\overline{\nabla} f({\m{x}}^t),{\m{x}}_i^{t+1/2}-{\m{x}}_i^{t}\Big\rangle \notag\\
		&\leq& \frac{1}{2bL}\Vert\nabla F(\bar{\m{x}}^t)-\overline{\nabla} f({\m{x}}^t)\Vert^2+\frac{Lb}{2n}\Vert{\m{x}}^{t+1/2}-{\m{x}}^{t}\Vert^2 \notag\\
		&\leq & \frac{L}{2bn}\Vert\m{M}\m{x}^t-\m{x}^t \Vert^2+\frac{\gamma_2Lb}{2n}\Vert{\m{v}}^{t}\Vert^2,
	\end{eqnarray}
	where the last inequality uses the $L$-Lipschitz continuity of each $\nabla f_i$ and \eqref{as3.2}.
	For term II, using Young's inequality with parameter $1/(dL)>0$, we have
	\begin{align}\label{termC}
		\frac{1}{n} \sum_{i=1}^n \Big\langle\overline{\nabla} f({\m{x}}^t)-{\m{v}}_i^t,{\m{x}}_i^{t+1/2}-{\m{x}}_i^{t}\Big\rangle \leq &\frac{1}{2Ldn}\Vert\m{1}\otimes\overline{\nabla} f({\m{x}}^t)-{\m{v}}^t\Vert^2+\frac{Ld}{2n}\Vert{\m{x}}^{t+1/2}-{\m{x}}^{t}\Vert^2 \notag \\
		= & \frac{1}{2Ldn}\Vert{\m{v}}^t-\m{M}{\m{v}}^t\Vert^2+\frac{\gamma_2Ld}{2n}\Vert{\m{v}}^{t}\Vert^2,
	\end{align}
	where the equality applies \eqref{vequalg} and \eqref{as3.2}.	For the term III, first using \eqref{as3.1} and rearranging, we obtain
	\begin{align*}
		f_i(\m{x}_i^{t+1/2}) -f_i(\m{x}_i^{t})+ \langle\m{v}_i^t - \nabla f_i(\m{x}_i^t), \m{x}_i^{t+1/2}-\m{x}_i^{t}\rangle
		\leq -\gamma_1 \Vert\m{v}_i^t\Vert^2.
	\end{align*}
	Then, employing the $L$-Lipschitz continuity of $\nabla f_i$ with \eqref{as3.2} yields
	\begin{align}\label{termA}
		\frac{1}{n}\sum_{i=1}^n \langle\m{v}_i^t, \m{x}_i^{t+1/2}-\m{x}_i^{t}\rangle
		\leq -\frac{2\gamma_1-L\gamma_2 }{2n}\Vert\m{v}^{t}\Vert^2.
	\end{align}
	Applying $\bar{\m{x}}^{t+1}=\bar{\m{x}}^{t+1/2}$ and \eqref{as3.2} to the quadratic term in \eqref{Fbarx11} gives
	\begin{align}\label{reminder}
		\frac{L}{2}\Vert\bar{\m{x}}^{t+1}-\bar{\m{x}}^{t}\Vert^{2}\leq \frac{L \gamma_2}{2n} \Vert\m{v}^t\Vert^2
	\end{align}
	Substituting \eqref{termC}, \eqref{termA}, \eqref{termB}, and \eqref{reminder} back into \eqref{Fbarx11} completes the proof.
\end{proof}

We next analyze the recursive upper bound of the gradient tracking error $\Vert\m{v}^t - \m{M}\m{v}^t\Vert^2$. The following lemma shows that this bound contracts over iterations, subject to perturbations from the consensus error $\Vert\m{x}^t-\m{M}\m{x}^t\Vert^2$ and the tracker norm $\Vert \m{v}^t \Vert^2$.
\begin{lemma}
	Suppose that Assumptions \ref{as0}, \ref{as1}, \ref{as_network}, and \ref{IC} hold. Let $\{\m{x}^t\}$ be the sequence generated by the RIA framework. Then, we have the following inequality
	\begin{align}\label{term2}
		&\frac{1-(1+\tau){\sigma}^{2K_c}}{16(1+1/\eta)L^2{\sigma}^{2K_c}}\Vert{\m{v}}^{t+1}-\m{M}{\m{v}}^{t+1}\Vert^2\leq\frac{1-(1+\tau){\sigma}^{2K_c}}{16(1+1/\eta)L^2{\sigma}^{2K_c}}\Vert{\m{v}}^{t}-\m{M}{\m{v}}^{t}\Vert^2\notag\\
		&- \frac{(1-(1+\tau){\sigma}^{2K_c})(1-(1+\eta)\sigma^{2K_c})}{16(1+1/\eta)L^2{\sigma}^{2K_c}}\Vert{\m{v}}^t-\m{M}{\m{v}}^t\Vert^2\notag\\
		&+\frac{1}{2}(1-(1+\tau){\sigma}^{2K_c})\Vert\m{M}\m{x}^t-\m{x}^t\Vert^2+\frac{1}{8}(1-(1+\tau){\sigma}^{2K_c}) \gamma_2\Vert\m{v}^t\Vert^2,
	\end{align}
	where $\eta,\tau >0$ are some constants introduced by Young's inequality.
\end{lemma}
\begin{proof}
	Using Young's inequality with parameter $\eta >0$, Lemma \ref{property W}, and the $L$-Lipschitz continuity of $\nabla f_i$, we establish 
	\begin{align*}
		&\Vert{\m{v}}^{t+1}-\m{M}{\m{v}}^{t+1}\Vert^2
		=\Vert \m{W}^{K_c}{\m{v}}^{t}+\m{W}^{K_c}\m{g}^{t+1}-\m{W}^{K_c}\m{g}^{t} -\m{M}\m{v}^t-\m{M}\m{g}^{t+1}+\m{M}\m{g}^t\Vert^2 \notag\\
		\leq & (1+\eta)\Vert\m{W}^{K_c}{\m{v}}^t-\m{M}{\m{v}}^t\Vert^2+(1+1/\eta)\Vert(\m{W}^{K_c}-\m{M})(\m{g}^{t+1}-\m{g}^t)\Vert^2\notag\\
		\leq & (1+\eta)\sigma^{2K_c}\Vert{\m{v}}^t-\m{M}{\m{v}}^t\Vert^2+(1+1/\eta)\sigma^{2K_c}L^2\Vert\m{x}^{t+1}-\m{x}^t\Vert^2.\notag
	\end{align*}
	Observe that $\m{x}^{t+1}-\m{x}^t=\m{W}^{K_c}\m{x}^t-\m{M}\m{x}^t+\m{M}\m{x}^t-\m{x}^t+\m{W}^{K_c}(\m{x}^{t+1/2}-\m{x}^t)$. Then, invoking Lemmas \ref{property W} and \ref{important} gives
	\begin{align*}
		\Vert\m{x}^{t+1}-\m{x}^t\Vert^2 \leq 8\Vert\m{M}\m{x}^t-\m{x}^t\Vert^2+2\Vert\m{x}^{t+1/2}-\m{x}^t\Vert^2.
	\end{align*}
	Combining the above two inequalities with \eqref{as3.2}, we get
	\begin{align*}
		&\Vert{\m{v}}^{t+1}-\m{M}{\m{v}}^{t+1}\Vert^2
		\leq  (1+\eta)\sigma^{2K_c}\Vert{\m{v}}^t-\m{M}{\m{v}}^t\Vert^2\\
		&+8(1+1/\eta)\sigma^{2K_c}L^2\Vert\m{M}\m{x}^t-\m{x}^t\Vert^2+2(1+1/\eta)\sigma^{2K_c}L^2 \gamma_2\Vert\m{v}^t\Vert^2,\notag
	\end{align*}
	Multiplying both sides of the above inequality by $\frac{1-(1+\tau){\sigma}^{2K_c}}{16(1+1/\eta)L^2{\sigma}^{2K_c}}$ and rearranging terms yields the desired inequality.
\end{proof}

The following Lemma derives a recursion relationship of the consensus error $\Vert\m{x}^t - \m{M}\m{x}^t\Vert^2$.
\begin{lemma}\label{lem:consensus_error}
	Suppose that Assumptions \ref{as0}, \ref{as1}, \ref{as_network}, and \ref{IC} hold. Let $\{\m{x}^t\}$ be the sequence generated by the RIA framework. Then, we have the following inequality
	\begin{align}\label{term3}
		\Vert\m{x}^{t+1}-\m{M}\m{x}^{t+1}\Vert^2
		\leq(1+\tau){\sigma}^{2K_c} \Vert\m{x}^{t}-\m{M}\m{x}^{t}\Vert^2+(1+1/\tau){\sigma}^{2K_c}\gamma_2\Vert{\m{v}}^t\Vert^2, 
	\end{align}	
	where $\tau >0$ is some constant introduced by Young's inequality
\end{lemma}
\begin{proof}
	Recall the relation that $\m{x}^{t+1}-\m{M}\m{x}^{t+1}=(\m{W}^{K_c}-\m{M})(\m{x}^{t}-\m{M}\m{x}^t)+(\m{W}^{K_c}-\m{M})(\m{x}^{t+1/2}-\m{x}^t)$. Applying Young's inequality with  parameter $\tau >0$, Lemma \ref{property W}, and \eqref{as3.2} directly leads to the stated result.
\end{proof}

With the recursion relationships of the objective function, gradient tracking error, and consensus error respectively established in Lemmas \ref{lem:descent_lemma}-\ref{lem:consensus_error}, we are now ready to state the convergence theorem.
\begin{theorem}\label{RIA_conv}
	Suppose that Assumptions \ref{as0}, \ref{as1}, \ref{as_network}, and \ref{IC} hold. Let $\{\m{x}^t\}$ be the sequence generated by the RIA framework and fix the positive parameters introduced by Young's inequality as $\eta = \tau = \frac{1-\sigma^{2K_c}}{2\sigma^{2K_c}}$, $b = \frac{4L}{n(1-\sigma^{2K_c})}$, and $d = \frac{64(1+\sigma^{2K_c})\sigma^{2K_c} L}{n (1-\sigma^{2K_c})^3}$. If the parameters $\gamma_1$ and $\gamma_2$ in Assumption \ref{IC} are selected to ensure 
	\begin{equation}\label{convergence-condition}
			 \gamma_1 - \gamma_2 \bigg( \frac{L}{2} + \frac{2L^2}{n(1-\sigma^{2K_c})} 
		 +\frac{64\sigma^{2K_c} L^2}{n(1-\sigma^{2K_c})^3} + \frac{3n(1-\sigma^{2K_c})}{16} \bigg) >0,
	\end{equation}
	then the RIA framework achieves the convergence rate:
	\begin{equation}\label{convergence-rate}
		\min_{0 \le t \le T} \left\{ \Vert\m{v}^t\Vert^2 + \Vert\m{v}^t - \m{M}\m{v}^t\Vert^2 + \Vert\m{x}^t - \m{M}\m{x}^t\Vert^2 \right\} \leq \frac{P(\m{x}^{0},\m{v}^{0})-\underline{F}}{c T},  
	\end{equation}
	where $P$ is the potential function explicitly defined as
	\begin{equation}\label{potential_explicit}
		P(\m{x}^{t},\m{v}^{t}) = F(\bar{\m{x}}^{t}) + \Vert\m{x}^{t}-\m{M}\m{x}^{t}\Vert^2 + \frac{(1-\sigma^{2K_c})^2}{16(1+\sigma^{2K_c})L^2}\Vert\m{v}^{t}-\m{M}\m{v}^{t}\Vert^2,
	\end{equation}
	and the constant $c > 0$ is defined as
	\begin{align}\label{c_new}
		c = \min \left\{ c_f, \frac{1-\sigma^{2K_c}}{8}, \frac{(1-\sigma^{2K_c})^3}{128(1+\sigma^{2K_c})L^2\sigma^{2K_c}} \right\},
	\end{align}
	with the coefficient $c_f$ given by
	\begin{align}\label{gamma_condition_new}
	c_f = &\frac{1}{n} \bigg[ \gamma_1 - \gamma_2 \bigg( \frac{L}{2} + \frac{2L^2}{n(1-\sigma^{2K_c})} \notag\\
& +\frac{32(1+\sigma^{2K_c})\sigma^{2K_c} L^2}{n(1-\sigma^{2K_c})^3} + \frac{n(1-\sigma^{2K_c})(1+2\sigma^{2K_c})}{16} \bigg) \bigg].
	\end{align}
\end{theorem}

\begin{proof}
	By summing the recursive upper bounds established in Lemma \ref{lem:descent_lemma}, Lemma 3.2, and Lemma 3.3, and applying the explicit formulation of the potential function $P(\m{x}^t, \m{v}^t)$ defined in \eqref{potential_explicit}, we obtain the following descent relation:
	\begin{align}\label{P-rec_new}
		P(\m{x}^{t+1}, \m{v}^{t+1}) \leq P(\m{x}^t, \m{v}^t) - c_f \Vert\m{v}^t\Vert^2 - c_x \Vert\m{x}^t - \m{M}\m{x}^t\Vert^2 - c_v \Vert\m{v}^t - \m{M}\m{v}^t\Vert^2,
	\end{align}
	where the structural tracking coefficients $c_f$, $c_x$, and $c_v$ are defined as
	\begin{align*}
		c_f &= \frac{2\gamma_1 - L\gamma_2(2+d+b)}{2n} - \frac{1}{8}(1-(1+\tau)\sigma^{2K_c})\gamma_2 \\
		& \quad - \left(\frac{(1-\sigma^{2K_c})^2}{16(1+\sigma^{2K_c})L^2}\right) \cdot 2(1+1/\eta)\sigma^{2K_c}L^2\gamma_2, \\
		c_x &= \frac{1}{2}(1-(1+\tau)\sigma^{2K_c}) - \frac{L}{2bn}, \\
		c_v &= \frac{(1-(1+\tau)\sigma^{2K_c})(1-(1+\eta)\sigma^{2K_c})}{16(1+1/\eta)L^2\sigma^{2K_c}} - \frac{1}{2Ldn}.
	\end{align*}
	 Recalling our parameter setting $\eta = \tau = \frac{1-\sigma^{2K_c}}{2\sigma^{2K_c}}$, we have $1-(1+\tau)\sigma^{2K_c} = \frac{1-\sigma^{2K_c}}{2}$, $1-(1+\eta)\sigma^{2K_c} = \frac{1-\sigma^{2K_c}}{2}$, and $1+1/\eta = \frac{1+\sigma^{2K_c}}{1-\sigma^{2K_c}}$. Substituting the choices of $b = \frac{4L}{n(1-\sigma^{2K_c})}$ and $d = \frac{64(1+\sigma^{2K_c})\sigma^{2K_c} L}{n (1-\sigma^{2K_c})^3}$ into $c_x$ and $c_v$, we obtain
	\begin{align*}
		c_x &= \frac{1-\sigma^{2K_c}}{4} - \frac{1-\sigma^{2K_c}}{8} = \frac{1-\sigma^{2K_c}}{8}, \\
		c_v &= \frac{(1-\sigma^{2K_c})^3}{64(1+\sigma^{2K_c})L^2\sigma^{2K_c}} - \frac{1}{2L\left(\frac{64(1+\sigma^{2K_c})\sigma^{2K_c} L}{n(1-\sigma^{2K_c})^3}\right)n} \\
		&= \frac{(1-\sigma^{2K_c})^3}{64(1+\sigma^{2K_c})L^2\sigma^{2K_c}} - \frac{(1-\sigma^{2K_c})^3}{128(1+\sigma^{2K_c})L^2\sigma^{2K_c}} 
		= \frac{(1-\sigma^{2K_c})^3}{128(1+\sigma^{2K_c})L^2\sigma^{2K_c}}.
	\end{align*}
	
	Now we simplify the coefficient $c_f$ and rewrite $c_f$ as
	\begin{align*}
		c_f &= \frac{\gamma_1}{n} - \frac{L\gamma_2}{2n} - \frac{L\gamma_2 b}{2n} - \frac{L\gamma_2 d}{2n} - \frac{1-\sigma^{2K_c}}{16}\gamma_2 - \frac{(1-\sigma^{2K_c})^2}{8(1+\sigma^{2K_c})} \cdot \left(\frac{1+\sigma^{2K_c}}{1-\sigma^{2K_c}}\right)\sigma^{2K_c}\gamma_2 \\
		&= \frac{\gamma_1}{n} - \frac{L\gamma_2}{2n} - \frac{L\gamma_2 b}{2n} - \frac{L\gamma_2 d}{2n} - \frac{1-\sigma^{2K_c}}{16}\gamma_2 - \frac{\sigma^{2K_c}(1-\sigma^{2K_c})}{8}\gamma_2 \\
		&= \frac{1}{n} \left[ \gamma_1 - \gamma_2 \left( \frac{L}{2} + \frac{L b}{2} + \frac{L d}{2} + \frac{n(1-\sigma^{2K_c})(1+2\sigma^{2K_c})}{16} \right) \right].
	\end{align*}
	Substituting the choices of $b = \frac{4L}{n(1-\sigma^{2K_c})}$ and $d = \frac{64(1+\sigma^{2K_c})\sigma^{2K_c} L}{n (1-\sigma^{2K_c})^3}$ into the above relation gives
	\begin{align*}
		c_f = &\frac{1}{n} \bigg[ \gamma_1 - \gamma_2 \bigg( \frac{L}{2} + \frac{2L^2}{n(1-\sigma^{2K_c})} \\
		& +\frac{32(1+\sigma^{2K_c})\sigma^{2K_c} L^2}{n(1-\sigma^{2K_c})^3} + \frac{n(1-\sigma^{2K_c})(1+2\sigma^{2K_c})}{16} \bigg) \bigg].
	\end{align*}
	Finally, an accumulation of \eqref{P-rec_new} from $t = 0$ to $T-1$ completes the proof.
\end{proof}

Now we verify the two RIA conditions for the fixed-local-update gradient-tracking schemes \cite{berahas2024balancing,nguyen2023performance,ge2023gradient,huang2023computation} listed in Remark \ref{remark_2.2}, whose local update recursion takes the form
	\begin{align*}
	\m{x}_i^{t,k+1} = \m{x}_i^{t,k} - \alpha \m{H}_i^{t,k}\nabla\phi_i^t(\m{x}_i^{t,k}), \quad k=0,\dots,K_g-1.
\end{align*}
Assume the matrices $\{\m{H}_i^{t,k}\}$ have uniform eigenvalue bounds, i.e., $\psi \m{I} \preceq \m{H}_i^{t,k} \preceq \Psi \m{I}$, where $0<\psi \leq \Psi$.

We first verify the bounded deviation condition that $\Vert\mathbf{x}_i^{t,K_g} - \mathbf{x}_i^{t}\Vert^2 \leq \gamma_2 \Vert\mathbf{v}_i^t\Vert^2$ for some $\gamma_2 > 0$. Recall the definition of the local surrogate function in the surrogate-based framework:
\begin{equation*}
	\phi_i^t(\mathbf{z}) = f_i(\mathbf{z}) + \langle \mathbf{v}_i^t - \nabla f_i(\mathbf{x}_i^t), \mathbf{z} \rangle.
\end{equation*}
Note that the gradient of this surrogate is $\nabla \phi_i^t(\mathbf{z}) = \nabla f_i(\mathbf{z}) + \mathbf{v}_i^t - \nabla f_i(\mathbf{x}_i^t)$. Crucially, the gradient tracking update $\mathbf{v}_i^{t,k} = \mathbf{v}_i^{t,k-1} + \nabla f_i(\mathbf{x}_i^{t,k}) - \nabla f_i(\mathbf{x}_i^{t,k-1})$ with $\mathbf{v}_i^{t,0} = \mathbf{v}_i^t$ ensures that
\begin{equation*}
	\mathbf{v}_i^{t,k} = \nabla \phi_i^t(\mathbf{x}_i^{t,k}).
\end{equation*}
By recursion, the updating formula for $\m{v}_i^{t,k}$ becomes
\begin{align}\label{vtk}
	\m{v}_i^{t,k}=\m{v}_i^{t}+\nabla f_i(\m{x}_i^{t,k})-\nabla f_i(\m{x}_i^{t}).
\end{align}
Using the above relation and the $L$-Lipschitz continuity of $\nabla f_i$, we have for $k \ge 1$ that
\begin{align*}
	&\quad \Vert\m{x}_i^{t,k}-\m{x}_i^t\Vert = \left\Vert \alpha\sum_{j=0}^{k-1} \m{H}_i^{t,j}\m{v}_i^{t,j} \right\Vert \leq \alpha \Psi \sum_{j=0}^{k-1}\Vert\m{v}_i^{t,j}\Vert\\
	&\leq \alpha \Psi \sum_{j=0}^{k-1} \left( \Vert\m{v}_i^t\Vert + \Vert\nabla f_i(\m{x}_i^{t,j})-\nabla f_i(\m{x}_i^{t})\Vert \right) 
	\leq k \alpha \Psi \Vert\m{v}_i^t\Vert + \alpha L \Psi\sum_{j=1}^{k-1} \Vert\m{x}_i^{t,j}-\m{x}_i^t\Vert.
\end{align*}
We proceed by induction. For $k=1$, $\Vert\m{x}_i^{t,1}-\m{x}_i^t\Vert = \alpha \Vert\m{H}_i^{t,0}\m{v}_i^t\Vert \leq \alpha\Psi\Vert\m{v}_i^t\Vert$, which satisfies the bound. Assume the hypothesis holds for all $j < k$. Then,
\begin{align*}
	&\quad \Vert\m{x}_i^{t,k}-\m{x}_i^t\Vert \leq k \alpha \Psi \Vert\m{v}_i^t\Vert + \alpha L \Psi\sum_{j=1}^{k-1} (2\alpha \Psi j \Vert\m{v}_i^t\Vert) \\
	&= \alpha \Psi k\Vert\m{v}_i^t\Vert + 2\alpha^2 L \Psi^2 \frac{k(k-1)}{2} \Vert\m{v}_i^t\Vert 
	= \alpha \Psi k \left[1 + \alpha L \Psi (k-1)\right] \Vert\m{v}_i^t\Vert.
\end{align*}
Provided that $\alpha \leq \frac{1}{L\Psi K_g}$, we have $1 + \alpha L \Psi (k-1) \leq 2$. Thus, 
\begin{equation}\label{eq 3.25}
	\Vert\mathbf{x}_i^{t,k} - \mathbf{x}_i^t\Vert \leq 2\alpha \Psi k \Vert\mathbf{v}_i^t\Vert.
\end{equation}
Taking $k=K_g$ and squaring both sides of \eqref{eq 3.25}, we obtain
\begin{align*}
	\Vert\mathbf{x}_i^{t,K_g} - \mathbf{x}_i^t\Vert^2 \leq 4\alpha^2 \Psi^2 K_g^2 \Vert\mathbf{v}_i^t\Vert^2.
\end{align*}
Thus, the bounded deviation condition holds with
\begin{equation}\label{val_gamma2}
	\gamma_2 = 4\alpha^2 \Psi^2 K_g^2.
\end{equation}

Then, we verify the sufficient descent condition that $\phi_i^t(\mathbf{x}_i^{t,K_g}) \leq \phi_i^t(\mathbf{x}_i^{t}) - \gamma_1 \Vert\mathbf{v}_i^t\Vert^2$ for some constant $\gamma_1 > 0$.  Since $\nabla f_i$ is $L$-Lipschitz continuous, the surrogate function $\phi_i^t$ is $L$-smooth. Applying the descent lemma for each inner step $k=0,\dots,K_g-1$ with the update rule $\mathbf{x}_i^{t,k+1} = \mathbf{x}_i^{t,k} - \alpha \mathbf{H}_i^{t,k}\mathbf{v}_i^{t,k}$ yields
\begin{align*}
	\phi_i^t(\mathbf{x}_i^{t,k+1}) &\leq \phi_i^t(\mathbf{x}_i^{t,k}) - \alpha \langle \mathbf{v}_i^{t,k}, \mathbf{H}_i^{t,k}\mathbf{v}_i^{t,k} \rangle + \frac{L}{2}\alpha^2 \Vert\mathbf{H}_i^{t,k}\mathbf{v}_i^{t,k}\Vert^2 \\
	&\leq \phi_i^t(\mathbf{x}_i^{t,k}) - \left( \alpha \psi - \frac{L}{2}\alpha^2 \Psi^2 \right) \Vert\mathbf{v}_i^{t,k}\Vert^2,
\end{align*}
where we utilize the uniform eigenvalue bounds $\psi \mathbf{I} \preceq \mathbf{H}_i^{t,k} \preceq \Psi \mathbf{I}$. Summing the above inequality over $k=0$ to $K_g-1$, we obtain
\begin{equation}\label{eq:total_descent}
	\phi_i^t(\mathbf{x}_i^{t,K_g}) \leq \phi_i^t(\mathbf{x}_i^{t}) - c_\alpha \sum_{k=0}^{K_g-1} \Vert\mathbf{v}_i^{t,k}\Vert^2,
\end{equation}
where $c_\alpha = \alpha \psi - \frac{L}{2}\alpha^2 \Psi^2 > 0$ for $\alpha < \frac{2\psi}{L\Psi^2}$. 
Using \eqref{vtk}, Lemma \ref{important}, and the $L$-Lipschitz continuity of each $\nabla f_i$, we have
\begin{equation*}
	\Vert\mathbf{v}_i^{t,k}\Vert^2 \geq \frac{1}{2}\Vert\mathbf{v}_i^t\Vert^2 - L^2 \Vert\mathbf{x}_i^{t,k} - \mathbf{x}_i^t\Vert^2.
\end{equation*}
Summing over $k=0$ to $K_g-1$ and substituting \eqref{eq 3.25}, we obtain
\begin{align*}
	\sum_{k=0}^{K_g-1} \Vert\mathbf{v}_i^{t,k}\Vert^2 \geq \frac{K_g}{2}\Vert\mathbf{v}_i^t\Vert^2 - L^2 \sum_{k=0}^{K_g-1} 4\alpha^2\Psi^2 k^2 \Vert\mathbf{v}_i^t\Vert^2 
	\geq \left( \frac{K_g}{2} - \frac{4}{3}\alpha^2 L^2 \Psi^2 K_g^3 \right) \Vert\mathbf{v}_i^t\Vert^2,
\end{align*}
where the last inequality uses the bound $\sum_{k=0}^{K_g-1} k^2 \leq K_g^3/3$. By restricting the stepsize such that $\frac{4}{3}\alpha^2 L^2 \Psi^2 K_g^3 \leq \frac{K_g}{4}$, which mandates $\alpha \leq \frac{\sqrt{3}}{4 L \Psi K_g}$, then we obtain $\sum_{k=0}^{K_g-1} \Vert\mathbf{v}_i^{t,k}\Vert^2 \geq \frac{K_g}{4}\Vert\mathbf{v}_i^t\Vert^2$. Substituting this back into \eqref{eq:total_descent} establishes the sufficient descent condition with
\begin{equation}\label{val_gamma1}
	\gamma_1 = \frac{K_g}{4} \left( \alpha \psi - \frac{L}{2}\alpha^2 \Psi^2 \right).
\end{equation}

As presented in Theorem \ref{RIA_conv}, the RIA framework should satisfy the condition \eqref{convergence-condition} for convergence. Substituting \eqref{val_gamma1} and \eqref{val_gamma2} back yields
\begin{align}
	&\frac{K_g}{4}\left(\alpha \psi - \frac{L}{2}\alpha^2 \Psi^2\right) - 4\alpha^2 \Psi^2 K_g^2 \bigg( \frac{L}{2} + \frac{2L^2}{n(1-\sigma^{2K_c})} \notag\\
	&+\frac{64\sigma^{2K_c} L^2}{n(1-\sigma^{2K_c})^3} + \frac{3n(1-\sigma^{2K_c})}{16} \bigg) >0,\label{rem3.3 eq1}
\end{align}
which implies that the convergence stepsize bound evaluates to
\begin{equation}\label{rem3.3_final_stepsize}
	\alpha < \frac{\psi}{\frac{L\Psi^2}{2} + 16\Psi^2 K_g \big( \frac{L}{2} + \frac{2L^2}{n(1-\sigma^{2K_c})} +\frac{64\sigma^{2K_c} L^2}{n(1-\sigma^{2K_c})^3} + \frac{3n(1-\sigma^{2K_c})}{16} \big)} = \mathcal{O}\left(\frac{1}{K_g}\right).
\end{equation}
This derivation verifies the two RIA conditions for the fixed-local-update gradient-tracking schemes considered above. On the other hand, it yields a sufficient stepsize bound scaling as $\mathcal{O}(1/K_g)$ to counteract the node deviation accumulated by fixed multiple local updates. 
This phenomenon is closely related to the limitations observed in federated learning under heterogeneous (non-IID) settings. Specifically, as demonstrated in the discussion of the SCAFFOLD algorithm \cite{karimireddy2020scaffold}, traditional local update methods (e.g., FedAvg \cite{Li2020On}) may require much smaller stepsizes to counteract the client drift caused by data heterogeneity. Inspired by their counterexample, we present the following proposition to establish the worst-case tightness of the $\mathcal{O}(1/K_g)$ stepsize scaling for the fixed-local-update scheme considered above.

\begin{proposition}[Worst-Case Tightness of the $\mathcal{O}(1/K_g)$ Stepsize Scaling]\label{prop:tightness}
	Consider the fixed-local-update gradient-tracking scheme in Algorithm \ref{alg:Framwork1} with $\m{H}_i^{t,k}=\m{I}$, $K_c=1$, and exact averaging over a two-node network. For every integer $K_g\geq4$ and every stepsize satisfying
	\begin{equation}\label{eq:necessary_stepsize}
		\alpha>\frac{8}{LK_g},
	\end{equation}
	there exists a convex quadratic decentralized optimization problem with $L$-Lipschitz continuous local gradients such that the generated iterates diverge geometrically from a suitable initialization. Consequently, any stepsize that guarantees convergence uniformly over this problem class must satisfy $\alpha=\mathcal{O}(1/(LK_g))$. Together with the sufficient bound derived above, this shows that the dependence on $K_g$ is tight up to constant factors for this fixed-local-update scheme.
\end{proposition}

\begin{proof}
	Let $a=\alpha L$ and consider a fully connected two-node network with the mixing matrix $\tilde{\m{W}}=\frac{1}{2}\m{1}_2\m{1}_2\tr$. We distinguish two cases.
	
	\textbf{Case 1: $0<a\leq2$.}
	Consider the one-dimensional local objective functions
	\begin{equation}\label{eq:lower_instance_heterogeneous}
		f_1(x)=\frac{L}{2}(x-c)^2,\qquad f_2(x)=0,
	\end{equation}
	where $c\in\mathbb{R}$ is a constant. The corresponding global objective is $F(x)=\frac{L}{4}(x-c)^2$, which is strongly convex and bounded below. Because the mixing matrix performs exact averaging and the gradient tracker is initialized by $\m{v}^0=\m{g}^0$, after the first outer iteration it holds that
	\begin{equation}\label{eq:lower_invariant}
		x_1^t=x_2^t=x^t,\qquad v_1^t=v_2^t=\nabla F(x^t)=\frac{L}{2}(x^t-c),\qquad t\geq1.
	\end{equation}
	This relation is preserved by all subsequent outer iterations. Let $e^t=x^t-c$. During the local-update phase, the surrogate gradient at node one is
	\begin{equation*}
		\nabla\phi_1^t(z)=L(z-c)-\frac{L}{2}e^t.
	\end{equation*}
	Therefore, its local recursion satisfies
	\begin{equation*}
		x_1^{t,k+1}-c-\frac{e^t}{2}=(1-a)\left(x_1^{t,k}-c-\frac{e^t}{2}\right),
	\end{equation*}
	which yields
	\begin{equation}\label{eq:lower_node1}
		x_1^{t,K_g}=c+\frac{e^t}{2}+(1-a)^{K_g}\frac{e^t}{2}.
	\end{equation}
	Furthermore, the surrogate gradient at node two is constant, i.e., $\nabla\phi_2^t(z)=\frac{L}{2}e^t$. Thus,
	\begin{equation}\label{eq:lower_node2}
		x_2^{t,K_g}=c+e^t-\frac{K_ga}{2}e^t.
	\end{equation}
	Averaging \eqref{eq:lower_node1} and \eqref{eq:lower_node2} gives the error dynamics
	\begin{equation}\label{eq:lower_error_dynamics}
		e^{t+1}=\underbrace{\frac{1}{4}\left[3+(1-a)^{K_g}-K_ga\right]}_{\mathcal{M}(a,K_g)}e^t.
	\end{equation}
	Since $0<a\leq2$, we have $(1-a)^{K_g}\leq1$. Under \eqref{eq:necessary_stepsize}, $K_ga>8$, and hence
	\begin{equation*}
		\mathcal{M}(a,K_g)<\frac{1}{4}(3+1-8)=-1.
	\end{equation*}
	Choosing an initialization such that the error after the first averaging step is nonzero, the error sequence generated by \eqref{eq:lower_error_dynamics} diverges geometrically.
	
	\textbf{Case 2: $a>2$.}
	We consider instead the homogeneous quadratic objectives
	\begin{equation}\label{eq:lower_instance_homogeneous}
		f_1(x)=f_2(x)=\frac{L}{2}x^2.
	\end{equation}
	Starting from a consensual initialization $x_1^0=x_2^0=x^0\neq0$, the two nodes follow the same local trajectory. After $K_g$ local gradient steps, we have
	\begin{equation*}
		x_i^{t,K_g}=(1-a)^{K_g}x^t,\qquad i=1,2.
	\end{equation*}
	The averaging step does not alter the common iterate, so
	\begin{equation}\label{eq:lower_homogeneous_dynamics}
		x^{t+1}=(1-a)^{K_g}x^t.
	\end{equation}
	Since $a>2$, we have $|1-a|>1$, and the iterates in \eqref{eq:lower_homogeneous_dynamics} diverge geometrically.
	
	Combining the two cases, for every $K_g\geq4$ and every $\alpha>8/(LK_g)$, there exists an $L$-smooth convex quadratic instance for which the fixed-local-update scheme diverges. Therefore, a stepsize that guarantees convergence uniformly over all such instances must satisfy $\alpha=\mathcal{O}(1/(LK_g))$, which completes the proof.
\end{proof}

\begin{corollary}[Gradient complexity of fixed local updates]\label{cor:fixed_oracle}
Let
\[
	\mathcal{R}^t:=\|\m{v}^t\|^2+\|\m{v}^t-\m{M}\m{v}^t\|^2+\|\m{x}^t-\m{M}\m{x}^t\|^2.
\]
Fix the network and $K_c$, and choose $\alpha=\bar{\alpha}/K_g$, where $\bar{\alpha}>0$ is sufficiently small and independent of $K_g$. Then there exists $c_0>0$, independent of $K_g$, such that the fixed-local-update scheme satisfies
\[
	\min_{0\leq t<T}\mathcal{R}^t
	\leq\frac{P(\m{x}^0,\m{v}^0)-\underline F}{c_0T}.
\]
Consequently, to obtain $\mathcal{R}^t\leq\delta$, the total number of local gradient evaluations and communication rounds satisfy
\begin{equation}\label{eq:fixed_oracle}
	\mathcal{G}_{\rm fix}(\delta)=\mathcal{O}\left(\frac{nK_g}{\delta}\right),
	\qquad
	\mathcal{C}_{\rm fix}(\delta)=\mathcal{O}\left(\frac{K_c}{\delta}\right).
\end{equation}
\end{corollary}
\begin{proof}
Substituting $\alpha=\bar{\alpha}/K_g$ into \eqref{val_gamma1}--\eqref{val_gamma2} gives
\[
	\gamma_1=\frac{\bar{\alpha}\psi}{4}
	-\frac{L\bar{\alpha}^2\Psi^2}{8K_g},
	\qquad
	\gamma_2=4\bar{\alpha}^2\Psi^2.
\]
Thus, for sufficiently small $\bar{\alpha}$, the coefficient $c$ in Theorem \ref{RIA_conv} is bounded below independently of $K_g$. Each outer iteration evaluates $K_g$ local gradients at every node and performs $K_c$ communication rounds, which gives \eqref{eq:fixed_oracle}. 
\end{proof}


\section{A new method with adaptive local updates and broadcasts}
The RIA framework in Section \ref{RIAsection} provides a sufficient convergence condition for the fixed-local-update gradient-tracking schemes considered above, leading to the worst-case $\mathcal{O}(1/K_g)$ stepsize scaling. In practice, multiplying $K_g$ by a certain factor does not actually require a corresponding proportional reduction in the stepsize to ensure convergence.
Moreover, employing fixed numbers of local updates and communication steps may be conservative because the appropriate computation-communication balance can vary with the optimization and consensus states. This motivates an adaptive method that adjusts local updates and broadcasts without requiring a prescribed system-cost model.

\subsection{Algorithm development}
Based on a quasi-Newton technique and double adaptive mechanism, a new Adaptive Decentralized Quasi-Newton (AdaDQN) method is proposed and its detailed procedure is given in Algorithm \ref{alg: Adpt. DNQN}.
The proposed algorithm adjusts the number of local updates and the broadcasts using locally available algorithmic states, aiming to avoid unnecessary computation and communication in practice while preserving the convergence guarantees in the presence of the resulting communication errors.

\begin{algorithm}[htbp]
	\caption{Adaptive Decentralized Quasi-Newton---AdaDQN}
	\label{alg: Adpt. DNQN}
	\begin{algorithmic}[1]
		\Require Initial point $\m{x}_i^0$, $i \in \mathcal{V}=\{1,\ldots,n\}$, Maximum iteration $T$, Stepsize $\alpha>0$, Mixing matrix $\m{W}$, Maximum computations $K_{g} \geq 1$, Maximum communications $K_c \geq 1$, Safeguard parameters $0<l\leq u$ and $\varrho>0$, Threshold sequences $\{\eta^t\}, \{\tau^s\}$, Parameters $\varepsilon \gg \tilde{\varepsilon} >0$, $\gamma>0$. 
		\State Initialize $t=0$, $\m{v}_i^{0}=\m{g}_i^{0}=\nabla f_i(\m{x}_i^0)$, $\m{H}_i^0=\m{I}$ for $i\in \mathcal{V}$.
		\While{$t < T$}
		\For{each node $i \in \mathcal{V}$ in parallel}
		\State $\m{x}_i^{t,0}=\m{x}_i^t$, $\m{v}_i^{t,0}=\m{v}_i^t$, $\m{H}_i^{t,0}=\m{H}_i^t$. 
		\For{$k = 0$ to $K_g-2$}
		\State $\Delta_1=\Vert\m{x}^t_i-\sum_{j \in \mathcal{N}_i}\tilde{W}_{ij}\m{x}^t_j\Vert$, $\Delta_2=\Vert \m{x}_i^{t,k} - \m{x}_i^t \Vert$.
		\If{$\alpha \Vert \m{H}_i^{t,k} \m{v}_i^{t,k} \Vert < \max\left\{\varepsilon \Delta_1, \tilde{\varepsilon}\right\}$  or $\Delta_2 > \gamma \alpha \Vert \m{v}_i^t \Vert$}
		\State $\m{x}_i^{t,k+1}= \m{x}_i^{t,k}$, $\m{v}_i^{t,k+1}=\m{v}_i^{t,k}$, $\m{H}_i^{t,k+1}=\m{H}_i^{t,k}$. 
		\Else
		\State $\m{x}_i^{t,k+1}= \m{x}_i^{t,k}-\alpha \m{H}_i^{t,k} \m{v}_i^{t,k}$. 
		\State $\m{v}_i^{t,k+1}=\m{v}_i^{t,k}+\m{g}_i^{t,k+1}-\m{g}_i^{t,k}$.
		\State $\Delta_x^{t,k}=\m{x}_i^{t,k+1}-\m{x}_i^{t,k}$, $\Delta_g^{t,k}=\m{g}_i^{t,k+1}-\m{g}_i^{t,k}$.
		\State Call Algorithm \ref{alg: MBFGS}: $\m{H}_i^{t,k+1} \leftarrow \text{MBFGS}(\Delta_x^{t,k}, \Delta_g^{t,k}, \Delta_g^{t,k}, l, u, \varrho)$.
		\EndIf
		\EndFor
		\State  $\m{x}_i^{t,K_g}=\m{x}_i^{t,K_g-1}-\alpha \m{H}_i^{t,K_g-1} \m{v}_i^{t,K_g-1}$. 
		\State Call Algorithm \ref{alg: sub_comm}: $\m{x}^{t+1}_i \leftarrow \text{EventTriggerComm}(\m{x}_i^{t,K_g}, K_c, \eta^t, \{\tau^s\}, \tilde{\m{W}})$.
		\State $\m{v}_i^{t,K_g}=\m{v}_i^{t,K_g-1}+\m{g}_i^{t+1}-\m{g}_i^{t,K_g-1}$. 
		\State Call Algorithm \ref{alg: sub_comm}: $\m{v}^{t+1}_i \leftarrow \text{EventTriggerComm}(\m{v}_i^{t,K_g}, K_c, \eta^t, \{\tau^s\}, \tilde{\m{W}})$.
		\State $\Delta_x^t=\m{x}_i^{t+1}-\m{x}_i^{t}$, $\Delta_v^t=\m{v}_i^{t+1}-\m{v}_i^{t}$, $\Delta_g^t=\m{g}_i^{t+1}-\m{g}_i^{t}$.
		\State Call Algorithm \ref{alg: MBFGS}: $\m{H}_i^{t+1} \leftarrow \text{MBFGS}(\Delta_x^t, \Delta_v^t, \Delta_g^t, l, u, \varrho)$.
		\EndFor
		\State $t \leftarrow t+1$.
		\EndWhile
		\Ensure $\{\m{x}_i^T\}_{i=1}^n$.
	\end{algorithmic}
\end{algorithm}

\begin{algorithm}[htbp]
	\caption{Memoryless BFGS Update---$\text{MBFGS}(\m{s}, \check{\m{y}}, \Delta\m{g}, l, u, \varrho)$}
	\label{alg: MBFGS}
	\begin{algorithmic}[1]
		\If{$\|\m{s}\|=0$}
		\State Set $\m{H}^+=\m{I}_p$ and terminate the update.
		\EndIf
		\If{$\m{s}\tr\check{\m{y}} > 0$}
		\State Calculate the extremal eigenvalues of the rank-two BFGS matrix generated by $\m{s}$ and $\check{\m{y}}$.
		\EndIf
		\If{$\m{s}\tr\check{\m{y}} > 0$ \textbf{and} $[\lambda_{\min}, \lambda_{\max}] \subset [l, u]$}
		\State $\m{y} = \check{\m{y}}$.
		\Else
		\State $h = \varrho + \max\left\{-\frac{\m{s}\tr\Delta\m{g}}{\|\m{s}\|^2}, 0\right\}$.
		\State $\m{y} = \Delta\m{g} + h\m{s}$.
		\EndIf
		\State Compute $\m{H}^+ = \frac{\m{s}\tr \m{y}}{\|\m{y}\|^2} \m{I}_p - \frac{\m{s}\m{y}\tr+\m{y}\m{s}\tr}{\|\m{y}\|^2} + 2 \frac{\m{s}\m{s}\tr}{\m{s}\tr\m{y}}$.
		\Ensure $\m{H}^+$.
	\end{algorithmic}
\end{algorithm}

\begin{algorithm}[htbp]
	\caption{Event-Triggered Broadcast---\text{EventTriggerComm}($\m{u}_{in}$, $K_c$, $\eta^t$, $\{\tau^s\}$, $\tilde{\m{W}}$)}
	\label{alg: sub_comm}
	\begin{algorithmic}[1]
		\State Initialize $\m{u}_i^{0} = \m{u}_{in}$, $\hat{\m{u}}_i^{0} = \m{u}_i^{0}$.
		\For{$s=0$ to $K_c-1$}
		\State $\m{u}_i^{s+1}=\m{u}_i^{s}- (\hat{\m{u}}_i^{s}-\sum_{j \in \mathcal{N}_i}\tilde{W}_{ij} \hat{\m{u}}_j^{s})$.
		\State $\hat{\m{u}}_i^{s+1}=\left\{\begin{array}{ll}
			\m{u}_i^{s+1}, & \text {if}~\Vert\m{u}_i^{s+1}-\hat{\m{u}}_i^{s}\Vert \geq \eta^t \tau^{s+1}; \\
			\hat{\m{u}}_i^{s}, & \text {otherwise}.
		\end{array}\right.$
		\EndFor
		\Ensure $\m{u}_i^{K_c}$.
	\end{algorithmic}
\end{algorithm}
\textbf{Memoryless BFGS update.}
Storing a dense inverse Hessian approximation requires $\mathcal{O}(p^2)$ memory. We therefore adopt the safeguarded memoryless BFGS update developed in \cite{wu2026unified} and summarize it in Algorithm \ref{alg: MBFGS}. The update uses only the latest curvature pair and has the following uniform bounds.

\begin{lemma}[Uniform bounds for MBFGS]\label{lem:mbfgs_bounds}
Suppose that Assumption \ref{as1} holds and $0<l\leq u$, $\varrho>0$. The matrix generated by Algorithm \ref{alg: MBFGS} satisfies
\begin{equation}\label{boundH}
	\psi\m{I}_p\preceq \m{H}^+\preceq\Psi\m{I}_p,\qquad
	\psi:=\min\left\{1,l,\frac{\varrho}{2(2L+\varrho)^2}\right\},\quad
	\Psi:=\max\left\{1,u,\frac{2}{\varrho}\right\}.
\end{equation}
\end{lemma}
\begin{proof}
If $\|\m{s}\|=0$, the returned identity matrix satisfies \eqref{boundH}. If the candidate pair passes the safeguard, the result follows directly from the eigenvalue test. Otherwise, the corrected vector satisfies
$\m{s}\tr\m{y}\geq\varrho\|\m{s}\|^2$ and
$\|\m{y}\|\leq(2L+\varrho)\|\m{s}\|$. The standard extremal-eigenvalue bounds for the rank-two memoryless BFGS update then give \eqref{boundH}; see \cite{wu2026unified} for the complete derivation.
\end{proof}

The rank-two representation in Algorithm \ref{alg: MBFGS} need not be formed explicitly. Its product with a vector uses only vector inner products and scaled vector additions, and hence requires $\mathcal{O}(p)$ memory and arithmetic.

\textbf{Safeguarded consensus-aware adaptive computation.}
To avoid the over-fitting (i.e., client drift) caused by fixed multiple local updates and to dynamically balance local gains against the consensus error, we introduce a hybrid dynamic termination mechanism for the local computation loop (Steps 5-12 in Algorithm \ref{alg: Adpt. DNQN}). 

At the $k$-th inner iteration of the outer iteration $t$, node $i$ evaluates a dual-criterion condition. It will terminate the inner loop early if either of the following criteria is met.
\begin{align}
	\alpha \Vert \m{H}_i^{t,k} \m{v}_i^{t,k} \Vert &< \max\left\{\varepsilon \left\Vert\m{x}^t_i - \sum\nolimits_{j \in \mathcal{N}_i}\tilde{W}_{ij}\m{x}^t_j\right\Vert, \tilde{\varepsilon}\right\}, \label{eq:comp_condition_1} \\
	\Vert \m{x}_i^{t,k} - \m{x}_i^t \Vert &\ge \gamma \alpha \Vert \m{v}_i^t \Vert. \label{eq:comp_condition_2}
\end{align}
These two criteria serve complementary purposes in practice and theory. On the one hand, the condition \eqref{eq:comp_condition_1} monitors the magnitude of the single quasi-Newton step relative to the distance between the node with its neighborhood average. If met, it indicates that the local stationary gap has become negligible compared to the consensus error. Continuing local update would only exacerbate the consensus error without yielding meaningful objective descent. Thus, terminating the loop effectively saves computational resources.
On the other hand, while \eqref{eq:comp_condition_1} restricts the step-wise increment, \eqref{eq:comp_condition_2} places a restriction on the cumulative drift $\Vert \m{x}_i^{t,k} - \m{x}_i^t \Vert$. It guarantees that the total deviation of the local iterates from the anchor point $\m{x}_i^t$ never exceeds a margin proportional to the initial tracker's norm $\Vert \m{v}_i^t \Vert$. As we will demonstrate in the subsequent analysis, this bound on cumulative drift $\Vert \m{x}_i^{t,k} - \m{x}_i^t \Vert$ is the linchpin that decouples the convergence stepsize $\alpha$ from the reciprocal of the maximum local steps $1/K_g$.


\textbf{Adaptive broadcasts.}
We employ the standard event-triggered mechanism presented in Algorithm \ref{alg: sub_comm}. Instead of transmitting variables at every inner communication step $s$, node $i$ broadcasts its state only when it has significantly evolved.
Let $\hat{\m{z}}_i^{t,s}$ denote the latest variable transmitted by node $i$ to its neighbors up to step $s$. The update rule for the consensus variable $\m{z}$ becomes
\begin{equation}\label{update-z}
	\m{z}_i^{t,s+1} = \m{z}_i^{t,s} - \sum_{j \in \mathcal{N}_i}\tilde{W}_{ij} \left(\hat{\m{z}}_i^{t,s} - \hat{\m{z}}_j^{t,s}\right).
\end{equation}
Node $i$ broadcasts its new state $\m{z}_i^{t,s+1}$ (updating $\hat{\m{z}}_i^{t,s+1} = \m{z}_i^{t,s+1}$) only if the innovation exceeds a time-varying threshold:
\begin{equation}
	\Vert\m{z}_i^{t,s+1}-\hat{\m{z}}_i^{t,s}\Vert \geq \eta^t\cdot \tau^{s+1},
\end{equation}
where $\{\eta^t\}$ and $\{\tau^s\}$ are non-increasing sequences governing the decay of the threshold. Otherwise, node $i$ remains silent (maintaining $\hat{\m{z}}_i^{t,s+1} = \hat{\m{z}}_i^{t,s}$). Thus, a broadcast can be avoided when the local state variation remains below the prescribed threshold. The same logic is applied to the gradient tracking variable $\m{q}$.
\begin{remark}[Relationship between existing event-triggered mechanisms] It is important to emphasize that we do not propose a novel event-triggered mechanism in this paper; instead, we adopt a well-established approach. Similar event-triggered mechanisms have been utilized in \cite{liu2019communication,li2019communication,george2020distributed,liu2021dqc,zhang2023decentralized}. Differently, \cite{liu2023event} employs a threshold based on the difference between two consecutive updates, whereas \cite{gao2021event} utilizes a threshold related to the discrepancy on iterates between a node and its neighbors.
\end{remark}

\subsection{Convergence analysis for AdaDQN}

Now turn our attention to the convergence properties of the AdaDQN algorithm. Throughout this subsection, we consider $0<\sigma<1$. 
To demonstrate that AdaDQN generates a sequence satisfying \eqref{eps-stationary}, we construct the following potential function:
\begin{equation}\label{potential+}
	P(\m{x}^{t},\m{v}^{t})=F(\bar{\m{x}}^{t})+\Vert\m{x}^{t}-\m{M}\m{x}^{t}\Vert^2+\frac{1-(1+\eta){\sigma}^{2K_c}}{144(1+1/\tilde{\eta})L^2}\Vert\m{v}^{t}-\m{M}\m{v}^{t}\Vert^2,
\end{equation}
where $\tilde{\eta}$ and $\eta$ are some positive constants introduced by Young's inequality. The following  lemma provides several key inequalities.

\begin{lemma}\label{lem:adaptive_bounds}
	Suppose that Assumptions \ref{as0}, \ref{as1}, and \ref{as_network} hold. Let $\{\m{x}^t\}$ be the sequence generated by the AdaDQN algorithm. If the stepsize satisfies $\alpha \leq \frac{1}{\Gamma L}$, then for all $t \geq 0$ and $0 \le k \le K_g$, the following inequalities hold.
	\begin{equation}\label{lem3.2 eq2}
		\Vert\m{x}_i^{t,k}-\m{x}_i^t\Vert \leq  \Gamma \alpha \Vert\m{v}_i^t\Vert,
	\end{equation}
	\begin{align}\label{lem3.2 eq2.5}
		\Vert\m{x}^{t,k}-\m{M}\m{x}^{t,k}\Vert^2\leq  2 \Gamma^2 \alpha^2 \Vert\m{v}^t\Vert^2+2\Vert\m{x}^{t}-\m{M}\m{x}^t\Vert^2,
	\end{align}
	where $\Gamma = \gamma + 4\Psi$ is a constant.
\end{lemma}

\begin{proof}
	\textbf{Proof of \eqref{lem3.2 eq2}:} We first identify the iterate from which the last mandatory quasi-Newton step in Algorithm \ref{alg: Adpt. DNQN} is taken. If node $i$ triggers either termination condition for the first time at an index $r_i^t\in\{0,\ldots,K_g-2\}$, then the iterates are frozen until the last mandatory update, namely,
	\begin{align*}
		\m{x}_i^{t,k}=\m{x}_i^{t,r_i^t},\quad
		\m{v}_i^{t,k}=\m{v}_i^{t,r_i^t},\quad
		\m{H}_i^{t,k}=\m{H}_i^{t,r_i^t},
		\qquad r_i^t\leq k\leq K_g-1.
	\end{align*}
	If neither termination condition is triggered, we set $r_i^t=K_g-1$. In both cases, the last mandatory update takes the unified form
	\begin{align}\label{eq:last_mandatory_step}
		\m{x}_i^{t,K_g}=\m{x}_i^{t,r_i^t}-\alpha\m{H}_i^{t,r_i^t}\m{v}_i^{t,r_i^t}.
	\end{align}
	For any accepted inner update before the index $r_i^t$, the preceding iterate did not trigger the safeguard condition \eqref{eq:comp_condition_2}, which guarantees
	\begin{align}\label{safeguard_bound}
		\Vert \m{x}_i^{t,k-1} - \m{x}_i^t \Vert < \gamma \alpha \Vert \m{v}_i^t \Vert.
	\end{align}
	By the recursion $\m{v}_i^{t,k-1} = \m{v}_i^t + \m{g}_i^{t,k-1} - \m{g}_i^t$ and the $L$-Lipschitz continuity of $\nabla f_i$, the tracking variable is bounded by
	\begin{align}\label{bound_v_tk}
		\Vert \m{v}_i^{t,k-1} \Vert \leq \Vert \m{v}_i^t \Vert + L \Vert \m{x}_i^{t,k-1} - \m{x}_i^t \Vert 
		< (1 + \Gamma L \alpha) \Vert \m{v}_i^t \Vert \leq 2 \Vert \m{v}_i^t \Vert,
	\end{align}
	where the last inequality applies the stepsize condition $\alpha \leq \frac{1}{\Gamma L}$. 
	According to the update rule $\m{x}_i^{t,k} = \m{x}_i^{t,k-1} - \alpha \m{H}_i^{t,k-1} \m{v}_i^{t,k-1}$ and the bounded approximation $\m{H}_i^{t,k-1} \preceq \Psi \m{I}$, it holds at step $k$ that
	\begin{align}\label{increment_bound}
		\Vert \m{x}_i^{t,k} - \m{x}_i^{t,k-1} \Vert \leq \alpha \Psi \Vert \m{v}_i^{t,k-1} \Vert < 2 \alpha \Psi \Vert \m{v}_i^t \Vert.
	\end{align}
	Applying the triangle inequality combining \eqref{safeguard_bound} and \eqref{increment_bound} shows that every iterate before the last mandatory update satisfies
	\begin{align*}
		\Vert \m{x}_i^{t,k} - \m{x}_i^t \Vert < (\gamma + 2\Psi) \alpha \Vert \m{v}_i^t \Vert.
	\end{align*}
	In particular, this bound holds for $\m{x}_i^{t,r_i^t}$. Using again the tracker recursion and $\alpha\leq 1/(\Gamma L)$ gives
	\begin{align*}
		\Vert\m{v}_i^{t,r_i^t}\Vert
		\leq \left(1+L(\gamma+2\Psi)\alpha\right)\Vert\m{v}_i^t\Vert
		\leq 2\Vert\m{v}_i^t\Vert.
	\end{align*}
	Combining this inequality with \eqref{eq:last_mandatory_step} yields
	\begin{align*}
		\Vert\m{x}_i^{t,K_g}-\m{x}_i^t\Vert
		\leq \Vert\m{x}_i^{t,r_i^t}-\m{x}_i^t\Vert
		+\alpha\Psi\Vert\m{v}_i^{t,r_i^t}\Vert
		<(\gamma+4\Psi)\alpha\Vert\m{v}_i^t\Vert
		=\Gamma\alpha\Vert\m{v}_i^t\Vert.
	\end{align*}
	The frozen iterates satisfy the same bound. Therefore, \eqref{lem3.2 eq2} holds for all $0\leq k\leq K_g$ regardless of whether a termination condition is triggered.
	
	\textbf{Proof of \eqref{lem3.2 eq2.5}:} Using the property $\Vert\m{I}-\m{M}\Vert \le 1$ and \eqref{lem3.2 eq2}, we derive
	\begin{align*}
		&\Vert\m{x}^{t,k}-\m{M}\m{x}^{t,k}\Vert^2 = \Vert(\m{I}-\m{M})(\m{x}^{t,k}-\m{x}^t) + (\m{x}^t-\m{M}\m{x}^t)\Vert^2\\
		\leq & 2\Vert\m{x}^{t,k}-\m{x}^t\Vert^2 + 2\Vert\m{x}^t-\m{M}\m{x}^t\Vert^2 
		\leq 2 \Gamma^2 \alpha^2 \Vert\m{v}^t\Vert^2 + 2\Vert\m{x}^t-\m{M}\m{x}^t\Vert^2,
	\end{align*}
which completes the proof.
\end{proof}

The following lemma quantifies both the descent and the amount of local computation generated by the adaptive loop. Let $q_i^t\in\{0,\ldots,K_g-1\}$ denote the number of nonmandatory local updates performed by node $i$ at iteration $t$, and let $\ell_i^t:=q_i^t+1$ include the last mandatory update. Define
\begin{equation}\label{eq:local_threshold}
	\chi_i^t:=\max\left\{\varepsilon\left\|\m{x}_i^t-\sum_{j\in\mathcal N_i}\tilde W_{ij}\m{x}_j^t\right\|,\tilde\varepsilon\right\},
	\qquad c_\alpha:=\alpha\psi-\frac{L}{2}\alpha^2\Psi^2.
\end{equation}

\begin{lemma}[Adaptive local-computation bound]\label{lem:adaptive_computation}
Suppose that the conditions of Lemma \ref{lem:adaptive_bounds} hold and
\begin{equation}\label{eq:descent_stepsize}
	\alpha\leq\min\left\{\frac{1}{2\Gamma L},\frac{\psi}{L\Psi^2}\right\}.
\end{equation}
Then
\begin{align}
	\phi_i^t(\m{x}_i^{t,K_g})
	&\leq \phi_i^t(\m{x}_i^t)
	-\frac{c_\alpha}{4}\ell_i^t\|\m{v}_i^t\|^2,\label{eq:adaptive_surrogate_descent}\\
	q_i^t
	&\leq \min\left\{K_g-1,\,
	C_{\rm loc}\frac{\alpha^2\|\m{v}_i^t\|^2}{(\chi_i^t)^2}\right\},\label{eq:adaptive_step_bound}
\end{align}
where
\begin{equation}\label{eq:C_loc}
	C_{\rm loc}:=
	\frac{\Psi^2\left(\Gamma+\frac{L}{2}\Gamma^2\alpha\right)}
	{\psi-\frac{L}{2}\alpha\Psi^2}.
\end{equation}
\end{lemma}
\begin{proof}
Every executed local step satisfies
\begin{equation}\label{eq:accepted_step_descent}
	\phi_i^t(\m{x}_i^{t,k+1})
	\leq\phi_i^t(\m{x}_i^{t,k})-c_\alpha\|\m{v}_i^{t,k}\|^2.
\end{equation}
Moreover, Lemma \ref{lem:adaptive_bounds} and the tracker recursion give
\[
	\|\m{v}_i^{t,k}\|
	\geq\|\m{v}_i^t\|-L\|\m{x}_i^{t,k}-\m{x}_i^t\|
	\geq\frac{1}{2}\|\m{v}_i^t\|.
\]
Summing \eqref{eq:accepted_step_descent} over the $\ell_i^t$ executed steps proves \eqref{eq:adaptive_surrogate_descent}.

For every executed nonmandatory step, condition \eqref{eq:comp_condition_1} is false. Hence,
\[
	\alpha\|\m{H}_i^{t,k}\m{v}_i^{t,k}\|\geq\chi_i^t,
	\qquad
	\|\m{v}_i^{t,k}\|\geq\frac{\chi_i^t}{\alpha\Psi}.
\]
Summing the resulting decrease over the $q_i^t$ nonmandatory steps yields
\[
	q_i^t\frac{c_\alpha(\chi_i^t)^2}{\alpha^2\Psi^2}
	\leq\phi_i^t(\m{x}_i^t)-\phi_i^t(\m{x}_i^{t,K_g}).
\]
Since $\nabla\phi_i^t(\m{x}_i^t)=\m{v}_i^t$, the lower smoothness inequality and Lemma \ref{lem:adaptive_bounds} imply
\[
	\phi_i^t(\m{x}_i^t)-\phi_i^t(\m{x}_i^{t,K_g})
	\leq\left(\Gamma\alpha+\frac{L}{2}\Gamma^2\alpha^2\right)\|\m{v}_i^t\|^2.
\]
Combining the last two inequalities proves \eqref{eq:adaptive_step_bound}.
\end{proof}

The event-triggered communication step preserves $\m{M}\m{v}^t=\m{M}\m{g}^t$ and $\bar{\m{x}}^{t+1}=\bar{\m{x}}^{t+1/2}$. Therefore, Lemma \ref{lem:descent_lemma} applies with
\[
	\gamma_1=\frac{c_\alpha}{4},\qquad
	\gamma_2=\Gamma^2\alpha^2.
\]
Retaining the additional descent generated by the $q_i^t$ nonmandatory steps gives
\begin{align}\label{eq:lemma1_descent}
	F(\bar{\m{x}}^{t+1})\leq&F(\bar{\m{x}}^t)
	-\left(\frac{\alpha\psi}{4n}-\frac{L\alpha^2\Psi^2}{8n}
	-\frac{L\Gamma^2\alpha^2(2+d+b)}{2n}\right)\|\m{v}^t\|^2\notag\\
	&-\frac{c_\alpha}{4n}\sum_{i=1}^nq_i^t\|\m{v}_i^t\|^2
	+\frac{L}{2bn}\|\m{x}^t-\m{M}\m{x}^t\|^2
	+\frac{1}{2Ldn}\|\m{v}^t-\m{M}\m{v}^t\|^2,
\end{align}
where $b,d>0$ are arbitrary. Thus, the objective-descent component of the RIA analysis is reused directly; the remaining analysis only controls the event-triggered consensus and tracking errors.

We define the event-triggered error at the $s$-th inner communication step as $\m{E}^{t,s}=\hat{\m{z}}^{t,s}-\m{z}^{t,s}$. Based on the event-triggered communication protocol, the recursion of $\m{z}^{t,s}$ can be expressed as
\begin{align*}
	\m{z}^{t,s}=\m{W}\m{z}^{t,s-1}+(\m{W}-\m{I})\m{E}^{t,s-1}=\m{W}^s\m{z}^{t,0}+\sum_{k=0}^{s-1}(\m{W}-\m{I})\m{W}^k \m{E}^{t,s-1-k}.
\end{align*}
Recalling the relationships $\m{x}^{t+1}=\m{z}^{t,K_c}$ and $\m{x}^{t,K_g}=\m{z}^{t,0}$, we obtain the updating formula for the iterate $\m{x}^t$:
\begin{equation}
	\m{x}^{t+1}=\m{W}^{K_c}\m{x}^{t,K_g}+\sum_{k=0}^{K_c-1}(\m{W}-\m{I})\m{W}^k \m{E}^{t,K_c-1-k}.
\end{equation}
By incorporating the accumulation of $K_g$ inner steps for $\m{x}^{t,K_g}$, we have
\begin{align}\label{x-mx}
	\m{x}^{t+1}-\m{M}\m{x}^{t+1}=&(\m{I}-\m{M})\m{W}^{K_c}\left(\m{x}^{t}+\m{x}^{t,K_g}-\m{x}^{t}\right)+\m{Z}^t,
\end{align}
where $\m{Z}^t$ aggregates the event-triggered errors, specifically,
$$
\m{Z}^t=(\m{I}-\m{M})\sum_{k=0}^{K_c-1}\m{W}^{k+1} \m{E}^{t,K_c-1-k}-(\m{I}-\m{M})\sum_{k=0}^{K_c-1}\m{W}^{k} \m{E}^{t,K_c-1-k}.
$$
According to the event-triggered mechanism, if $\Vert\m{z}_i^{t,s}-\hat{\m{z}}_i^{t,s-1}\Vert \geq \eta^t\cdot \tau^{s}$, node $i$ broadcasts its state, resulting in $\m{E}^{t,s}=\m{0}$. Otherwise, the node remains silent, implying $\hat{\m{z}}_i^{t,s}=\hat{\m{z}}_i^{t,s-1}$ and $\m{E}^{t,s}$ is bounded by the threshold. Consequently, in either case, the error satisfies
\begin{equation}\label{ets}
	\Vert\m{E}^{t,s}\Vert \leq \sqrt{n}\eta^t\cdot \tau^{s}.
\end{equation}
Analogously, applying a similar analysis to the gradient tracking variable $\m{v}$ and using \eqref{vtk}, we derive
\begin{align}\label{eq 2.65}
	&\m{v}^{t+1}-\m{M}\m{v}^{t+1}
	=(\m{I}-\m{M})\m{W}^{K_c}(\m{v}^{t,K_g-1}+\m{g}^{t+1}-\m{g}^{t,K_g-1})+\m{Q}^t\notag\\
	=&(\m{I}-\m{M})\m{W}^{K_c}(\m{v}^{t}+\m{g}^{t,K_g-1}-\m{g}^{t})
	+(\m{I}-\m{M})\m{W}^{K_c}(\m{g}^{t+1}-\m{g}^{t,K_g-1})+\m{Q}^t,
\end{align}
where $\m{Q}^t$ represents the accumulated noise for the tracking variable, that is
$$
\m{Q}^t=(\m{I}-\m{M})\sum_{k=0}^{K_c-1}\m{W}^{k+1} \tilde{\m{E}}^{t,K_c-1-k}-(\m{I}-\m{M})\sum_{k=0}^{K_c-1}\m{W}^{k} \tilde{\m{E}}^{t,K_c-1-k},
$$
with $\tilde{\m{E}}^{t,s}=\hat{\m{q}}^{t,s}-\m{q}^{t,s}$. Similar to \eqref{ets}, this error $\tilde{\m{E}}^{t,s}$ is bounded, i.e.,
\begin{equation}\label{hatets}
	\Vert\tilde{\m{E}}^{t,s}\Vert \leq \sqrt{n}\eta^t\cdot \tau^{s}.
\end{equation}

\begin{lemma}\label{lem 2.19}
	Suppose that Assumptions \ref{as0}, \ref{as1}, and \ref{as_network} hold. Let $\{\m{x}^t\}$ be the sequence generated by the AdaDQN algorithm. Then we have for all $t \geq 0$ that
	\begin{align}\label{lem3.4 eq0+}
\Vert\m{x}^{t+1}-\m{M}\m{x}^{t+1}\Vert^2 \leq & (1+\eta)\sigma^{2K_c}\Vert\m{x}^t-\m{M}\m{x}^t\Vert^2 + \theta_1(\eta^t)^2 \notag\\
& + 2(1+1/\eta)\alpha^2\sigma^{2K_c}\Gamma^2 \Vert\m{v}^t\Vert^2.
	\end{align}
	where $\theta_1=8(1+1/\eta){n}\left(\sum_{k=0}^{K_c-1}\sigma^{k} \tau^{ K_c-1-k}\right)^2$, and $\eta>0$ is a constant introduced by Young's inequality.
\end{lemma}

\begin{proof}
	Applying Young's inequality with parameter $\eta>0$ to \eqref{x-mx}, and invoking Lemma \ref{property W} and Lemma \ref{important}, we obtain
	\begin{align}\label{lem3.4_step1}
		\Vert\m{x}^{t+1}-\m{M}\m{x}^{t+1}\Vert^2 
		\leq & (1+\eta)\sigma^{2K_c}\Vert\m{x}^t-\m{M}\m{x}^t\Vert^2 + 2(1+1/\eta)\Vert\m{Z}^t\Vert^2 \notag\\
		&+ {2(1+1/\eta)\sigma^{2K_c} \left\Vert \m{x}^{t,K_g}-\m{x}^{t} \right\Vert^2}. 
	\end{align}
	Based on {Lemma \ref{property W}}, \eqref{ets}, and the fact that $0<\sigma<1$, the aggregated error $\m{Z}^t$ is bounded by
	\begin{align}\label{boundZ}
		\Vert\m{Z}^t\Vert\leq& \sum_{k=0}^{K_c-1}\Vert\m{W}^{k+1}-\m{M}\Vert \Vert\m{E}^{t,K_c-1-k}\Vert+\sum_{k=0}^{K_c-1}\Vert\m{W}^{k}-\m{M}\Vert \Vert\m{E}^{t,K_c-1-k}\Vert \notag\\
		\leq &2\sqrt{n}\eta^{t}\sum_{k=0}^{K_c-1}\sigma^{k} \tau^{ K_c-1-k}.
	\end{align}
	Substituting \eqref{lem3.2 eq2} and \eqref{boundZ} into \eqref{lem3.4_step1}, we get the desired result \eqref{lem3.4 eq0+}.
\end{proof}

\begin{lemma}\label{lem:adaptive_tracking_error}
	Suppose that Assumptions \ref{as0}, \ref{as1}, and \ref{as_network} hold. Let $\{\m{x}^t\}$ be the sequence generated by the AdaDQN algorithm. If the stepsize satisfies $\alpha \leq \frac{1}{\Gamma L}$, we have for all $t \geq 0$ that
	\begin{align}\label{lem3.5 eq0+}
		&\frac{1-(1+\eta){\sigma}^{2K_c}}{144(1+1/\tilde{\eta})L^2}\Vert\m{v}^{t+1}-\m{M}\m{v}^{t+1}\Vert^2
		\leq \frac{1-(1+\eta){\sigma}^{2K_c}}{144(1+1/\tilde{\eta})L^2} \Vert\m{v}^{t}-\m{M}\m{v}^{t}\Vert^2\notag\\
		&-\frac{(1-(1+\eta){\sigma}^{2K_c})(1-(1+\tilde{\eta}){\sigma}^{2K_c})}{144(1+1/\tilde{\eta})L^2}\Vert\m{v}^{t}-\m{M}\m{v}^{t}\Vert^2\notag\\
		&+\frac{1}{2}\sigma^{2K_c}(1-(1+\eta){\sigma}^{2K_c})\Vert{\m{x}}^t-\m{M}{\m{x}}^t\Vert^2\notag\\
		&+\frac{7}{12}(1-(1+\eta){\sigma}^{2K_c})\alpha^2\sigma^{2K_c}\Gamma^2 \Vert\m{v}^t\Vert^2+ \theta_2(\eta^{t})^2,
	\end{align}
	where $\theta_2=12(1+1/\tilde{\eta}){n}\left(\sum_{k=0}^{K_c-1}\sigma^{k} \tau^{ K_c-1-k}\right)^2(1+3\sigma^{2K_c}L^2) \frac{1-(1+\eta){\sigma}^{2K_c}}{144(1+1/\tilde{\eta})L^2}$ and  $\eta>0$, $\tilde{\eta} >0$ are constants introduced by Young's inequality.
\end{lemma}

\begin{proof}
	Applying Young's inequality with parameter $\tilde{\eta}>0$ to \eqref{eq 2.65}, and invoking Lemma \ref{property W}, Lemma \ref{important} and the $L$-Lipschitz continuity of $\nabla f$, we obtain
	\begin{align}\label{eq 2.69}
		&\Vert\m{v}^{t+1}-\m{M}\m{v}^{t+1}\Vert^2 \notag\\
		\leq &(1+\tilde{\eta})\sigma^{2K_c}\Vert\m{v}^{t}-\m{M}\m{v}^{t}\Vert^2+3(1+1/\tilde{\eta})\sigma^{2K_c}L^2\Vert\m{x}^{t,K_g-1}-\m{x}^{t}\Vert^2\notag\\
		&+3(1+1/\tilde{\eta})\sigma^{2K_c}L^2\Vert\m{x}^{t+1}-\m{x}^{t,K_g-1}\Vert^2+3(1+1/\tilde{\eta})\Vert\m{Q}^t\Vert^2.
	\end{align}
	Recall that $\m{x}^{t+1} = \m{W}^{K_c}(\m{x}^{t,K_g-1}-\alpha\m{H}^{t,K_g-1}\m{v}^{t,K_g-1})+\m{Z}^t$. Using Lemma \ref{important}, \eqref{boundH}, and the spectral property $\Vert\m{W}^{K_c}-\m{I}\Vert^2 \le 4$, we bound the term $\Vert\m{x}^{t+1}-\m{x}^{t,K_g-1}\Vert^2$ by
	\begin{align*}
		&\Vert\m{x}^{t+1}-\m{x}^{t,K_g-1}\Vert^2
		= \Vert\m{W}^{K_c}(\m{x}^{t,K_g-1}-\alpha\m{H}^{t,K_g-1}\m{v}^{t,K_g-1})-\m{x}^{t,K_g-1}+\m{Z}^t\Vert^2\\
		\leq &3\Vert(\m{W}^{K_c}-\m{I})(\m{x}^{t,K_g-1}-\m{M}\m{x}^{t,K_g-1})\Vert^2+3\alpha^2\Psi^2\Vert\m{v}^{t,K_g-1}\Vert^2+3\Vert\m{Z}^t\Vert^2\\
		\leq &12\Vert\m{x}^{t,K_g-1}-\m{M}\m{x}^{t,K_g-1}\Vert^2+12\alpha^2\Psi^2\Vert\m{v}^t\Vert^2+3\Vert\m{Z}^t\Vert^2,
	\end{align*}
	where the last inequality applies the bound $\Vert \m{v}_i^{t,k-1} \Vert \leq 2 \Vert \m{v}_i^t \Vert$ established in \eqref{bound_v_tk}.
	Substituting \eqref{lem3.2 eq2.5} into the above inequality yields
	\begin{align}\label{eq:disp_bound}
		\Vert\m{x}^{t+1}-\m{x}^{t,K_g-1}\Vert^2 &\leq 12\left( 2\Gamma^2 \alpha^2 \Vert\m{v}^t\Vert^2 + 2\Vert\m{x}^t-\m{M}\m{x}^t\Vert^2 \right) + 12\alpha^2\Psi^2\Vert\m{v}^t\Vert^2+3\Vert\m{Z}^t\Vert^2 \notag\\
		&= 24\Vert\m{x}^t-\m{M}\m{x}^t\Vert^2 + (24\Gamma^2 + 12\Psi^2)\alpha^2\Vert\m{v}^t\Vert^2 + 3\Vert\m{Z}^t\Vert^2\notag\\
		& \leq 24\Vert\m{x}^t-\m{M}\m{x}^t\Vert^2 + 27\Gamma^2\alpha^2\Vert\m{v}^t\Vert^2 + 3\Vert\m{Z}^t\Vert^2,
	\end{align}
where the inequality uses the relation that $12\Psi^2 \leq 3\Gamma^2$.
	Substituting \eqref{eq:disp_bound} and \eqref{lem3.2 eq2} into \eqref{eq 2.69} yields
	\begin{align}\label{vMv}
		&\Vert\m{v}^{t+1}-\m{M}\m{v}^{t+1}\Vert^2
		\leq (1+\tilde{\eta})\sigma^{2K_c}\Vert\m{v}^{t}-\m{M}\m{v}^{t}\Vert^2 + 72(1+1/\tilde{\eta})\sigma^{2K_c}L^2\Vert\m{x}^t-\m{M}\m{x}^t\Vert^2 \notag\\
		&+ 84(1+1/\tilde{\eta})\sigma^{2K_c} L^2  \Gamma^2  \alpha^2 \Vert\m{v}^t\Vert^2 + \text{NoiseTerm},
	\end{align}
	where $\text{NoiseTerm} = 3(1+1/\tilde{\eta})(\Vert\m{Q}^t\Vert^2+3\sigma^{2K_c}L^2\Vert\m{Z}^t\Vert^2)$. 
	The bound for $\text{NoiseTerm}$ follows the identical logic as in Lemma \ref{lem 2.19}, bounded by
	\begin{equation*}
		\text{NoiseTerm} \leq 12(1+1/\tilde{\eta}){n}\left(\sum_{k=0}^{K_c-1}\sigma^{k} \tau^{ K_c-1-k}\right)^2(1+3\sigma^{2K_c}L^2)(\eta^t)^2.
	\end{equation*}
	Finally, multiplying both sides of \eqref{vMv} by $\frac{1-(1+\eta){\sigma}^{2K_c}}{144(1+1/\tilde{\eta})L^2}$ yields the stated inequality \eqref{eq 2.69}.
\end{proof}

\begin{theorem}\label{thm:global_convergence_ada}
	Suppose that Assumptions \ref{as0}, \ref{as1}, and \ref{as_network} hold and the event-triggered threshold sequence satisfies $\sum_{t=0}^{\infty}(\eta^t)^2<\infty$. Let $\{\m{x}^t\}$ be the sequence generated by the AdaDQN algorithm. If the stepsize $\alpha$ satisfies 
	\begin{equation}\label{eq:theorem_stepsize_new}
		\alpha \leq \min \left\{ \frac{1}{2\Gamma L},\frac{\psi}{L\Psi^2}, \frac{\psi}{16 \Gamma^2 \mathcal{C}_{net}} \right\},
	\end{equation}
	where $\Gamma=\gamma+4\Psi$ and the network-related constant $\mathcal{C}_{net}$ is defined as
	\begin{align}\label{eq:c_net_def}
		\mathcal{C}_{net} =& L + \frac{2L^2}{n(1-\sigma^{2K_c})} + \frac{288(1+\sigma^{2K_c}) L^2}{n(1-\sigma^{2K_c})^3} \notag\\
		&+ 2n\sigma^{2K_c} \left(\frac{1+\sigma^{2K_c}}{1-\sigma^{2K_c}}\right) + \frac{7}{24} n\sigma^{2K_c} (1-\sigma^{2K_c}),
	\end{align}
	then the algorithm achieves the following convergence rate.
	\begin{align}\label{eq:theorem_rate_new}
		\min_{0 \le t \le T-1}\left\{ \Vert{\m{v}}^t\Vert^2+\Vert{\m{v}}^t-\m{M}{\m{v}}^t\Vert^2+\Vert\m{x}^{t}-\m{M}\m{x}^{t}\Vert^2 \right\} \leq \frac{P(\m{x}^{0},\m{v}^{0})-\underline{F} + \Theta_T}{\widehat{C} T},
	\end{align}
	where $\widehat{C}>0$ is explicitly given in the proof and
	\begin{align}\label{Theta}
		\Theta_T &= (\theta_1+\theta_2)\sum_{t=0}^{T-1} (\eta^t)^2~\text{with}\\
		\theta_1 &= 8\left(\frac{1+\sigma^{2K_c}}{1-\sigma^{2K_c}}\right)n\left(\sum_{k=0}^{K_c-1}\sigma^{k} \tau^{ K_c-1-k}\right)^2, \notag\\
		\theta_2 &= \frac{(1-\sigma^{2K_c})(1+3\sigma^{2K_c}L^2)}{24 L^2} n \left(\sum_{k=0}^{K_c-1}\sigma^{k} \tau^{ K_c-1-k}\right)^2.\notag
	\end{align}
\end{theorem}

\begin{proof}
	Combining \eqref{eq:lemma1_descent}, Lemma \ref{lem 2.19}, and Lemma \ref{lem:adaptive_tracking_error}, we obtain
	\begin{align}\label{eq:P_recursion_final}
		P(\m{x}^{t+1},\m{v}^{t+1}) \leq& P(\m{x}^{t},\m{v}^{t}) - C_f \Vert\m{v}^t\Vert^2 - C_{x} \Vert{\m{x}}^t-\m{M}{\m{x}}^t\Vert^2 \notag\\
		&- C_{v} \Vert\m{v}^t-\m{M}\m{v}^t\Vert^2 -\frac{c_\alpha}{4n}\sum_{i=1}^n q_i^t\|\m{v}_i^t\|^2 \notag\\
		&+ (\theta_1+\theta_2)(\eta^t)^2,
	\end{align}
	where the coefficients $C_x, C_v$, and $C_f$ are given by
	\begin{align}
		C_x &= 1 - (1+\eta)\sigma^{2K_c} - \frac{L}{2bn} - \frac{1}{2}\sigma^{2K_c}(1-(1+\eta){\sigma}^{2K_c}), \label{eq:Cx_raw} \\
		C_v &= \frac{(1-(1+\eta){\sigma}^{2K_c})(1-(1+\tilde{\eta}){\sigma}^{2K_c})}{144(1+1/\tilde{\eta})L^2} - \frac{1}{2Ldn}, \label{eq:Cv_raw} \\
		C_f &= \frac{\alpha \psi - \frac{L}{2}\alpha^2 \Psi^2}{4n} - \frac{L\Gamma^2\alpha^2(2+d+b)}{2n} - 2(1+1/\eta)\sigma^{2K_c}\Gamma^2\alpha^2 \notag\\
		& \quad - \frac{7}{12}(1-(1+\eta){\sigma}^{2K_c})\alpha^2\sigma^{2K_c}\Gamma^2. \label{eq:Cf_raw}
	\end{align}
	Now we simplify these coefficients by strategically assigning the parameters ($b, d, \eta, \tilde{\eta}$).
	Selecting $\eta = \tilde{\eta} = \frac{1-\sigma^{2K_c}}{2\sigma^{2K_c}}$ enforces $1-(1+\eta)\sigma^{2K_c} = \frac{1-\sigma^{2K_c}}{2}$ and $1+1/\eta = \frac{1+\sigma^{2K_c}}{1-\sigma^{2K_c}}$.
	Substituting $b = \frac{4L}{n(1-\sigma^{2K_c})}$ into \eqref{eq:Cx_raw} yields
	$$
	C_x = \frac{1-\sigma^{2K_c}}{2} - \frac{1-\sigma^{2K_c}}{8} - \frac{\sigma^{2K_c}(1-\sigma^{2K_c})}{4} = \frac{1-\sigma^{2K_c}}{8} (3 - 2\sigma^{2K_c}) > 0.
	$$
	Substituting $d = \frac{576(1+\sigma^{2K_c}) L}{n (1-\sigma^{2K_c})^3}$ into \eqref{eq:Cv_raw} yields
	$$
	C_v = \frac{(1-\sigma^{2K_c})^3}{576(1+\sigma^{2K_c})L^2} - \frac{(1-\sigma^{2K_c})^3}{1152(1+\sigma^{2K_c})L^2} = \frac{(1-\sigma^{2K_c})^3}{1152(1+\sigma^{2K_c})L^2} > 0.
	$$
	Embedding the defined parameters ($\eta, b, d$) into \eqref{eq:Cf_raw}, the coefficient $C_f$ evaluates to
	$$
	C_f = \frac{\alpha}{n} \left[ \frac{\psi}{4} - \alpha \left( \frac{L\Psi^2}{8} + \Gamma^2 \mathcal{C}_{net} \right) \right],
	$$
	where $\mathcal{C}_{net}$ is given by \eqref{eq:c_net_def}. By enforcing the stepsize condition \eqref{eq:theorem_stepsize_new}, we derive
	$$
	\alpha\frac{L\Psi^2}{8}\leq\frac{\psi}{8},\qquad
	\alpha\Gamma^2\mathcal{C}_{net}\leq\frac{\psi}{16},
	$$
	which implies
	$$
	C_f \geq \frac{\alpha}{n} \left[ \frac{\psi}{4} - \frac{\psi}{8} - \frac{\psi}{16} \right] = \frac{\alpha \psi}{16n} > 0.
	$$
	Finally, summing \eqref{eq:P_recursion_final} from $t=0$ to $t=T-1$ and using the lower boundedness of the potential function yields the stated result \eqref{eq:theorem_rate_new} where $\widehat{C}:=\min\left\{C_f,C_x,C_v\right\}>0$.
\end{proof}

\begin{corollary}[Gradient complexity of AdaDQN]\label{cor:ada_oracle}
Let
\[
	D_T:=P(\m{x}^0,\m{v}^0)-\underline F+\Theta_T
\]
and count one local gradient evaluation for every executed nonmandatory step and one mandatory gradient evaluation per node and outer iteration. Then
\begin{equation}\label{eq:ada_oracle_T}
	\mathcal{G}_{\rm Ada}(T)
	:=n+nT+\sum_{t=0}^{T-1}\sum_{i=1}^nq_i^t
	\leq n+nT+
	\frac{C_{\rm loc}\alpha^2D_T}{C_f\tilde\varepsilon^2}.
\end{equation}
In particular, letting
\[
	D_\infty:=P(\m{x}^0,\m{v}^0)-\underline F
	+(\theta_1+\theta_2)\sum_{t=0}^\infty(\eta^t)^2,
\]
AdaDQN returns an iterate with $\mathcal{R}^t\leq\delta$ using
\begin{equation}\label{eq:ada_oracle_delta}
	\mathcal{G}_{\rm Ada}(\delta)
	\leq n+n\left\lceil\frac{D_\infty}{\widehat C\delta}\right\rceil
	+\frac{C_{\rm loc}\alpha^2D_\infty}{C_f\tilde\varepsilon^2}.
\end{equation}
Since $C_f\geq\alpha\psi/(16n)$, this gives
$\mathcal{G}_{\rm Ada}(\delta)=\mathcal{O}(n/\delta+n\alpha\tilde\varepsilon^{-2})$, with no dependence on the maximum local-update budget $K_g$.
\end{corollary}
\begin{proof}
Since $\chi_i^t\geq\tilde\varepsilon$, \eqref{eq:adaptive_step_bound} gives
\[
	\sum_{t=0}^{T-1}\sum_{i=1}^nq_i^t
	\leq\frac{C_{\rm loc}\alpha^2}{\tilde\varepsilon^2}
	\sum_{t=0}^{T-1}\|\m{v}^t\|^2.
\]
Summing \eqref{eq:P_recursion_final} yields
$\sum_{t=0}^{T-1}\|\m{v}^t\|^2\leq D_T/C_f$, which proves \eqref{eq:ada_oracle_T}. The rate \eqref{eq:theorem_rate_new} then gives \eqref{eq:ada_oracle_delta}.
\end{proof}

\section{Numerical experiments}
In this section, we would like to examine the performance of our developed algorithms that contain Algorithm \ref{alg:Framwork1}, and AdaDQN ({Algorithm \ref{alg: Adpt. DNQN}}), as the following outline:
\begin{itemize}
\item[1.] Investigate the difference in computation and communication efficiency between {Algorithm \ref{alg:Framwork1}} and AdaDQN.
\item[2.] Compare AdaDQN with other state-of-the-art algorithms.
\end{itemize}
The considered optimization problem is smooth but nonconvex, over a connected undirected network with edge density $d \in (0,1]$.
For the generated network, we choose the Metropolis constant edge weight matrix \cite{xiao2007distributed} as the mixing matrix, that is
\begin{equation*}
	\tilde{W}_{i j}=\left\{\begin{array}{cl}
		\frac{1}{\max \{\operatorname{deg}(i), \operatorname{deg}(j)\}+1}, & \text { if }(i, j) \in \mathcal{E}, \\
		0, & \text { if }(i, j) \notin \mathcal{E} \text { and } i \neq j, \\
		1-\sum_{k \in \mathcal{N}_i/ \{i\}} \tilde{W}_{i k}, & \text { if } i=j,
	\end{array}\right.
\end{equation*}
where $(i, j) \in \mathcal{E}$ indicates there is an edge between 
node $i$ and node $j$, and $\operatorname{deg}(i)$ means the degree of node $i$. 
In our experiments, we introduce the communication volume which can be calculated as follows: 
\begin{align*}\notag
	&\text{Communication volume}
	=\text{~number of iterations}\\
	&\times \text{number of communication rounds per iteration}\\
	&\times \text{number of edges, i.e., }dn(n-1)/2\\
	&\times \text{dimension of transmitted vectors on each edge}.
\end{align*}
We set the number of nodes $n=10$ and the edge density $d=5/9$ for the network, and set $\eta^t=\kappa \rho^{t}$ and $ \tau^{s}=\rho^s$ for the AdaDQN algorithm, where $\kappa>0$ and $\rho \in (0,1)$. For all comparison algorithms, we initialize $\m{x}^0=\m{0}$. From the first-order stationarity given in \eqref{stationarity}, the success of each algorithm is measured by the optimality error stated as 
\begin{equation*}
	\left\Vert \frac{1}{n}\sum_{i=1}^n\nabla f_i(\m{x}^t_{i}) \right\Vert+\Vert\m{x}^t-\m{M}\m{x}^t\Vert.
\end{equation*}
All experiments are coded in MATLAB R2017b and run on a laptop with Intel Core i5-9300H CPU, 16GB RAM, and Windows 10 operating system.
\subsection{Investigation on computation and communication efficiency}
We consider the nonconvex decentralized binary classification problem. Using a logistic regression formulation with a nonconvex regularization, the optimization problem is given by
\begin{equation}\label{noncovex_logistic_problem}
	\mathop {\min }\limits_{\m{z} \in {\mathbb{R}^p}} \sum_{i=1}^n  \sum_{j=1}^{n_i} \log \left(1+\exp (-b_{ij} \m{a}_{ij}\tr \m{z} ) \right)+\hat{\lambda} \sum_{k=1}^p \frac{\m{z}_{[k]}^2}{1+\m{z}_{[k]}^2},
\end{equation}
where $\m{a}_{ij} \in \mathbb{R}^p$ is the feature vector, $ b_{ij} \in \{-1,+1\}$ represents the label, $\m{z}_{[k]}$ denotes the $k$-th component of the vector $\m{z}$, and $\hat{\lambda}>0$ is the regularization parameter.
The experiments are conducted on four datasets in Table~\ref{table2}
from the LIBSVM library: \textbf{mushroom}, \textbf{ijcnn1}, \textbf{w8a} and \textbf{a9a}.
The regularization parameter $\hat{\lambda}=1$. 

\begin{table}[H]
	\caption{Datasets}\label{table2}
	\centering
	\begin{tabular}{ccc}
		\hline
		{Dataset}&\# of samples ($\sum_{i=1}^n n_i$)  &\# of features ($p$)  \\
		\hline
		\textbf{mushroom}&8120 &112\\

		\textbf{ijcnn1}&49990&22\\

		\textbf{w8a}&49740&300\\

		\textbf{a9a}&32560&123\\
		\hline
	\end{tabular}
\end{table}

We validate the effectiveness of the proposed double-level adaptive mechanism by ablation experiments that compare AdaDQN with the following variants: 
\begin{itemize}
	\item Fixed-DQN ($K_g,K_c$): the variant with fixed number of computation ($K_g$) and communication ($K_c$), i.e., Algorithm \ref{alg:Framwork1}. Consider Fixed-DQN (1,1) and Fixed-DQN (5,5).
	\item AdaComp-DQN: the variant that only uses adaptive local updates and a fixed number of communication rounds.
	\item AdaComm-DQN: the variant that only uses adaptive broadcasts and a fixed number of local updates.
\end{itemize}
For datasets \textbf{mushroom}(\textbf{ijcnn1};\textbf{w8a};\textbf{a9a}), algorithm parameters are set as follows their better performance and parameter notations follow the source papers. We set
\begin{itemize}
	\item $\alpha=0.22(0.32;0.26;0.34)$ and $\varrho=0.05$ in Fixed-DQN (1,1);
	\item $\alpha=0.82(0.86;0.3;0.88)$ and $\varrho=0.05$ in Fixed-DQN (5,5);
	\item $K_g=5$, $K_c=5$, $\alpha=0.74(0.6;0.44;0.69)$, $\varepsilon=1$, and $\varrho=0.05$ in AdaComp-DQN; 
	\item $K_g=5$, $K_c=5$, $\alpha=0.58(0.4;0.4;0.72)$, $\kappa=1$, $\rho=0.25(0.25;0.25;0.2)$, and $\varrho=0.05$ in AdaComm-DQN; 
	\item  $K_g=5$, $K_c=5$, $\alpha=0.89(0.88;0.89;0.88)$, $\varepsilon=1$, $\kappa=1$, $\rho=0.25(0.25;0.25;0.2)$, $\varrho=0.05$ in AdaDQN.
\end{itemize}
Figs. \ref{adap_iter} to \ref{adap_time} present optimality error against iteration number, communication volume, gradient evaluation number, and CPU time, respectively.
As expected, Fixed-DQN (5,5), AdaComp-DQN, AdaComm-DQN, and AdaDQN converge significantly faster in terms of iteration number compared to Fixed-DQN (1,1). 
However, Fixed-DQN (5,5) employs fixed multiple steps of computation and communication, leading to high cost of gradient evaluations and communication volume. 
The ablation study shows that AdaComm-DQN saves communication volume(e.g. steepest in Fig. \ref{m_c}) but incurs excessive computation overhead (e.g. slowest in Fig. \ref{i_g}) while  AdaComp-DQN saves computation cost(e.g. steepest in Fig. \ref{m_g}) but suffers from high communication volume (e.g. slowest in Fig. \ref{i_c}). Adaptive DQN has significant superiority in both communication and computation (e.g. steepest in Figs. \ref{i_c} and \ref{i_g}).

To explicitly visualize the computation-communication trade-off, we plot the total gradient evaluations (Y-axis) against total communication volume (X-axis) required to reach an accuracy of $\epsilon=10^{-10}$ in Fig. \ref{adap_scatter}. The AdaDQN (black star) is located at the bottom-left corner relative to the other algorithms. This clearly demonstrates that AdaDQN achieves a better computation-communication trade-off, minimizing both resources simultaneously rather than sacrificing one for the other.

\begin{table}
	\small
		\caption{Iteration number / Communication volume / Gradient evaluation number and CPU time (s) of different algorithms to achieve $\epsilon=10^{-10}$ accuracy on datasets \textbf{mushroom}, \textbf{ijcnn1}, \textbf{w8a}, \textbf{a9a}.}
	\centering
\begin{tabular}{ccccc}
	\hline
	& mushroom & ijcnn1 & w8a & a9a \\
	\hline
	Fixed-DQN (1,1) & 
	\makecell{92/1236480 \\ 920/0.3497} & 
	\makecell{54/142560 \\ 540/0.6992} & 
	\makecell{72/2592000 \\ 720/1.0023} & 
	\makecell{52/767520 \\ 520/0.5389} \\
\hline
	Fixed-DQN (5,5) & 
	\makecell{14/940800 \\ 700/0.2584} & 
	\makecell{13/171600 \\ 650/0.7380} & 
	\makecell{13/2340000 \\ 650/0.8368} & 
	\makecell{12/885600 \\ 600/0.5365} \\
\hline
	AdaComp-DQN & 
	\makecell{14/940800 \\ 480/0.1862} & 
	\makecell{13/171600 \\ 370/0.4573} & 
	\makecell{14/2520000 \\ 500/0.6560} & 
	\makecell{12/885600 \\ 330/0.3347} \\
\hline
	AdaComm-DQN & 
	\makecell{14/624960 \\ 700/0.2471} & 
	\makecell{14/102564 \\ 700/0.8250} & 
	\makecell{14/1443600 \\ 700/0.8918} & 
	\makecell{11/484374 \\ 550/0.5124} \\
\hline
	AdaDQN & 
	\makecell{14/636496 \\ 520/0.2013} & 
	\makecell{14/101992 \\ 300/0.3963} & 
	\makecell{14/1446600 \\ 340/0.4697} & 
	\makecell{11/482406 \\ 240/0.2484} \\
	\hline
\end{tabular}
\end{table}

\begin{figure*}[!t]
	\centering
	\subfloat[mushroom]{\includegraphics[width=1.3in]{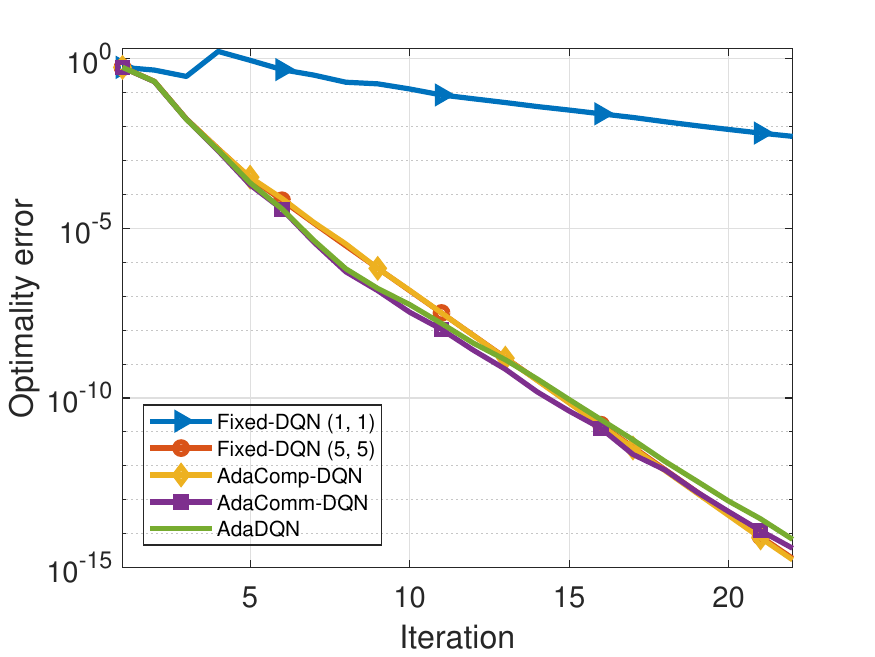}%
		\label{m_i}}
	\subfloat[ijcnn1]{\includegraphics[width=1.3in]{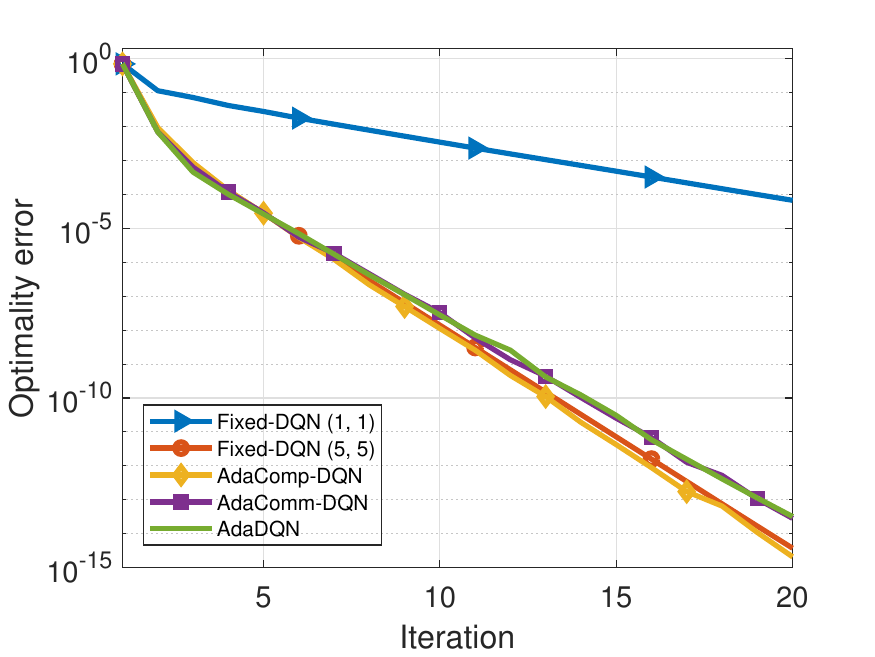}%
		\label{i_i}}
	\subfloat[w8a]{\includegraphics[width=1.3in]{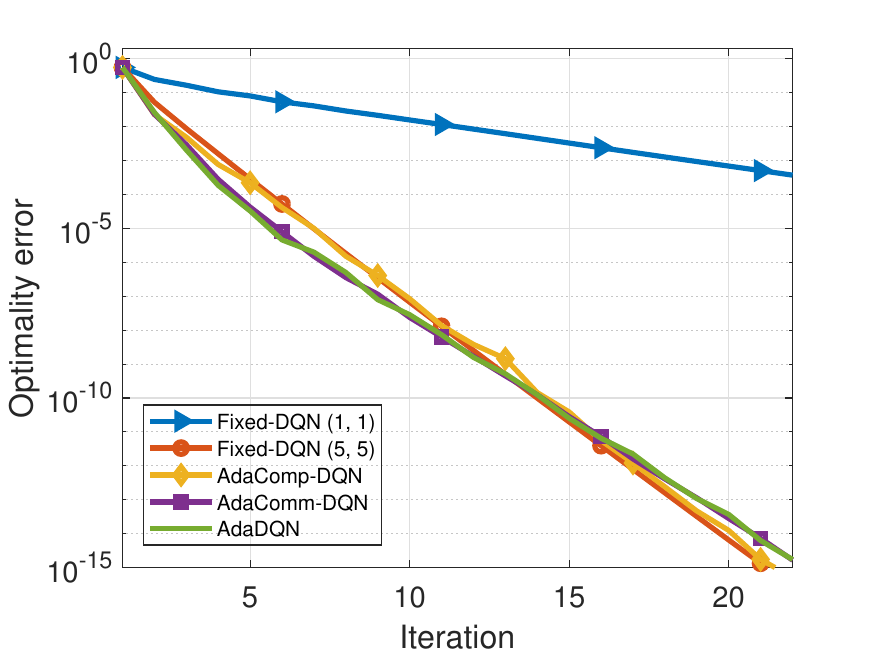}%
		\label{w_i}}
	\subfloat[a9a]{\includegraphics[width=1.3in]{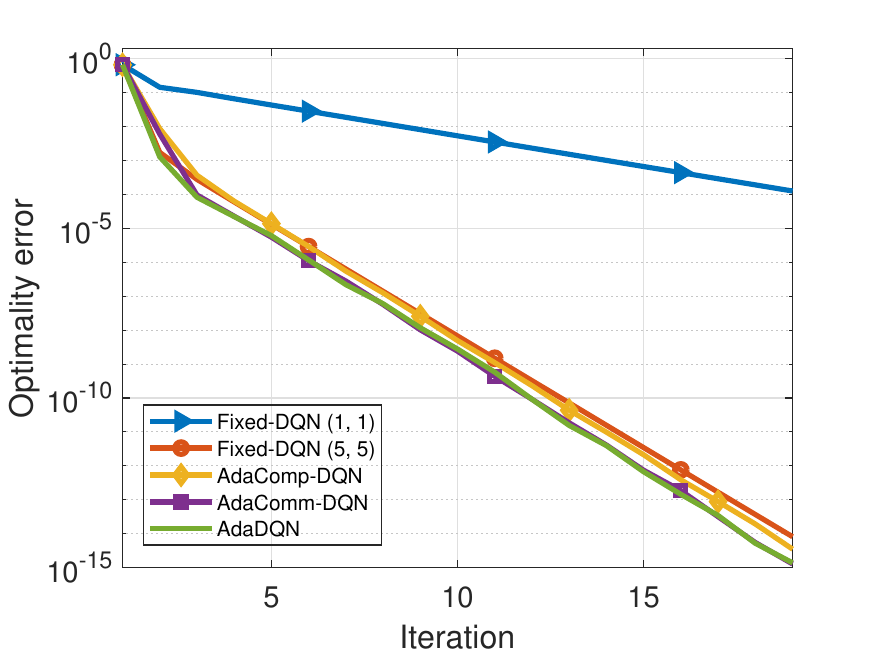}%
		\label{a_i}}
	\caption{Optimality error of comparison algorithms for minimizing the nonconvex logistic regression problem \eqref{noncovex_logistic_problem} on different datasets w.r.t. iteration.}
	\label{adap_iter}
\end{figure*}

\begin{figure*}[!t]
	\centering
	\subfloat[mushroom]{\includegraphics[width=1.3in]{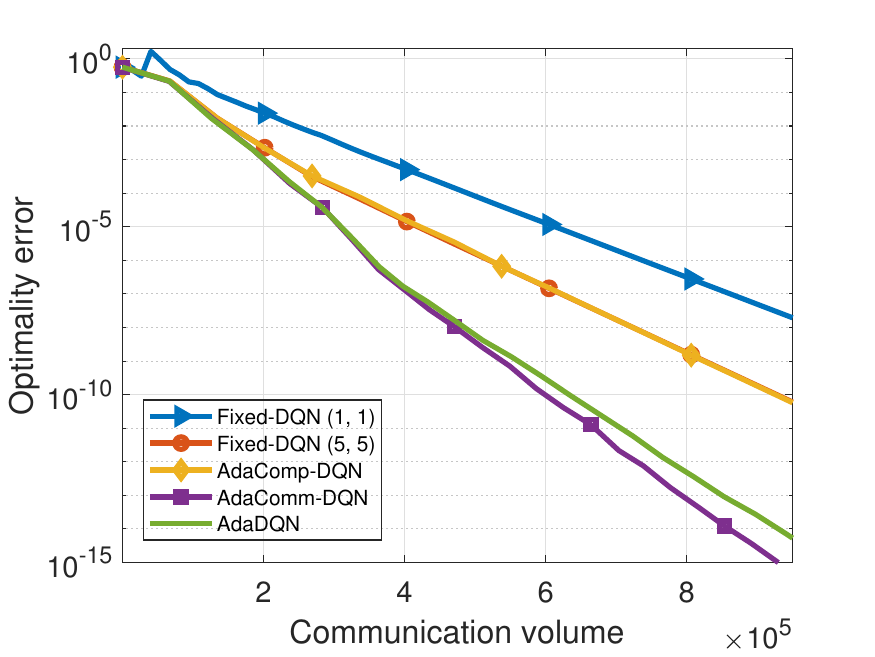}%
		\label{m_c}}
	\subfloat[ijcnn1]{\includegraphics[width=1.3in]{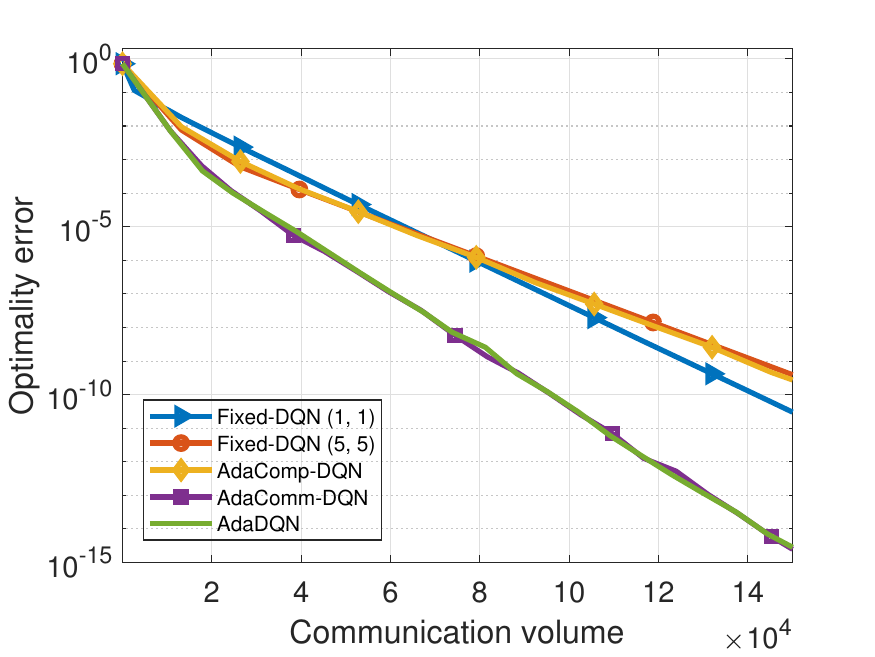}%
		\label{i_c}}
	\subfloat[w8a]{\includegraphics[width=1.3in]{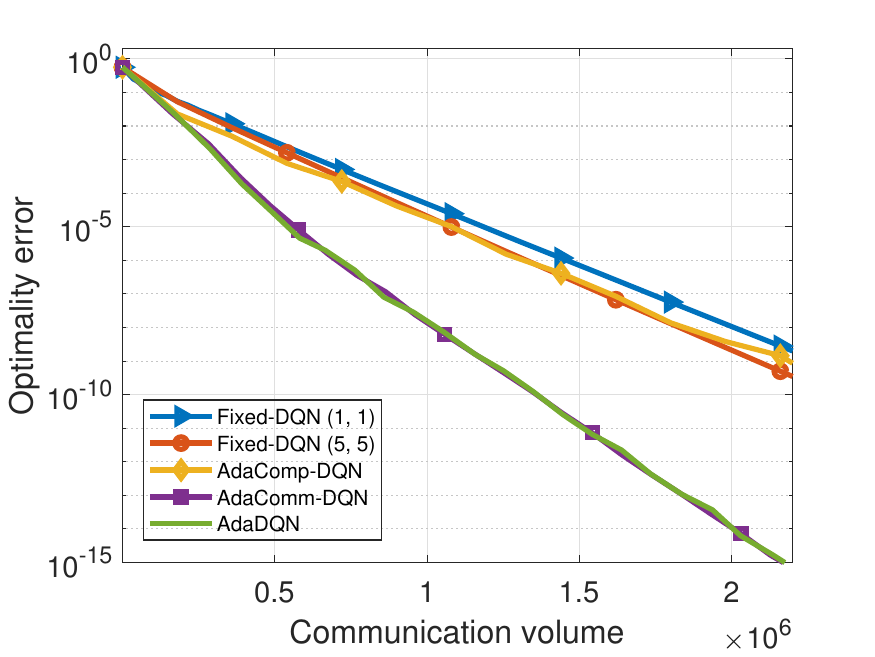}%
		\label{w_c}}
	\subfloat[a9a]{\includegraphics[width=1.3in]{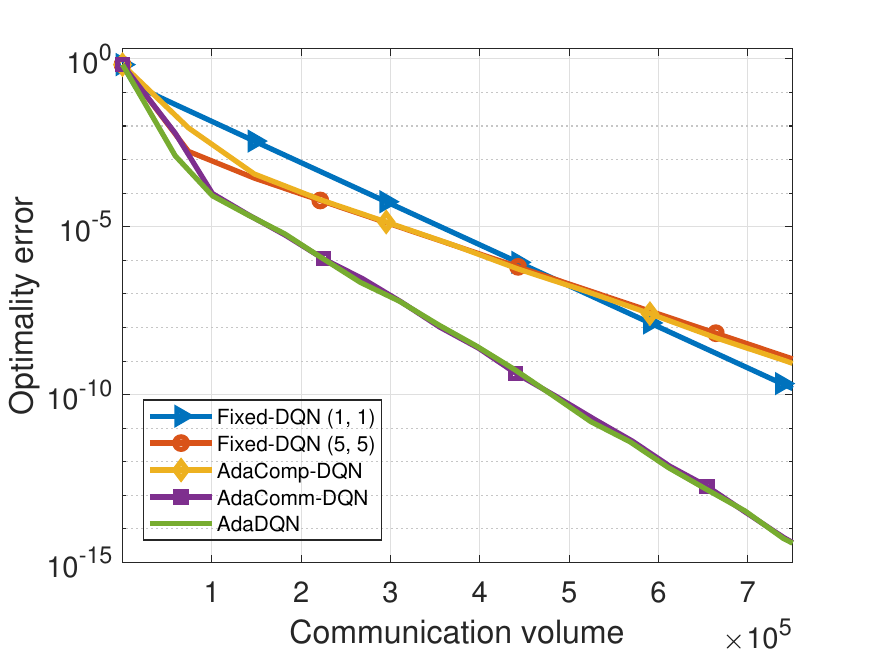}%
		\label{a_c}}
	\caption{Optimality error of comparison algorithms for minimizing the nonconvex logistic regression problem \eqref{noncovex_logistic_problem} on different datasets w.r.t. communication volume.}
	\label{adap_com}
\end{figure*}

\begin{figure*}[!t]
	\centering
	\subfloat[mushroom]{\includegraphics[width=1.3in]{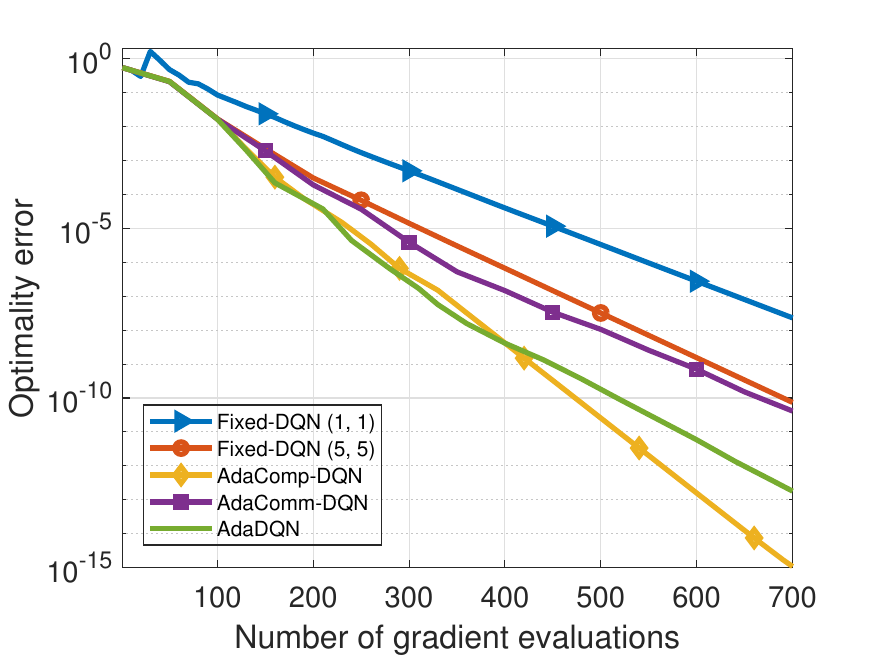}%
		\label{m_g}}
	\subfloat[ijcnn1]{\includegraphics[width=1.3in]{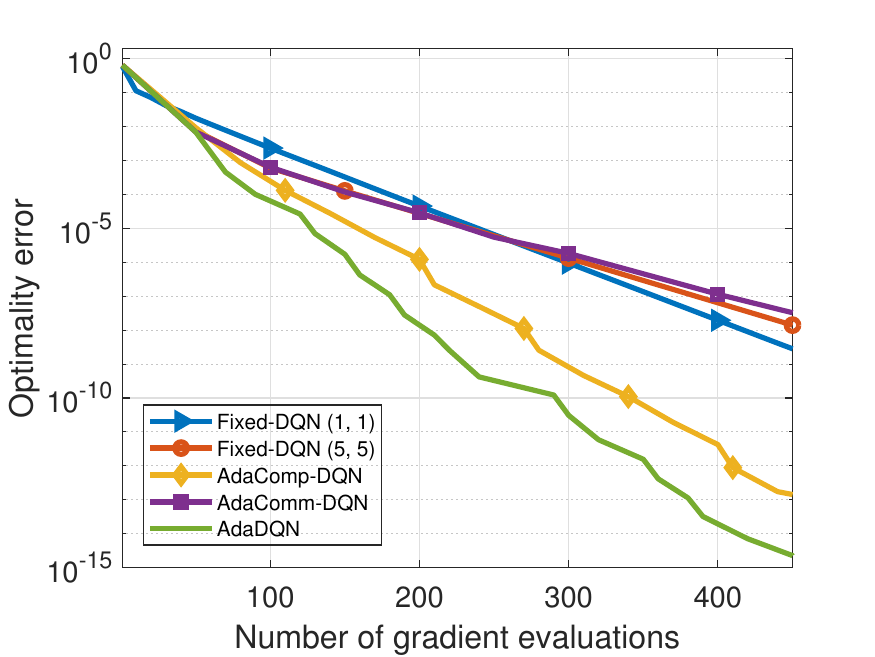}%
		\label{i_g}}
	\subfloat[w8a]{\includegraphics[width=1.3in]{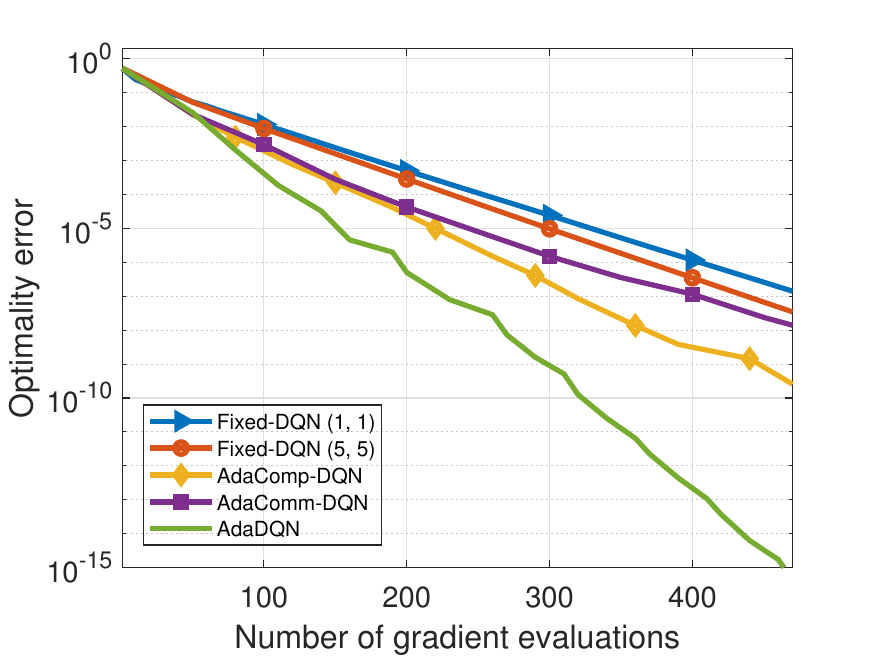}%
		\label{w_g}}
	\subfloat[a9a]{\includegraphics[width=1.3in]{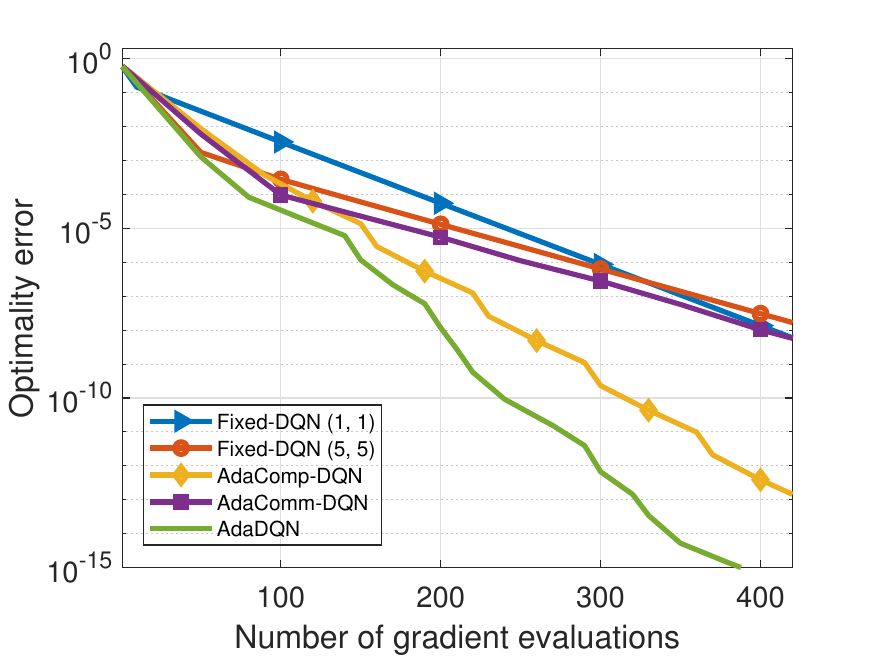}%
		\label{a_g}}
	\caption{Optimality error of comparison algorithms for minimizing the nonconvex logistic regression problem \eqref{noncovex_logistic_problem} on different datasets w.r.t. gradient evaluation number.}
	\label{adap_grad}
\end{figure*}

\begin{figure*}[!t]
	\centering
	\subfloat[mushroom]{\includegraphics[width=1.3in]{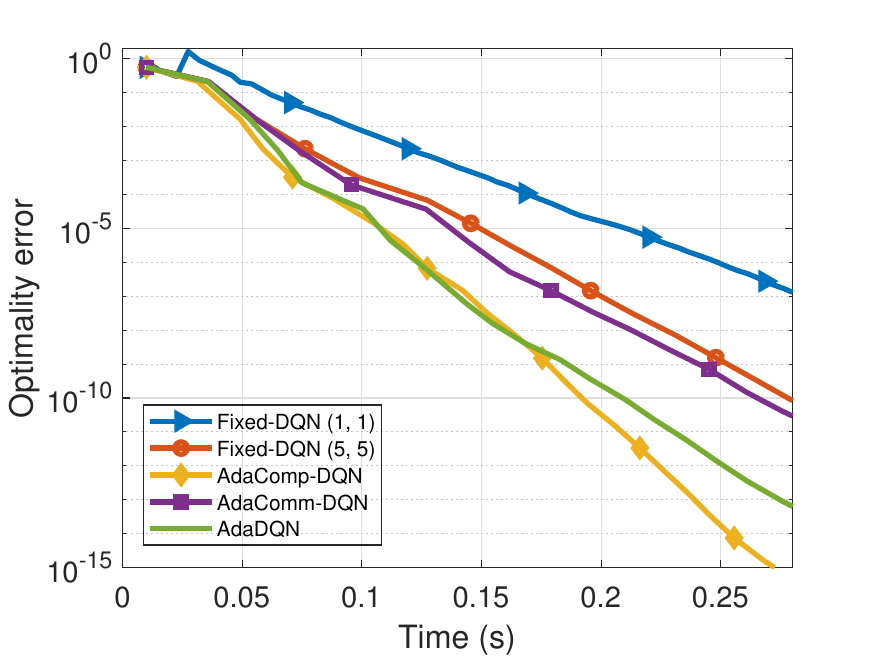}%
		\label{m_t}}
	\subfloat[ijcnn1]{\includegraphics[width=1.3in]{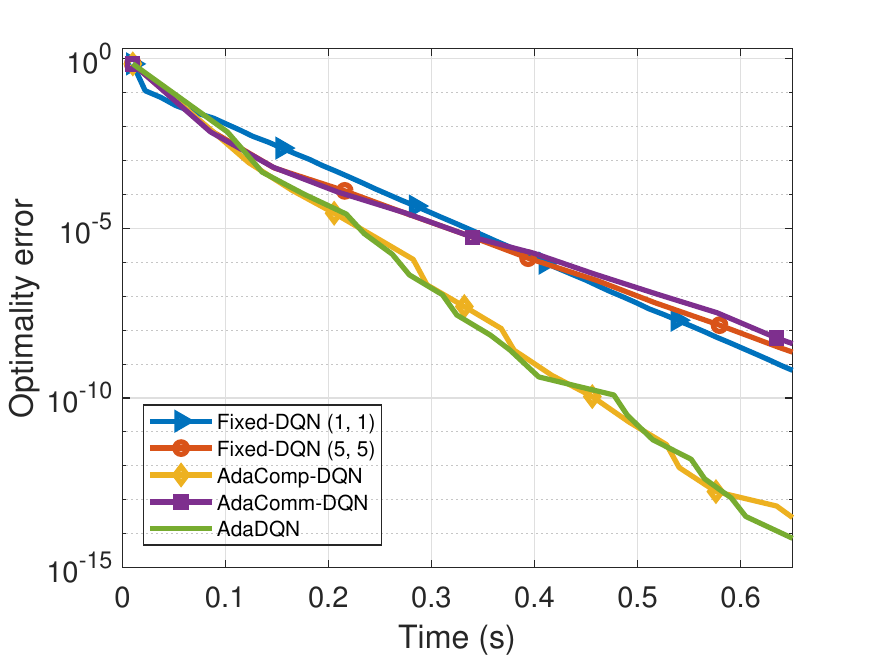}%
		\label{i_t}}
	\subfloat[w8a]{\includegraphics[width=1.3in]{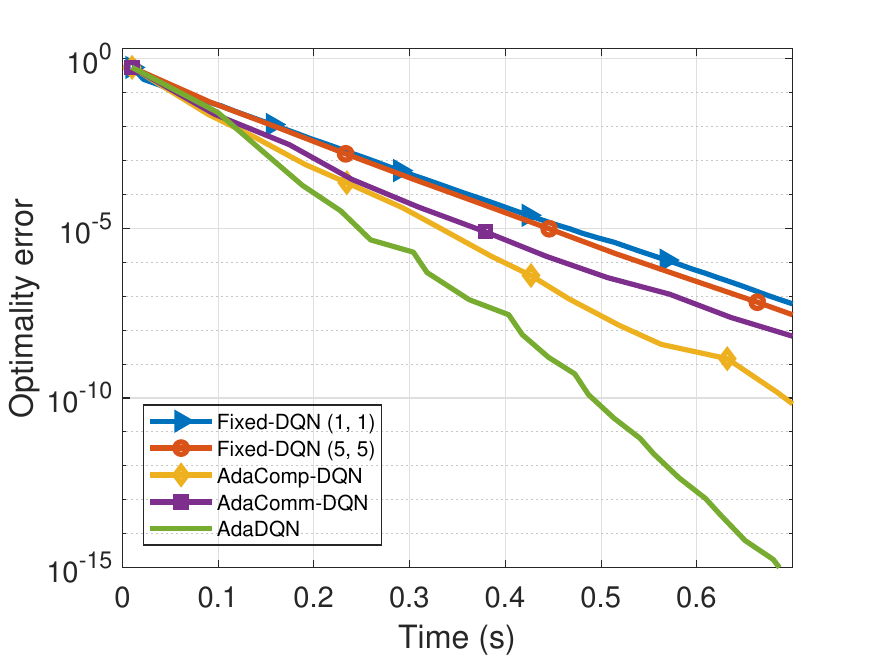}%
		\label{w_t}}
	\subfloat[a9a]{\includegraphics[width=1.3in]{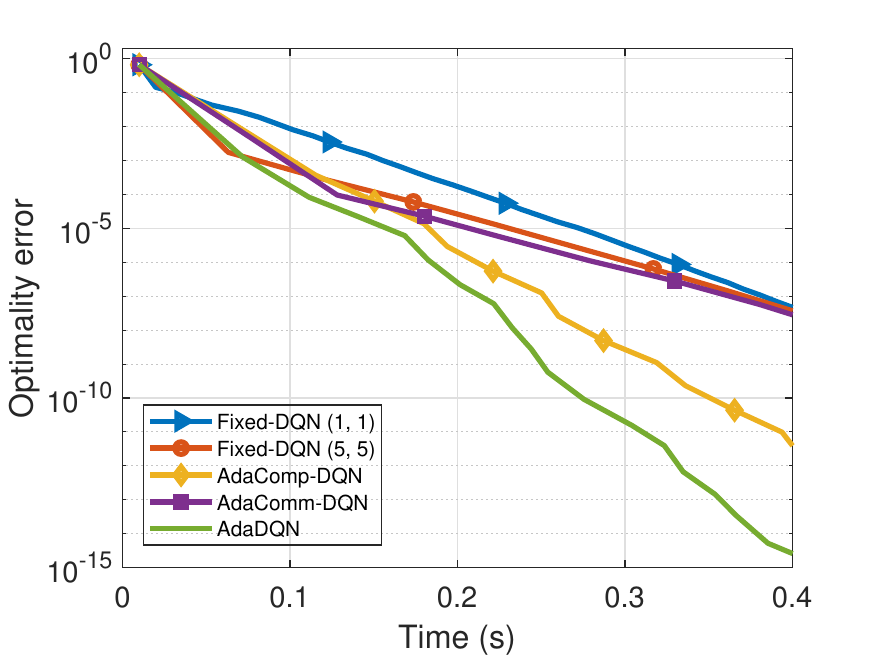}%
		\label{a_t}}
	\caption{Optimality error of comparison algorithms for minimizing the nonconvex logistic regression problem \eqref{noncovex_logistic_problem} on different datasets w.r.t. CPU time.}
	\label{adap_time}
\end{figure*}

\begin{figure*}[!t]
	\centering
	\subfloat[mushroom]{\includegraphics[width=1.3in]{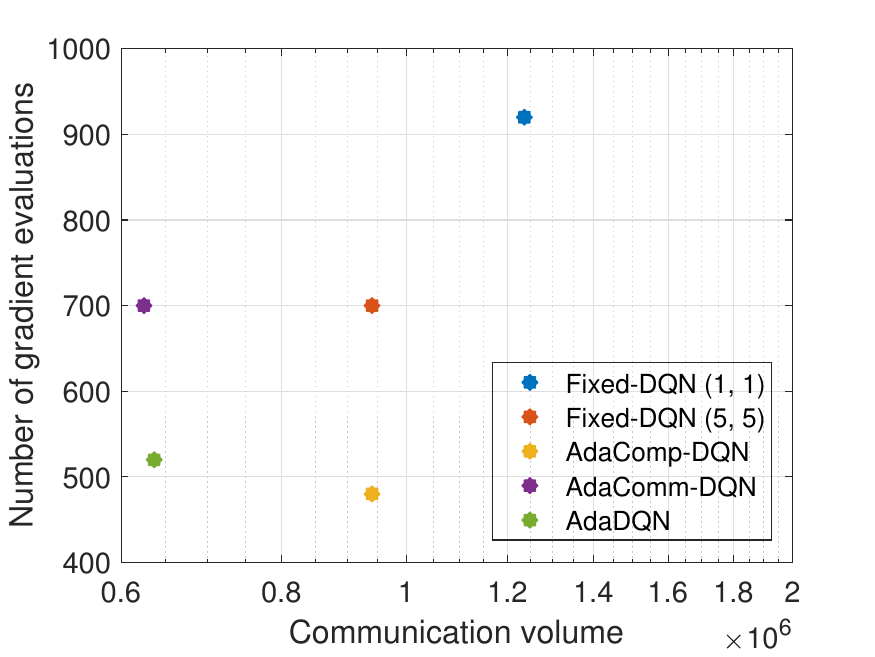}%
		\label{m_s}}
	\subfloat[ijcnn1]{\includegraphics[width=1.3in]{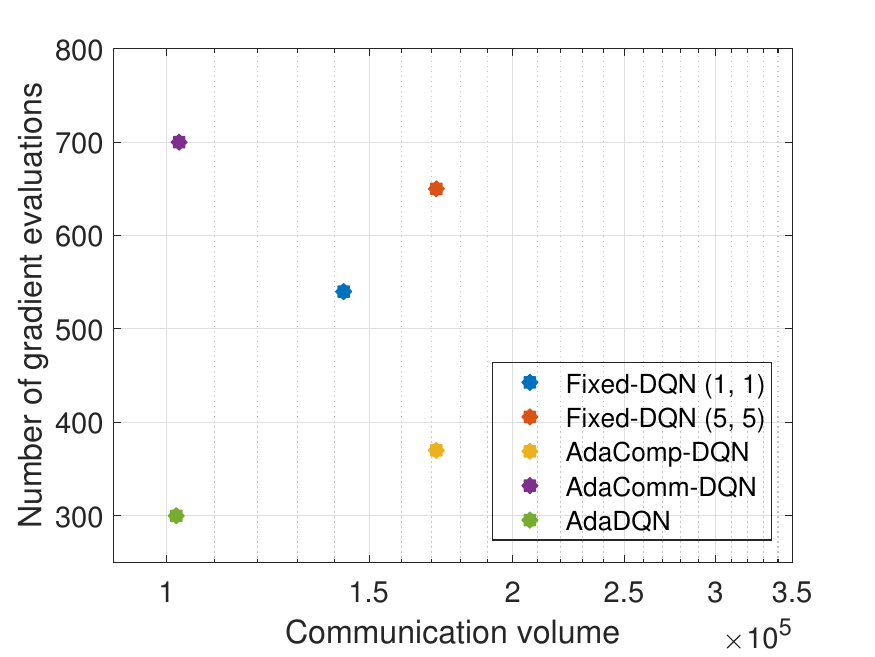}%
		\label{i_s}}
	\subfloat[w8a]{\includegraphics[width=1.3in]{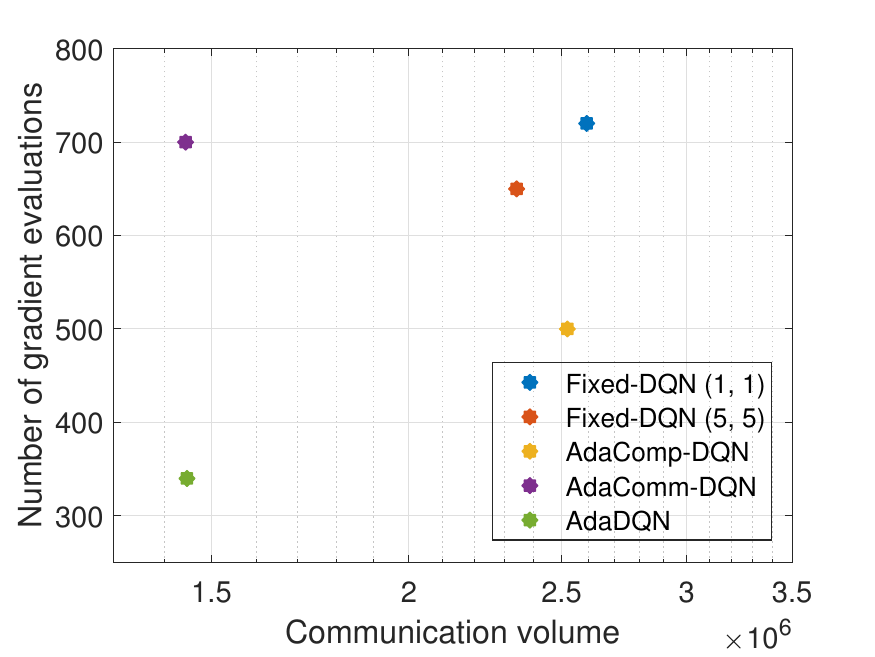}%
		\label{w_s}}
	\subfloat[a9a]{\includegraphics[width=1.3in]{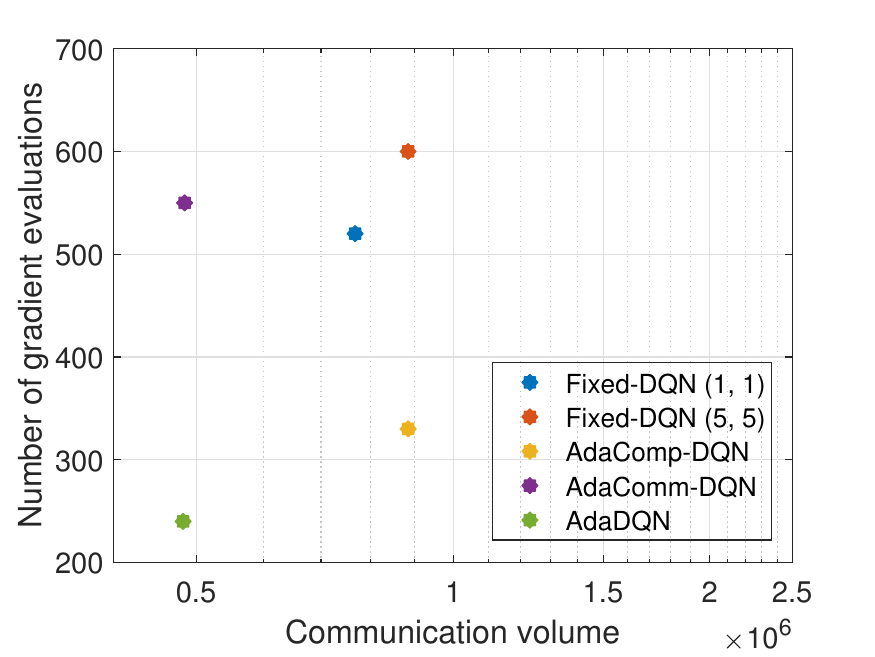}%
		\label{a_s}}
	\caption{Balance between gradient evaluation number and communication volume on different datasets. The coordinates of each point represent the
	gradient evaluation number (vertical axis) and communication volume (horizontal axis) required to achieve an accuracy of $\epsilon=10^{-10}$.}
	\label{adap_scatter}
\end{figure*}

\subsection{Comparison with existing relevant algorithms}
In this part of our experiment, we compare our AdaDQN with several most relevant algorithms, including ET-DOGT\cite{liu2023event}, DETNO \cite{gao2024distributed}, LSGT \cite{ge2023gradient}, and GTA \cite{berahas2024balancing}.

We first consider the following linear regression problem
\begin{equation}\label{linear_problem}
	\mathop {\min }\limits_{\m{z} \in {\R^p}} \sum_{i=1}^n \frac{1}{2} \m{z}\tr \m{A}_i \m{z} + \m{b}_i\tr \m{z} +\hat{\lambda} \sum_{k=1}^p \frac{\m{z}_{[k]}^2}{1+\m{z}_{[k]}^2},
\end{equation}
where $\m{A}_i \in \R^{p \times p}$ and $\m{b}_i \in \R^p$ are private data available to node $i$. We construct $\m{A}_i = \m{Q}\tr \operatorname {diag} \{ a_1,...,a_p \} \m{Q}$, where $\m{Q}$ is a random orthogonal matrix. We set $a_1=1$ and $a_p$ as an arbitrarily large number, 
and generate $a_j \sim \operatorname{U}(1,2)$ for $j=2, \ldots, p-1$, where $\operatorname{U}(1,2)$ represents the uniform distribution from 1 to 2. We consider two cases with $a_p=10$ and $a_p=1000$, where $p$ is set to 500 in both scenarios. For the problem with $a_p=10(1000)$,
we set 
\begin{itemize}
		\item $\alpha=0.062(0.022)$ and $\eta=1.4(1.9)$ in ET-DOGT;
	\item $\alpha=0.04(0.0062)$, $\beta=0.1(0.09)$, $\gamma=0.9(1)$, and $S=0.9(0.98)$ in DETNO;
	\item $E=5$ and $\gamma=0.011(0.0044)$ in LSGT;
	\item $n_g=5$, $n_c=5$, $\alpha=0.15(0.024)$, and $\m{Z}_l$ for $l=1,2,3,4$ in GTA ;
	\item $K_g=5$, $K_c=5$, $\alpha=0.29(0.14)$, $\varepsilon=1$, $\kappa=1$, $\rho=0.26(0.5)$, and $\varrho=0.05$ in AdaDQN.
\end{itemize}
In the well-conditioned case ($a_p=10$, Fig. \ref{cond10}), GTA and AdaDQN perform comparably. However, in the ill-conditioned case ($a_p=1000$, Fig. \ref{cond1000}), AdaDQN outperforms all other methods (ET-DOGT, DETNO, LSGT, GTA) by a large margin. This verifies that incorporating quasi-Newton approximations captures essential curvature information, enabling the algorithm to be more robust and efficient to badly scaled problems than gradient-based methods.

\begin{figure*}[!t]
	\centering
	\subfloat[]{\includegraphics[width=1.3in]{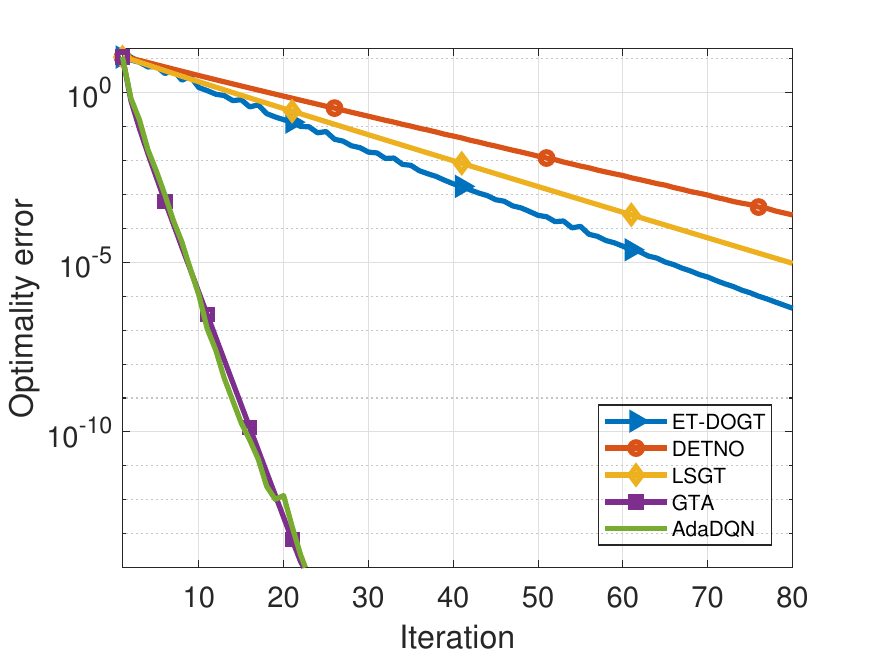}%
		\label{cond10_1}}
	\subfloat[]{\includegraphics[width=1.3in]{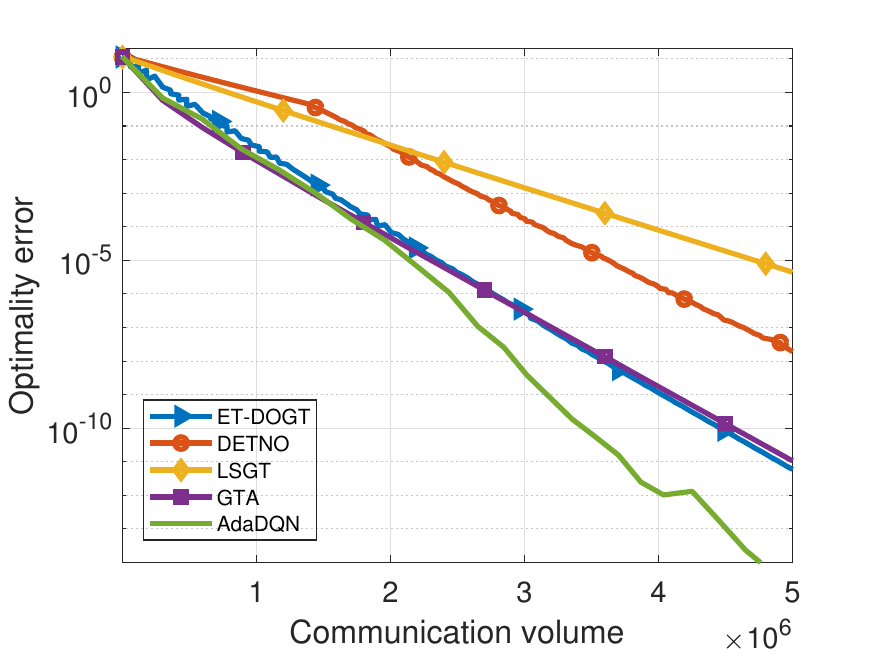}%
		\label{cond10_2}}
	\subfloat[]{\includegraphics[width=1.3in]{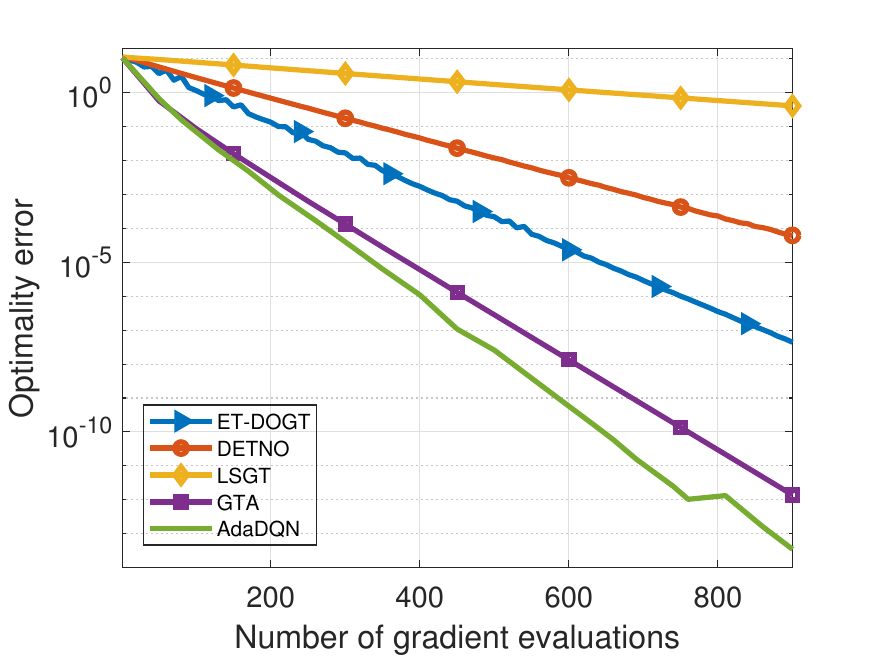}%
		\label{cond10_3}}
	\subfloat[]{\includegraphics[width=1.3in]{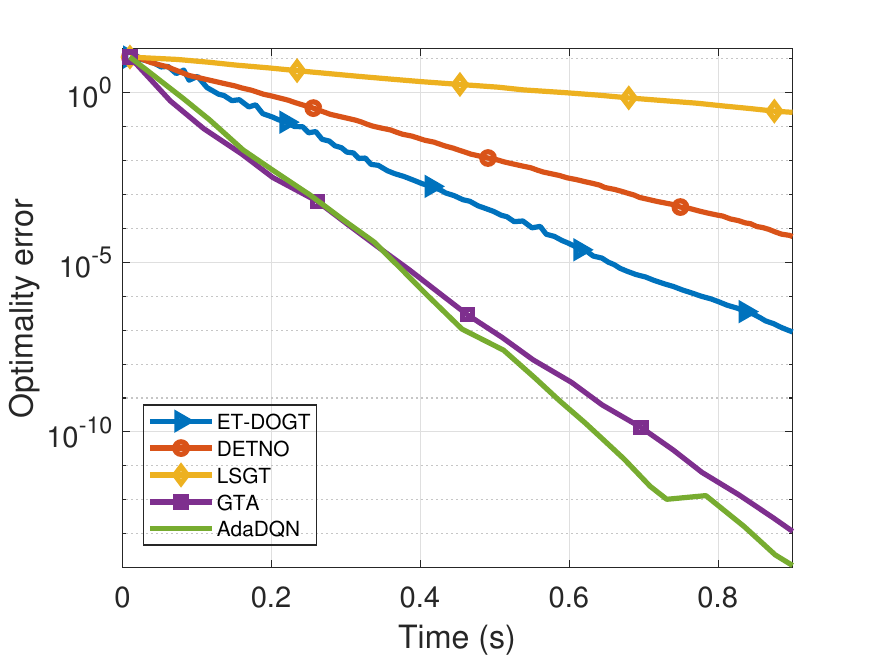}%
		\label{cond10_4}}
	\caption{Optimality error of comparison algorithms for minimizing the nonconvex linear regression problem \eqref{linear_problem} with $a_p=10$.}
	\label{cond10}
\end{figure*}

\begin{figure*}[!t]
	\centering
	\subfloat[]{\includegraphics[width=1.3in]{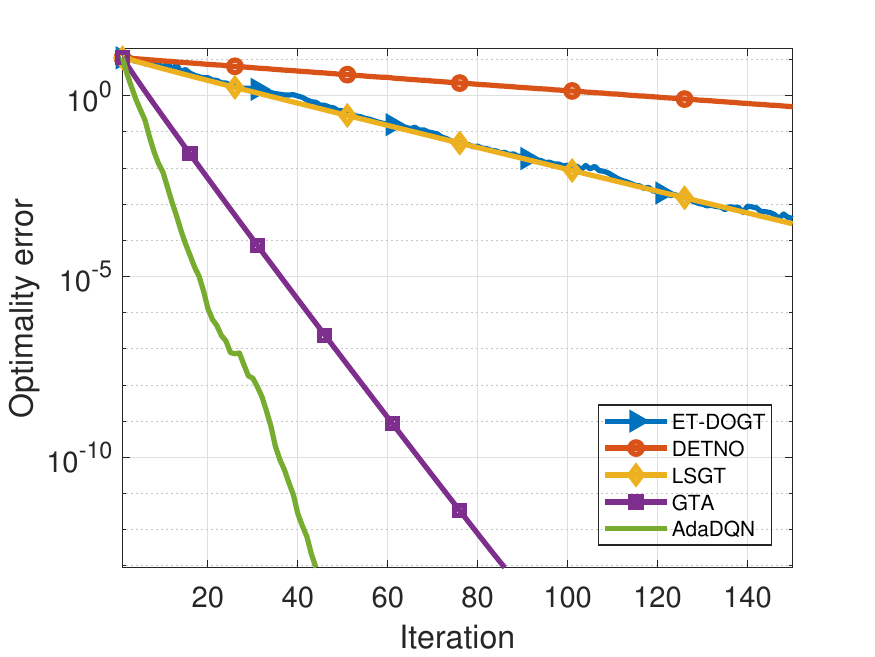}%
		\label{cond1000_1}}
	\subfloat[]{\includegraphics[width=1.3in]{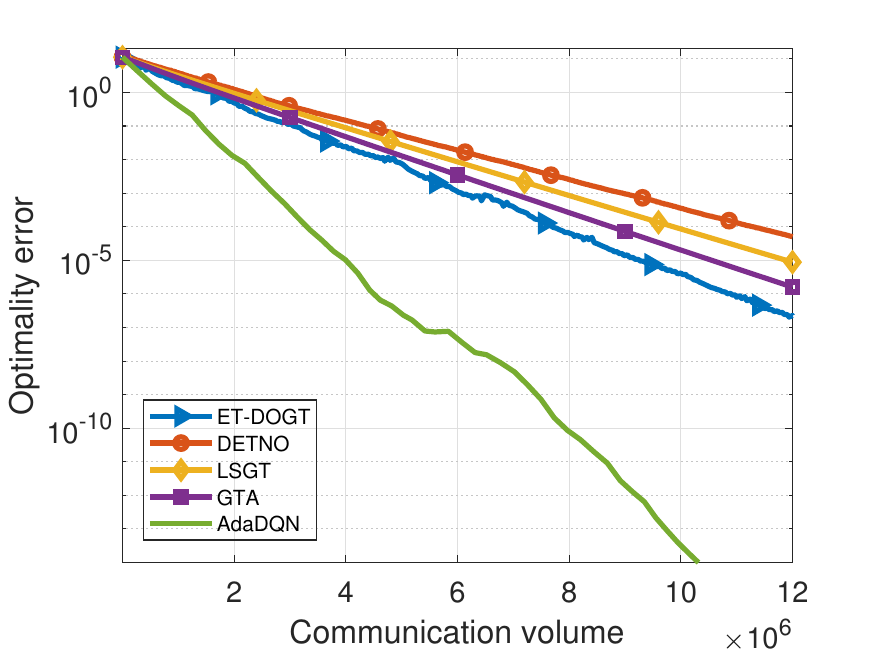}%
		\label{cond1000_2}}
	\subfloat[]{\includegraphics[width=1.3in]{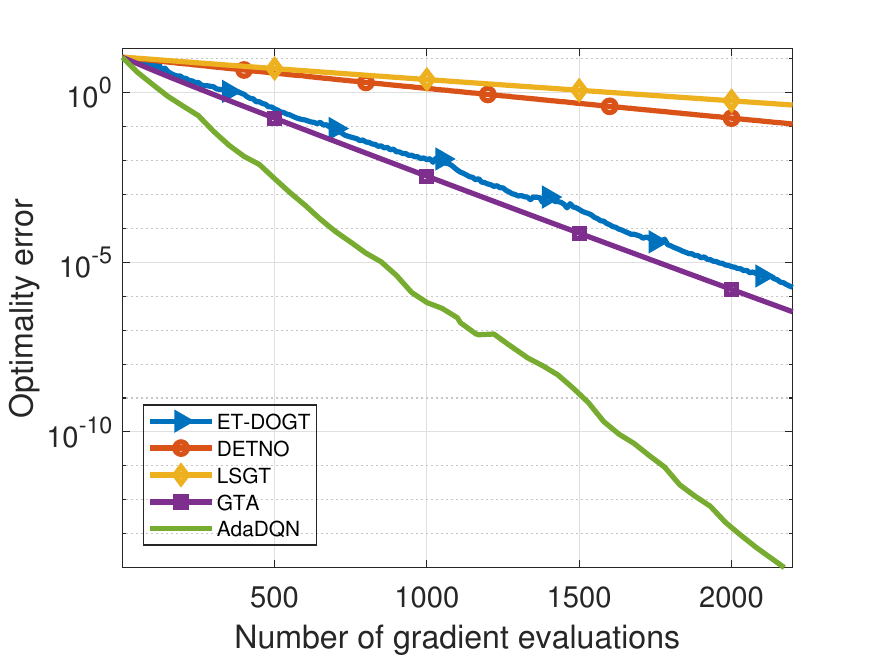}%
		\label{cond1000_3}}
	\subfloat[]{\includegraphics[width=1.3in]{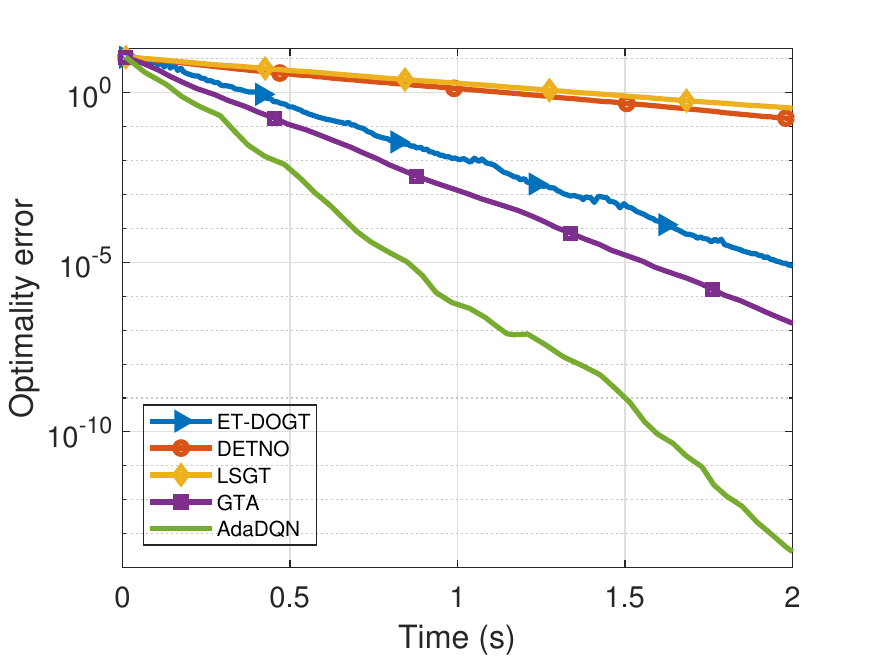}%
		\label{cond1000_4}}
	\caption{Optimality error of comparison algorithms for minimizing the nonconvex linear regression problem \eqref{linear_problem} with $a_p=1000$.}
	\label{cond1000}
\end{figure*}

Then, we consider the nonconvex logistic regression problem \eqref{noncovex_logistic_problem}, using the four datasets given by Table \ref{table2}.
For datasets \textbf{mushroom}(\textbf{ijcnn1};\textbf{w8a};\textbf{a9a}), we set 
\begin{itemize}
	\item $\alpha=0.09(0.1;0.12;0.09)$ and $\eta=1.3(1.6;1.4;1.5)$ in ET-DOGT;
	\item $\alpha=0.06(0.1;0.08;0.08)$, $\beta=0.11(0.1;0.11;0.1)$, $\gamma=0.9$, and $S$=0.05 (0.001; 0.003; 0.01) in DETNO;
	\item $E=5$ and $\gamma=0.015(0.02;0.019;0.018)$ in LSGT;
	\item $n_g=5$, $n_c=5$, and $\alpha=0.15(0.18;0.14;0.15)$ in GTA.
\end{itemize}
As shown in Figs \ref{ET_iter}-\ref{ET_time}, compared to alternative methods on all selected datasets, AdaDQN consistently exhibits superior performance in terms of iterations, communication volume, gradient evaluations, and CPU time.

\begin{figure*}[!t]
	\centering
	\subfloat[mushroom]{\includegraphics[width=1.3in]{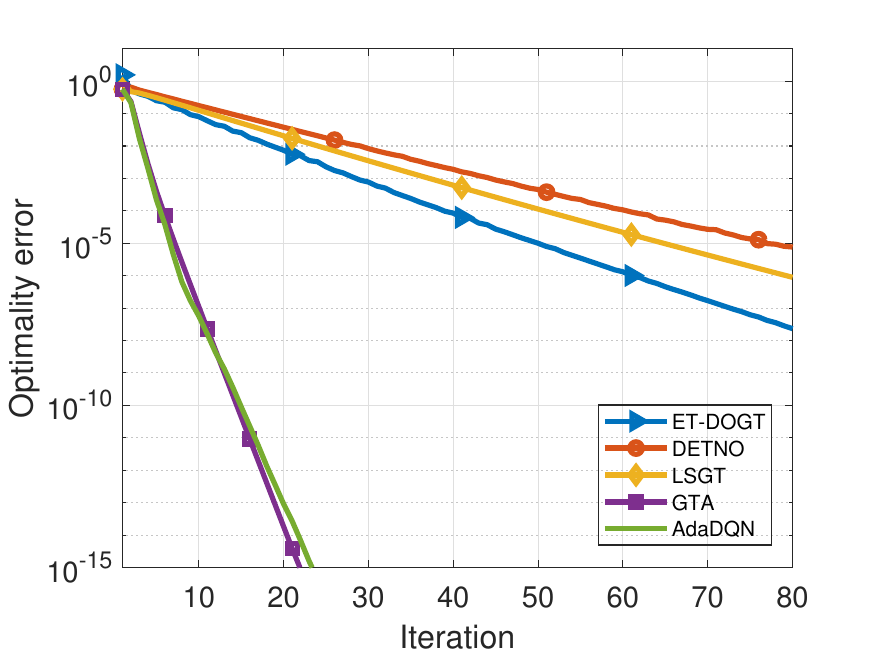}%
		\label{ET_mush_i}}
	\subfloat[ijcnn1]{\includegraphics[width=1.3in]{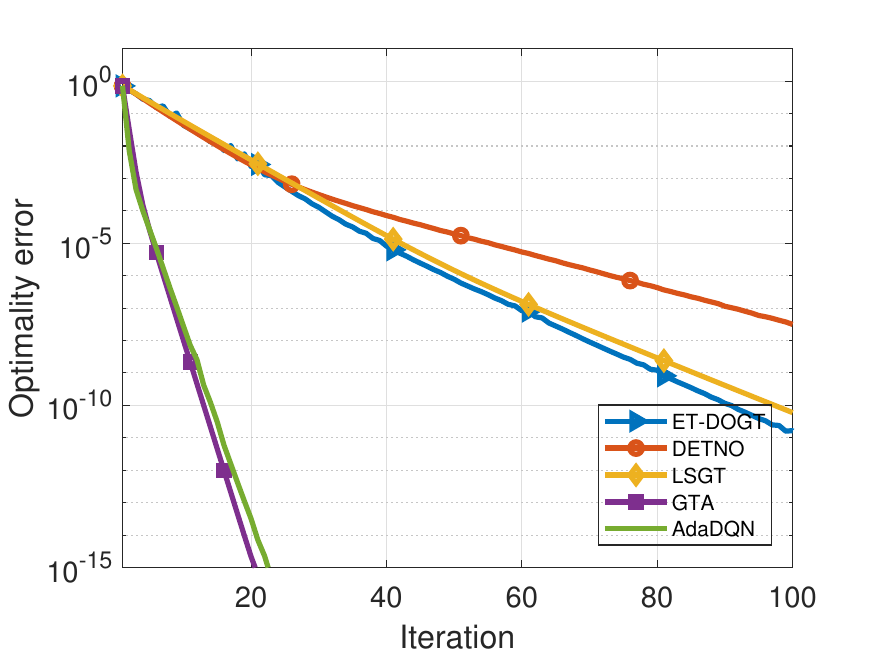}%
		\label{ET_ijcnn_i}}
	\subfloat[w8a]{\includegraphics[width=1.3in]{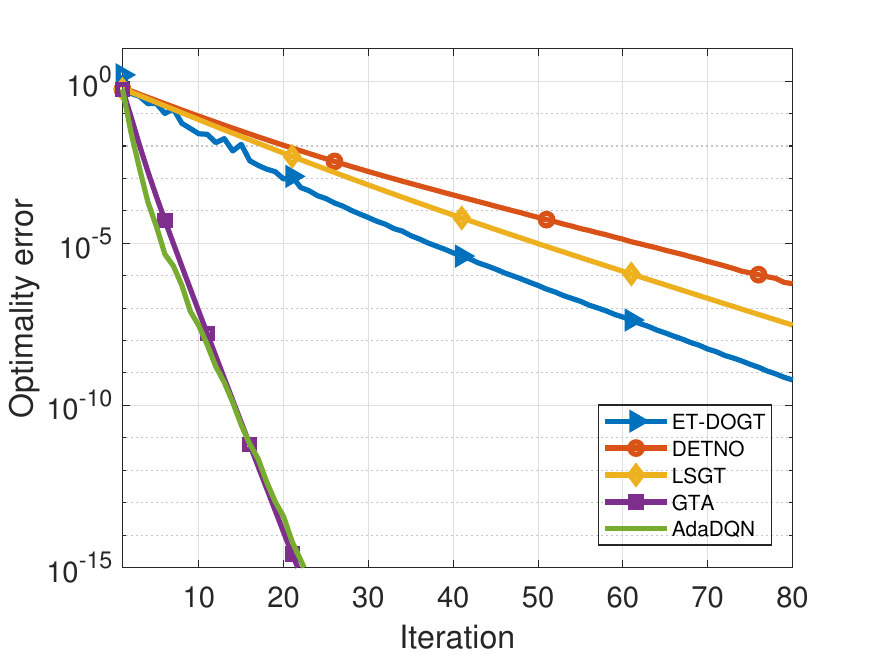}%
		\label{ET_w8a_i}}
	\subfloat[a9a]{\includegraphics[width=1.3in]{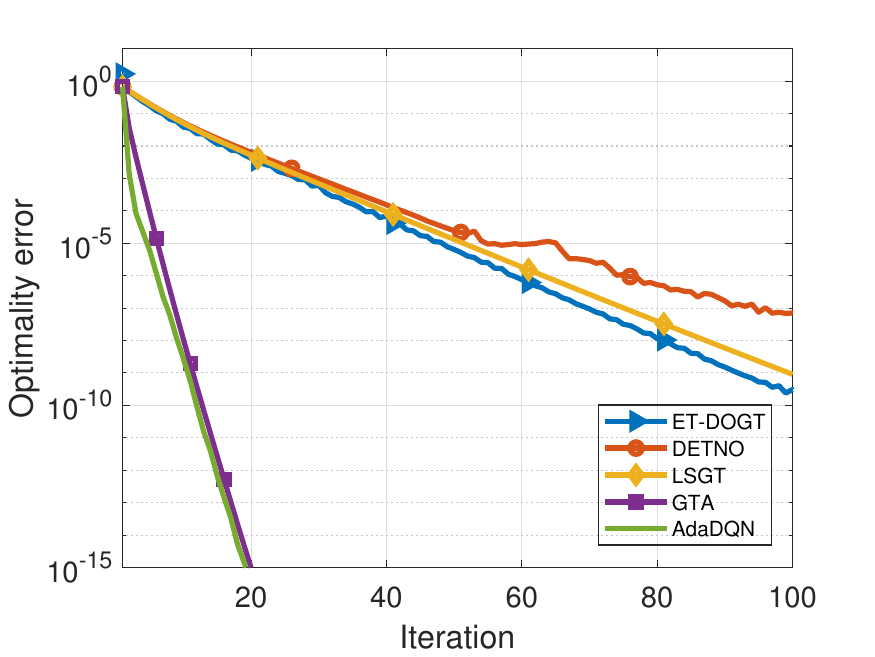}%
		\label{ET_a9a_i}}
	\caption{Optimality error of comparison algorithms for minimizing the nonconvex logistic regression problem \eqref{noncovex_logistic_problem} on different datasets w.r.t. iteration number.}
	\label{ET_iter}
\end{figure*}

\begin{figure*}[!t]
	\centering
	\subfloat[mushroom]{\includegraphics[width=1.3in]{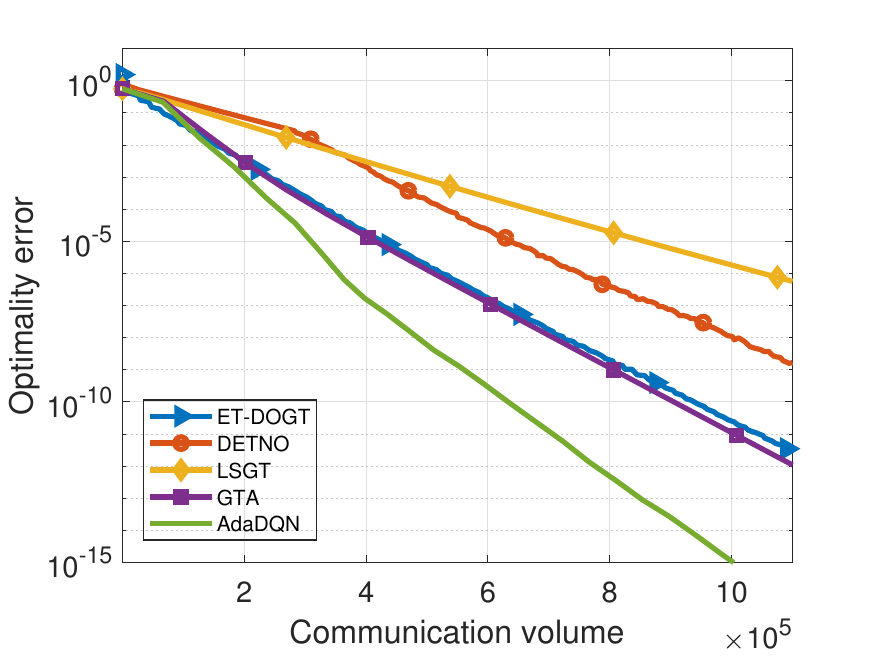}%
		\label{ET_mush_c}}
	\subfloat[ijcnn1]{\includegraphics[width=1.3in]{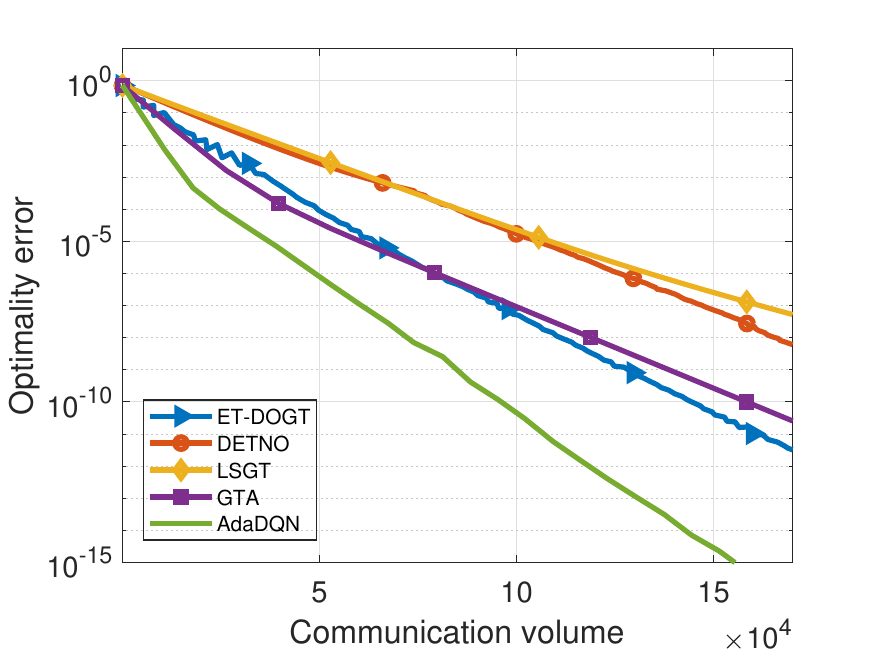}%
		\label{ET_ijcnn_c}}
	\subfloat[w8a]{\includegraphics[width=1.3in]{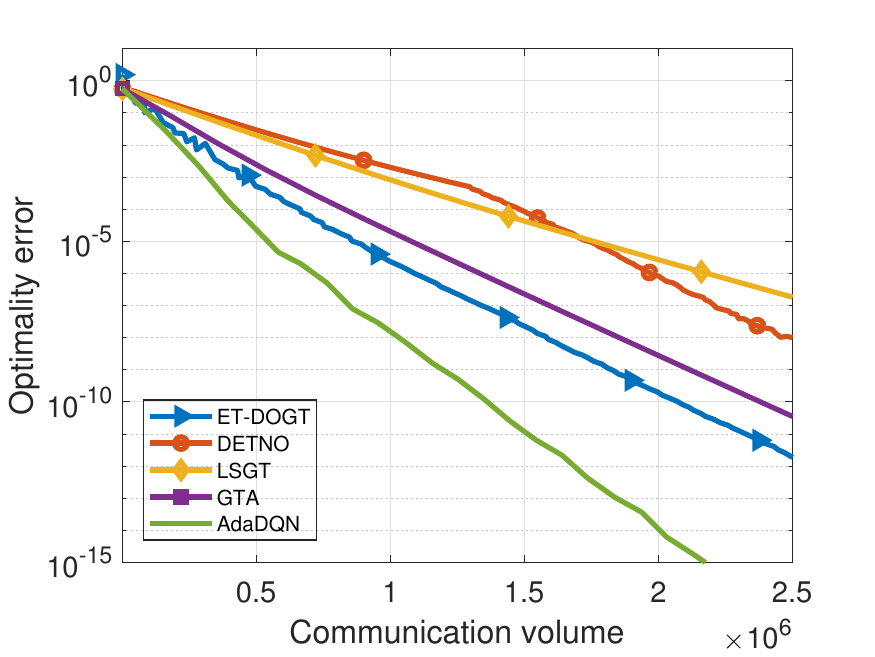}%
		\label{ET_w8a_c}}
	\subfloat[a9a]{\includegraphics[width=1.3in]{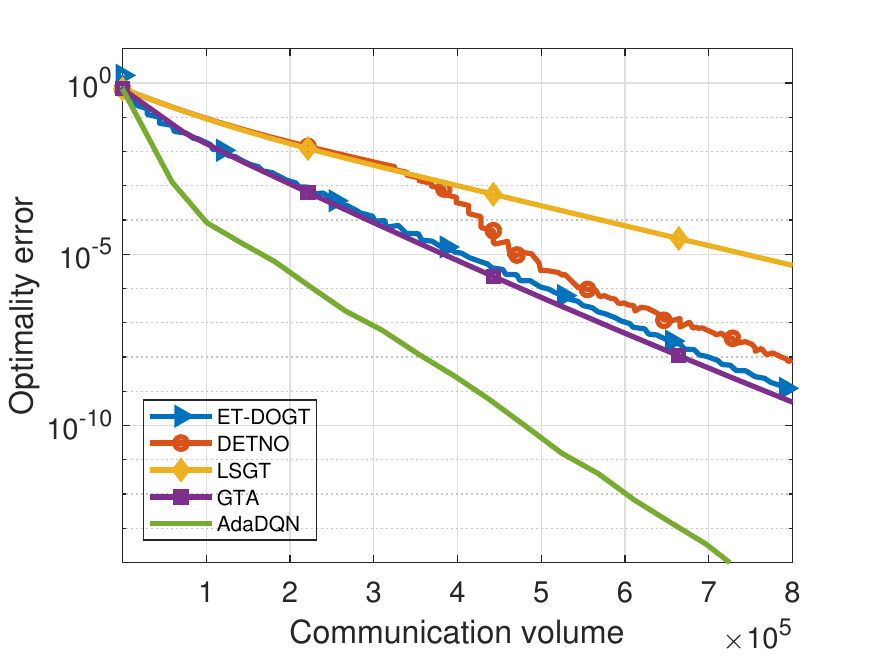}%
		\label{ET_a9a_c}}
	\caption{Optimality error of comparison algorithms for minimizing the nonconvex logistic regression problem \eqref{noncovex_logistic_problem} on different datasets w.r.t. communication volume.}
	\label{ET_com}
\end{figure*}

\begin{figure*}[!t]
	\centering
	\subfloat[mushroom]{\includegraphics[width=1.3in]{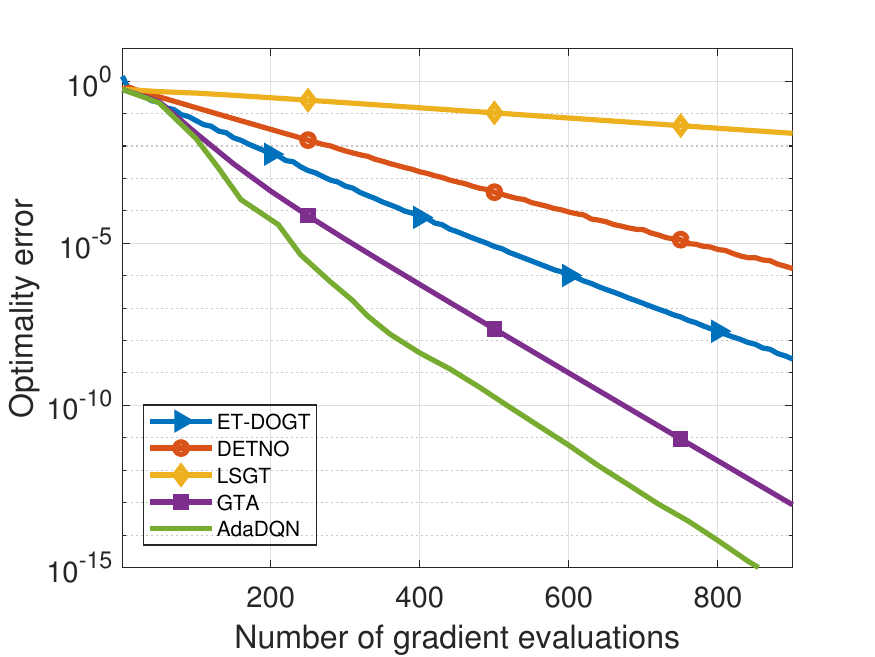}%
		\label{ET_mush_g}}
	\subfloat[ijcnn1]{\includegraphics[width=1.3in]{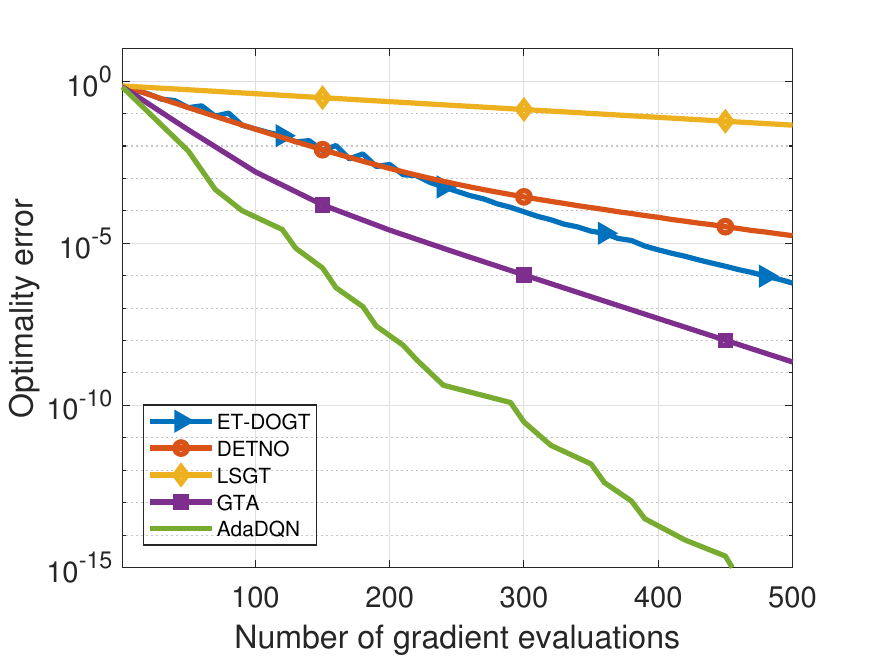}%
		\label{ET_ijcnn_g}}
	\subfloat[w8a]{\includegraphics[width=1.3in]{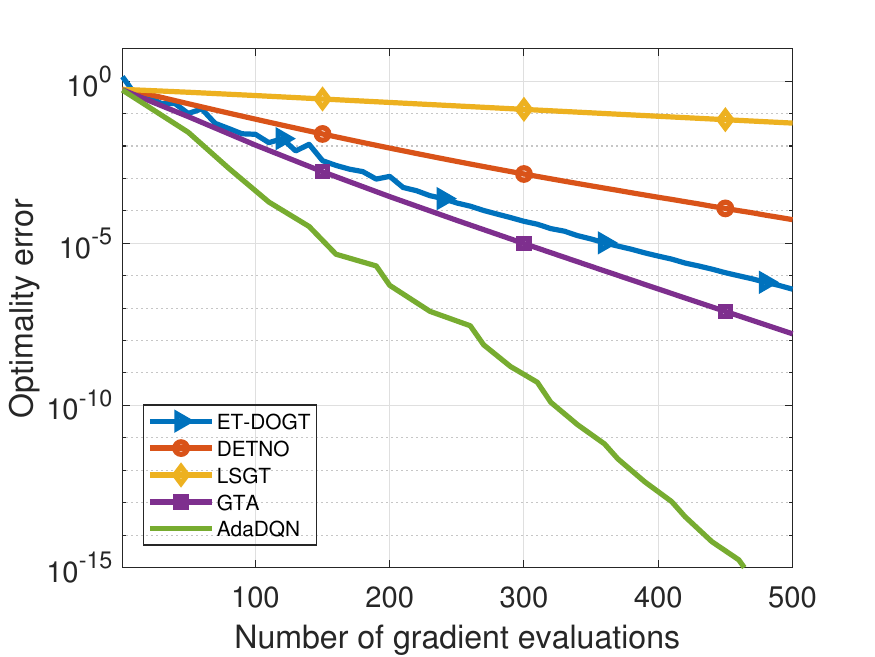}%
		\label{ET_w8a_g}}
	\subfloat[a9a]{\includegraphics[width=1.3in]{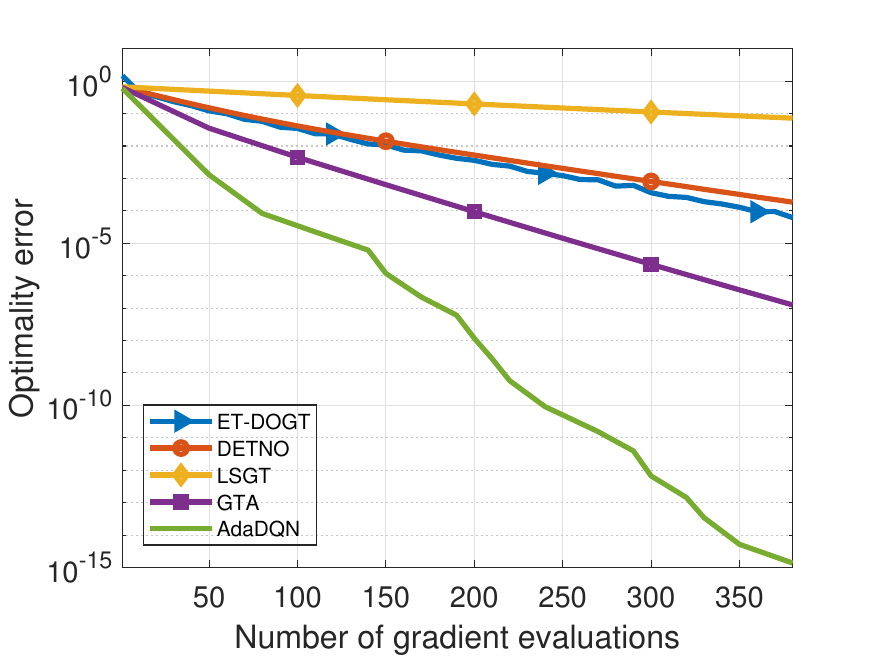}%
		\label{ET_a9a_g}}
	\caption{Optimality error of comparison algorithms for minimizing the nonconvex logistic regression problem \eqref{noncovex_logistic_problem} on different datasets w.r.t. gradient evaluation number.}
	\label{ET_grad}
\end{figure*}

\begin{figure*}[!t]
	\centering
	\subfloat[mushroom]{\includegraphics[width=1.3in]{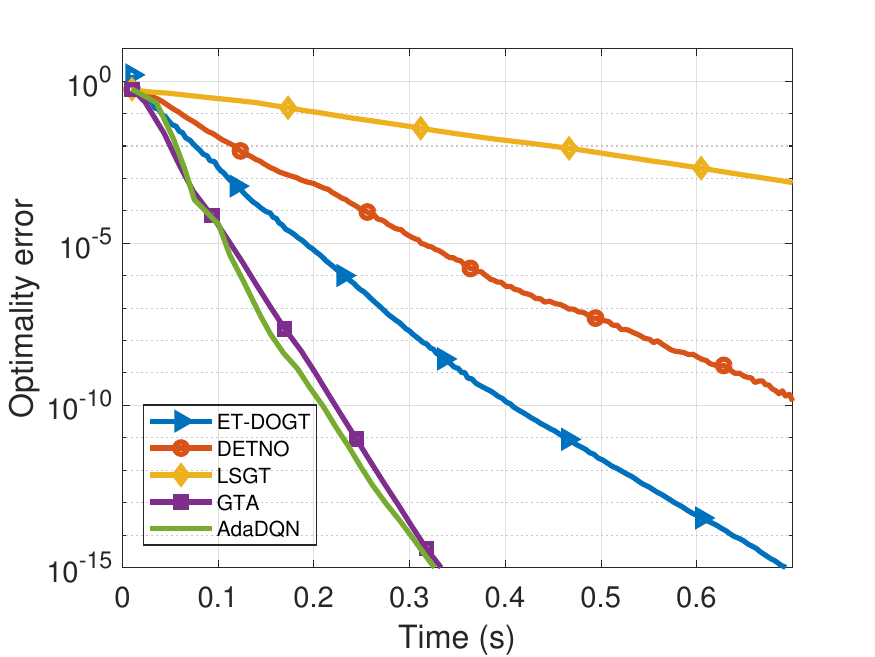}%
		\label{ET_mush_t}}
	\subfloat[ijcnn1]{\includegraphics[width=1.3in]{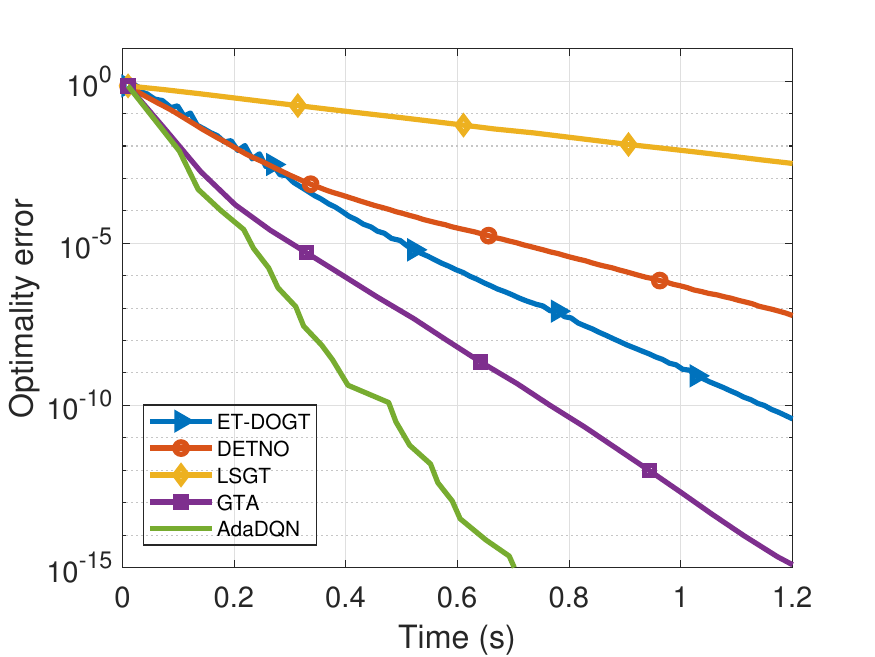}%
		\label{ET_ijcnn_t}}
	\subfloat[w8a]{\includegraphics[width=1.3in]{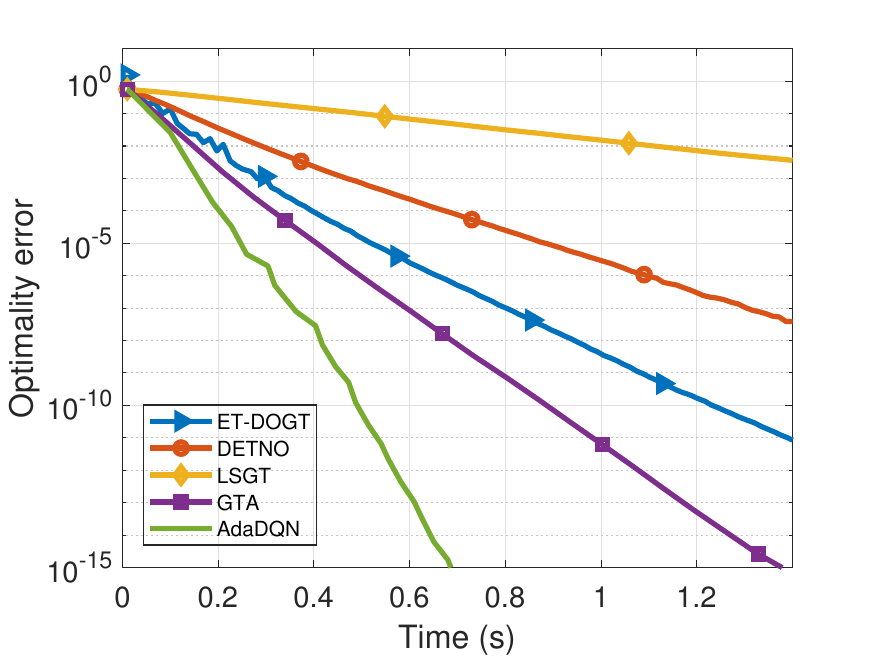}%
		\label{ET_w8a_t}}
	\subfloat[a9a]{\includegraphics[width=1.3in]{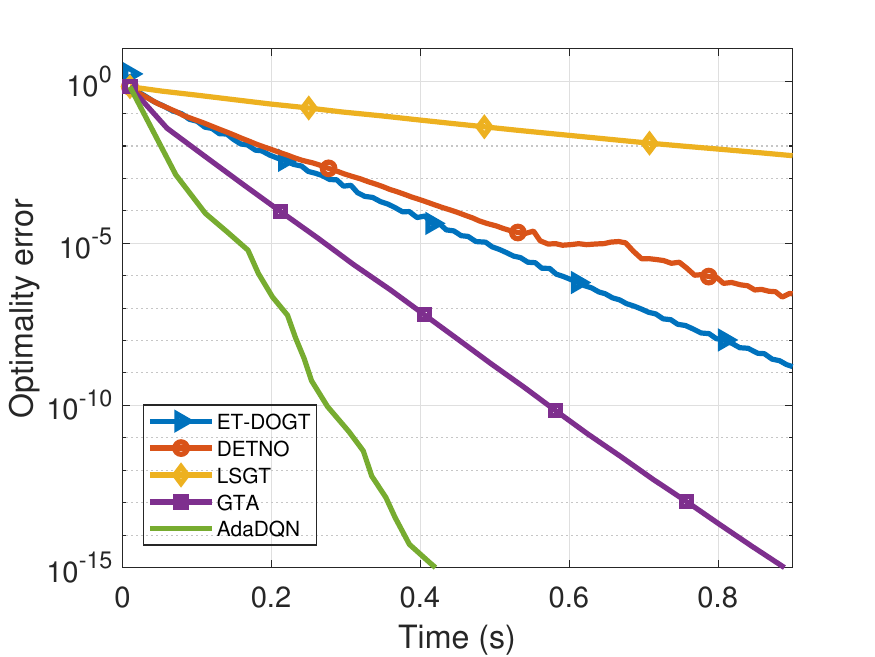}%
		\label{ET_a9a_t}}
	\caption{Optimality error of comparison algorithms for minimizing the nonconvex logistic regression problem \eqref{noncovex_logistic_problem} on different datasets w.r.t. CPU time.}
	\label{ET_time}
\end{figure*}

\clearpage
We emphasize that the proposed adaptive rules are driven by algorithmic states and do not explicitly model system-level heterogeneity in processor speed, bandwidth, communication delay, or node energy consumption. Accordingly, the experiments evaluate reductions in algorithmic computation and communication costs rather than performance under heterogeneous hardware or network conditions. Extending the method and its evaluation to such settings is left for future work.

\section{Conclusions}

In this paper, we investigate the computation-communication tradeoff in decentralized nonconvex optimization. We first introduce the Robust Inexact Algorithm (RIA) framework from a surrogate-based perspective, which provides a common representation for the fixed-local-update gradient-tracking schemes verified in Section 3 and shows that the $\mathcal{O}(1/K_g)$ sufficient stepsize scaling is worst-case tight for the unscaled fixed-local-update scheme. To circumvent this limitation, we propose the Adaptive Decentralized Quasi-Newton (AdaDQN) method. By integrating a safeguarded consensus-aware termination criterion, an event-triggered communication protocol, and a scalable memoryless BFGS update, AdaDQN achieves double-level adaptivity through adaptive local updates and broadcasts.

Assuming the objective function is bounded from below and has Lipschitz gradients, we establish that the first-order stationarity residual of AdaDQN converges to zero and the best-iterate stationarity measure admits an $\mathcal{O}(1/T)$ rate. We also prove that the theoretically required stepsize is not inversely proportional to the maximum number of local updates. The resulting gradient-complexity bounds are $\mathcal{O}(nK_g/\delta)$ for fixed local updates and $\mathcal{O}(n/\delta+n\alpha\tilde\varepsilon^{-2})$ for AdaDQN. Ablation experiments and comparison experiments
on nonconvex machine learning tasks with various datasets demonstrate that AdaDQN achieves a superior computation-communication tradeoff and significantly outperforms state-of-the-art decentralized optimization methods, including ET-DOGT\cite{liu2023event}, DETNO \cite{gao2024distributed}, LSGT \cite{ge2023gradient}, and GTA \cite{berahas2024balancing}.

\section*{Acknowledgments}
We thank Prof. Hongchao Zhang from Louisiana State University for insightful comments and constructive suggestions during the preparation of this work.


\bibliographystyle{IEEEtran}
\bibliography{refs}

\appendix

\section{Analytical tools}

Two analytical tools are provided below.

\begin{lemma}\label{young}
	(Young's inequality) For any two vectors $\m{v}_1,\m{v}_2\in \mathbb{R}^p$, and $\eta>0$,
	\begin{align*}
		2\m{v}_1\tr\m{v}_2 \leq& \eta \Vert\m{v}_1\Vert^2+ \frac{1}{\eta}\Vert\m{v}_2\Vert^2,\\
		\Vert\m{v}_1+\m{v}_2\Vert^2 \leq& (1+\eta) \Vert\m{v}_1\Vert^2+ \left(1+\frac{1}{\eta}\right)\Vert\m{v}_2\Vert^2.
	\end{align*}
\end{lemma}
\begin{lemma}\label{important}
	(Jensen's inequality) For a set of vectors $\{\m{v}_i\}_{i=1}^n \subset \mathbb{R}^p$,
	$$\left\Vert\frac{1}{n}\sum_{i=1}^n\m{v}_i\right\Vert^2\leq\frac{1}{n}\sum_{i=1}^n\Vert\m{v}_i\Vert^2.$$ 
\end{lemma}
\end{document}